\documentclass[10pt,reqno]{amsart}

\makeatletter
\@namedef{subjclassname@2020}{%
	\textup{2020} Mathematics Subject Classification}
\makeatother

\usepackage{amsthm,amsmath,amssymb,amscd}
\allowdisplaybreaks
\usepackage{mathrsfs}
\usepackage[colorlinks,citecolor=magenta,linkcolor=blue]{hyperref}
\usepackage{amssymb}
\usepackage{esint}
\usepackage{graphics}
\newtheorem{theorem}{Theorem}[section]

\newtheorem{assumption}[theorem]{Assumption}
\newtheorem{lemma}[theorem]{Lemma}
\newtheorem{proposition}[theorem]{Proposition}
\newtheorem{definition}{Definition}[section]
\theoremstyle{remark}
\newtheorem{remark}[theorem]{Remark}
\newtheorem*{ack}{Acknowledgments}

\numberwithin{equation}{section}
\begin{document}
\title[A Talenti-type comparison theorem for the $p$-Laplacian]{A Talenti-type comparison theorem for the $p$-Laplacian  on $\mathrm{RCD}(K,N)$ spaces and some applications}
\author[Wenjing Wu]{Wenjing Wu}%${}^*$}
\address{School of Mathematical Sciences, University of Science and Technology of China, Hefei 230026, P.R. China}
\email{\href{mailto:anprin@mail.ustc.edu.cn}{anprin@mail.ustc.edu.cn}}

\begin{abstract}
In this paper, we prove a Talenti-type comparison theorem for the $p$-Laplacian with Dirichlet boundary conditions on open subsets of a $\mathrm{RCD}(K,N)$ space with $K>0$ and $N\in (1,\infty)$.
The obtained Talenti-type comparison theorem is sharp, rigid and stable with respect to measured Gromov-Hausdorff topology. As an application of such Talenti-type comparison, we establish a sharp and rigid reverse H\"{o}lder inequality for first eigenfunctions of the $p$-Laplacian and a related quantitative stability result.
\end{abstract}

\date{\today}
%\thanks{}

\maketitle
\tableofcontents

\section{Introduction}
The technique of symmetrization is in general an important tool in geometry and analysis and have proven to be crucial in the study of geometric and functional inequalities, two examples being the celebrated isoperimetric inequality \cite{0} and Rayleigh-Faber-Krahn inequality \cite{33,34,35}. Recently this technique also found application in the theory of metric geometry with particular interest in the regularity of singular spaces, see \cite{18,11,0.3,0.1,0.2}. 

In 1976, Talenti \cite{36} gave a pointwise comparison between the symmetrization of the solution of the Poisson problem with Dirichlet boundary conditions on an open bounded set $\Omega$ in $\mathbb{R}^n$ and the solution of a symmetrization problem on a Euclidean ball $\Omega^\sharp$  with $\mathscr{L}^n(\Omega)=\mathscr{L}^n(\Omega^\sharp)$.
Talenti’s comparison results
were generalized to semilinear and nonlinear elliptic equations, for instance, in
\cite{39,37,38,11}. Our interest will be in extending Talenti's comparison results to the $p$-Laplacian with Dirichlet boundary conditions on open subsets of a $\mathrm{RCD}(K,N)$ space with $K>0$ and $N\in (1,\infty)$.

In order to state the main theorem, let us introduce some notations about the  one-dimensional
model space and the Schwarz symmetrization.

Let $(\mathrm{X}, \mathrm{d}, \mathfrak{m})$ be a $\operatorname{RCD}(K, N)$ space with $K>0$, $N \in(1, \infty)$, and $\Omega \subset \mathrm{X}$ an open set with $v:=\mathfrak{m}(\Omega)\in(0,1)$. Let $p,q\in (1,\infty)$  be conjugate exponents.

For any $K>0$ and $N \in(1, \infty)$, we define the one dimensional model space $(J_{K, N}, \mathrm{d}_{eu}, \mathfrak{m}_{K, N})$ for the curvature dimension condition of parameters $K$ and $N$ by 
\begin{equation*}
J_{K, N} =[0, \pi \sqrt{(N-1)/K}],\quad \mathfrak{m}_{K, N}=h_{K,N}\mathscr{L}^1\llcorner J_{K, N}
\end{equation*}
where $\mathrm{d}_{eu}$ is the restriction to $J_{K, N}$ of the Euclidean distance over the real line, $\mathscr{L}^1$ is the standard $1$-dimensional Lebesgue measure and
$$
h_{K, N}(t) = \frac{1}{c_{K, N}} \sin ^{N-1}\bigg(t \sqrt{\frac{K}{N-1}}\bigg),t\in J_{K,N},\quad c_{K, N} = \int_{J_{K, N}} \sin ^{N-1}\bigg(t \sqrt{\frac{K}{N-1}}\bigg) \mathrm{d}t.
$$
Notice that $(J_{K, N}, \mathrm{d}_{eu}, \mathfrak{m}_{K, N})$ is a $\mathrm{RCD}(K,N)$ space.

We now introduce the Schwarz symmetrization. To this aim, given an $\mathfrak{m}$-measurable function $u: \Omega \rightarrow\mathbb{R}$ we define its distribution function $\mu:=\mu_u:[0,+\infty) \rightarrow [0, \mathfrak{m}(\Omega)]$ by
$$
\mu(t):=\mathfrak{m}(\{|u|>t\}).
$$
We will let $u^{\sharp}$ be the generalized inverse of $\mu$, defined in the following way:
$$
u^{\sharp}(s):= \begin{cases}\operatorname{ess} \sup |u| & \text { if } s=0 \\ \inf \{t: \mu(t)<s\} & \text { if } s>0.\end{cases}
$$
Then we define the Schwarz symmetrization as follows: first, we consider $r\in J_{K,N}$ such that $\mathfrak{m}(\Omega)=\mathfrak{m}_{K,N}([0,r])$, then we define the Schwarz symmetrization $u^\star=u_{K,N}^\star: [0,r]\rightarrow [0,\infty]$ as
$$
u^\star(x):=u^{\sharp}(\mathfrak{m}_{K,N}([0, x])), \quad \forall x \in[0, r].
$$

Let $f \in L^q(\Omega)$, we consider the Dirichlet problem
\begin{equation}\label{eq}
\begin{cases}
-\mathscr{L}_p u=f &\text { in } \Omega \\ 
u=0  &\text { on } \partial \Omega.
\end{cases}
\end{equation}
A function $u \in W^{1,p}(\Omega)$ is a weak solution to $(\ref{eq})$ if $u \in W_0^{1,p}(\Omega)$ and 
$$
\int_{\Omega} D\psi(\nabla u)|\nabla u|^{p-2} \mathrm{d} \mathfrak{m} =\int_{\Omega} f \psi\mathrm{d} \mathfrak{m} \quad \forall \psi \in W_0^{1,p}(\Omega).
$$
We will establish a comparison between the Schwarz symmetrization of the solution to $(\ref{eq})$ and  the solution to the problem
\begin{equation}\label{eq00}
\begin{cases}
-\mathscr{L}_p w =f^{\star}  \text { in } [0,r_v), \\ w(r_v)=0.
\end{cases}
\end{equation}
where $r_v\in J_{K,N}$ is such that $\mathfrak{m}_{K, N}([0, r_v))=\mathfrak{m}(\Omega)=v$, and $f^{\star}$ is the Schwarz symmetrization of $f$. More precisely, we will prove the following result: 
\begin{theorem}[A Talenti-type comparison theorem]\label{thm}
Let $p,q\in (1,\infty)$  be conjugate exponents. Let $(\mathrm{X}, \mathrm{d}, \mathfrak{m})$ be a $\operatorname{RCD}(K, N)$ space with $K>0$, $N \in(1, \infty)$, and $\Omega \subset \mathrm{X}$ an open set with $\mathfrak{m}(\Omega)=v \in(0,1)$. 

Let $f \in  L^q(\Omega)$. Assume that $u \in W_0^{1,p}(\Omega)$ is a weak solution to $(\ref{eq})$. Let also $w \in W^{1,p}([0,r_v),  \mathfrak{m}_{K, N})$ be a weak solution to $(\ref{eq00})$. Then\\
$(i)$ $u^{\star}(x) \leq w(x)$, for every $x \in [0, r_v]$. In particular, $u\in L^\infty(\Omega)$. In addition, this inequality is sharp.\\
$(ii)$ For any $1 \leq r \leq p$, the following $L^r$-gradient estimate holds:
\begin{equation}\label{40}
\int_{\Omega}|\nabla u|^r \mathrm{d} \mathfrak{m} \leq \int_0^{r_v}|\nabla w|^r \mathrm{d} \mathfrak{m}_{K, N}.
\end{equation}
\end{theorem}

We are also interested in the rigidity and almost rigidity  associated to Theorem $\ref{thm}$. We will prove that \\
$\bullet$ A rigidity result (Theorem $\ref{thm5.6}$) roughly stating that if equality in the Talenti-type comparison Theorem $\ref{thm}$ is achieved, then the space is a spherical suspension;\\
$\bullet$ An almost rigidity result (Theorem $\ref{thm7}$) roughly stating that equality in the Talenti-type comparison Theorem $\ref{thm}$ is almost achieved (in $L^p$-sense) if and only if the space is mGH-close to a spherical suspension.

As an application of the Talenti's comparison, we also obtain a reverse H\"{o}lder inequality for the first Dirichlet eigenfunctions
of the $p$-Laplacian. 
In 1972, following the spirit of the works of Faber and Krahn \cite{33,34}, Payne and Rayner \cite{Pa2} proved a reverse H\"{o}lder inequality for the norms $L_1$ and $L_2$ of the first Dirichlet eigenfunction of the $2$-Laplacian in an open bounded set $\Omega$ of $\mathbb{R}^2$:
$$
\frac{(\int_{\Omega} u^2 \mathrm{d} \mathfrak{m})^{\frac{1}{2}}}{(\int_{\Omega} u \mathrm{d} \mathfrak{m})} \leq \frac{\sqrt{\lambda_1(\Omega)}}{2\sqrt{\pi}}.
$$
where $\lambda_1(\Omega)$ is the first Dirichlet eigenvalue of the $2$-Laplacian in $\Omega$, and $u$ is a first Dirichlet eigenfunction of the $2$-Laplacian in $\Omega$. Equality holds if and only if $\Omega$ is a disc. The inequality was generalized to  Riemannian manifolds with positive Ricci curvature by Gamara-Hasnaoui-Makni \cite{GHM15} and Colladay-Langford-McDonald \cite{38}. We refer also to \cite{PR73,KJ81,Chi82,AB01,BL07,C21,CL23} for further references.
\begin{theorem}
Let $p>1,r>0$ and let $(\mathrm{X}, \mathrm{d}, \mathfrak{m})$ be a $\operatorname{RCD}(N-1, N)$ space with $N>2$, $\Omega \subset \mathrm{X}$ an open set with $\mathfrak{m}(\Omega)=v \in(0,1)$, and $u\in W_0^{1,p}(\Omega)$ a positive first Dirichlet eigenfunction of the $p$-Laplacian in $\Omega$. Suppose that $\alpha \leq v$ is such that $\lambda_{N-1, N, \alpha}^p=\lambda_{\mathrm{X}}^p(\Omega)$. Let $z:[0, r_\alpha] \rightarrow [0,\infty)$ be the first Dirichlet eigenfunction of the $p$-Laplacian in $[0,r_\alpha)$, scaled such that:
$$
\int_{\Omega} u^r \mathrm{d} \mathfrak{m}=\int_0^{r_\alpha} z^r \mathrm{d} \mathfrak{m}_{N-1, N} .
$$
Then, for all $t \geq r$, it holds that:
\begin{equation}\label{1.4}
\frac{(\int_{\Omega} u^t \mathrm{d} \mathfrak{m})^{\frac{1}{t}}}{(\int_{\Omega} u^r \mathrm{d} \mathfrak{m})^{\frac{1}{r}}} \leq \frac{(\int_0^{r_\alpha} z^t\mathrm{d} \mathfrak{m}_{N-1, N})^{\frac{1}{t}}}{(\int_0^{r_\alpha} z^r \mathrm{d} \mathfrak{m}_{N-1, N})^{\frac{1}{r}}}.
\end{equation}
Furthermore, if the equality holds for some $t>\max\{1,r\}$ then $(\mathrm{X}, \mathrm{d}, \mathfrak{m})$ is a spherical suspension.
\end{theorem}

We also prove that\\
$\bullet$ A quantitative stability result (Theorem $\ref{thm8}$) roughly stating that equality in the reverse H\"{o}lder inequality $(\ref{1.4})$ is almost achieved for $t\in Q$ with an unbounded set $Q\subset (p-1,\infty)$ and $r=p-1$ if and only if the diameter of the space is close to $\pi$.

\section{Preliminaries}
Throughout this paper, a metric measure space is a triple $(\mathrm{X}, \mathrm{d}, \mathfrak{m})$ where $(\mathrm{X}, \mathrm{d})$ is a complete and separable metric space and $\mathfrak{m}$ is a \textbf{Borel probability measure} on $\mathrm{X}$ with $\operatorname{supp}[\mathfrak{m}]=\mathrm{X}$. Recall that a set $A \subset \mathrm{X}$ is called $\mathfrak{m}$-measurable if $A$ belongs to the completion of the Borel $\sigma$-algebra on $\mathrm{X}$ with respect to $\mathfrak{m}$ and a function $f:\mathrm{X} \rightarrow [-\infty,+\infty]$ is called $\mathfrak{m}$-measurable if all of its sub-level sets are $\mathfrak{m}$-measurable. We write $f=g$ $\mathfrak{m}$-a.e. if $\{f\neq g\}$ is an $\mathfrak{m}$-null set. Given $p\in [1,\infty]$ and an $\mathfrak{m}$-measurable set $E$, we use the standard notation $L^p(E,\mathfrak{m})$ for the $L^p$
spaces. We abbreviate $L^p(E,\mathfrak{m})$ by $L^p(E)$ or $L^p$ when this will cause no confusion.
\subsection{Riemannian curvature-dimension condition}
Let $(\mathrm{X}, \mathrm{d}, \mathfrak{m})$ be a metric measure space. The set of Borel probability measures on $\mathrm{X}$ will be denoted by $\mathscr{P}(\mathrm{X})$ and $\mathscr{P}_2(\mathrm{X})\subset \mathscr{P}(\mathrm{X})$ is the set of measures with finite $2$-moment, i.e. $\mu\in \mathscr{P}_2(\mathrm{X})$ if $\mu\in \mathscr{P}(\mathrm{X})$ and 
$\int \mathrm{d}^2(x_0, x) \mathrm{d} \mu(x)<\infty$ for some (and thus every) $x_0 \in \mathrm{X}$. $\mathscr{P}_2(\mathrm{X})$ is endowed with the $2$-Wasserstein distance $W_2$, defined by 
$$
W_2^2(\mu_1, \mu_2):=\inf \int \mathrm{d}^2(x, y) \mathrm{d}q(x, y)
$$
where the infimum is taken among all transport plans $q$ of $\mu_1$ and $\mu_2$, i.e. among all $q\in \mathscr{P}(\mathrm{X} \times \mathrm{X})$ such that $\pi_\sharp^1q=\mu_1$ and $\pi_\sharp^2q=\mu_2$. This infimum can be attained and we call optimal all the transport plans $q$ realizing the minimum. Let $\mathscr{P}_2(\mathrm{X}, \mathrm{d}, \mathfrak{m}) \subset \mathscr{P}_2(\mathrm{X})$ be the subspace of all $\mu\in \mathscr{P}_2(\mathrm{X})$ which are absolutely continuous w.r.t. $\mathfrak{m}$.
\begin{definition}
Let $(\mathrm{X}, \mathrm{d}, \mathfrak{m})$ be a metric measure space. We say that $(\mathrm{X}, \mathrm{d}, \mathfrak{m})$ is a $\mathrm{CD}(K,N)$ space, $K \in \mathbb{R}$ and $N \in(1, \infty)$ if for any $\nu_0=\rho_0 \mathfrak{m}$, $\nu_1=\rho_1 \mathfrak{m} \in \mathscr{P}_2(\mathrm{X}, \mathrm{d}, \mathfrak{m})$ there exist an optimal transport plan $q$ of them and a geodesic $(\nu_t:=\rho_t  \mathfrak{m})_{t \in[0,1]}$ in $\mathscr{P}_2(\mathrm{X}, \mathrm{d}, \mathfrak{m})$ connecting them such that for all $t \in[0,1]$ and all $N^{\prime} \geq N$:
\begin{align}
 \int \rho_t^{1-1 / N^{\prime}} \mathrm{d}\mathfrak{m} \geq \int [\tau_{K ,N^{\prime}}^{(1-t)}(\mathrm{d}(x_0, x_1)) \rho_0^{-1 / N^{\prime}}(x_0)+\tau_{K,N^{\prime}}^{(t)}(\mathrm{d}(x_0, x_1)) \rho_1^{-1 / N^{\prime}}(x_1)] \mathrm{d} q(x_0, x_1),\label{con}
\end{align}
where the distorsion coefficients are functions $[0,1]\times [0,\infty)\ni (t,\theta)\mapsto \tau_{K, N}^{(t)}(\theta)\in [0,\infty]$ defined by
\begin{align*}
\tau_{K, N}^{(t)}(\theta) &:= t^{\frac{1}{N}} \sigma_{K/(N-1)}^{(t)}(\theta)^{\frac{N-1}{N}}\\
\sigma_\kappa^{(t)}(\theta)&:= 
\begin{cases}
\frac{\sin (\sqrt{\kappa}  t \theta)}{\sin (\sqrt{\kappa}  \theta)}, & 0<\kappa \theta^2<\pi^2, \\ 
t, & \kappa \theta^2=0, \\ 
\frac{\sinh (\sqrt{-\kappa}  t \theta)}{\sinh (\sqrt{-\kappa}  \theta)}, & \kappa \theta^2<0, \\ +\infty, & \kappa \theta^2 \geq \pi^2 ,
\end{cases}
\end{align*}
where we adopt the convention that $0\cdot +\infty=+\infty$.
\end{definition}
Let $K \in \mathbb{R}$, $N \in(0, \infty)$, the reduced curvature dimension condition $\mathrm{CD}^*(K,N)$ asks for the same inequality $(\ref{con})$ of $\mathrm{CD}(K,N)$ but the coefficients
$\tau_{K,N^{\prime}}^{(t)}(\mathrm{d}(x_0, x_1))$ and $\tau_{K,N^{\prime}}^{(1-t)}(\mathrm{d}(x_0, x_1))$ are replaced by $\sigma_{K,N^{\prime}}^{(t)}(\mathrm{d}(x_0, x_1))$ and $\sigma_{K,N^{\prime}}^{(1-t)}(\mathrm{d}(x_0, x_1))$, respectively. If $N\in (1,\infty)$, it was shown by Cavalletti and Milman in \cite{27} that an essentially nonbranching metric measure space $(\mathrm{X}, \mathrm{d}, \mathfrak{m})$ is a $\mathrm{CD}^*(K,N)$ space if and only if it is a $\mathrm{CD}(K,N)$ space.

In the following proposition we collect those basic properties of $\mathrm{CD}(K, N)$ spaces we will use later on.
\begin{proposition}[Basic properties of $\mathrm{CD}(K, N)$ spaces]\cite{3,5}\label{prop5}
Let $(\mathrm{X}, \mathrm{d}, \mathfrak{m})$ be a $\mathrm{CD}(K, N)$ space, $K \in \mathbb{R}, N\in (1,\infty)$. Then $(\mathrm{X}, \mathrm{d})$ is proper and geodesic, $\mathfrak{m}$ is locally doubling, i.e. it holds
$$
0<\mathfrak{m}(B_{2 r}(x)) \leq C_D(R) \mathfrak{m}(B_r(x)), \quad \forall x \in \mathrm{X}, \quad 0<r \leq R,
$$
for some constant $C_D(R)$ depending only on $K,N,R$, and $(\mathrm{X}, \mathrm{d}, \mathfrak{m})$ supports a weak local Poincar\'{e} inequality, i.e. for any locally Lipschitz function $u: \mathrm{X} \rightarrow \mathbb{R}$ and any upper gradient $g$ of $u$ and $R>0$ it holds
\begin{equation}\label{poin}
\fint_{B_r(x)}\bigg|u-\fint_{B_r(x)} u \mathrm{d} \mathfrak{m}\bigg| \mathrm{d} \mathfrak{m} \leq C_P(R) r  \fint_{B_{2 r}(x)} g \mathrm{d} \mathfrak{m}, \quad \forall 0<r\leq R,
\end{equation}
for some constant $C_P(R)$ depending only on $K,N,R$. Also, the Bishop-Gromov comparison estimates hold,
i.e. for any $x \in \mathrm{X}$ it holds
$$
\frac{\mathfrak{m}(B_r(x))}{\mathfrak{m}(B_R(x))} \geq \begin{cases}\frac{\int_0^r \sin (t \sqrt{K /(N-1)})^{N-1} \mathrm{d} t}{\int_0^R \sin (t \sqrt{K /(N-1)})^{N-1} \mathrm{d} t} & \text { if } K>0 \\ \frac{r^N}{R^N} & \text { if } K=0 \\ \frac{\int_0^r \sinh (t \sqrt{K /(N-1)})^{N-1} \mathrm{d} t}{\int_0^R \sinh (t \sqrt{K /(N-1)})^{N-1} \mathrm{d} t} & \text { if } K<0\end{cases}
$$
for any $0<r \leq R \leq \pi \sqrt{(N-1) /(\max \{K, 0\})}$.

Finally, if $K>0$, then $(\mathrm{X}, \mathrm{d})$ is compact, with diameter at most $\pi\sqrt{\frac{N-1}{K}}$ and $C_P(R)$ can be taken by $2^{N+3}$ for all $R>0$.
\end{proposition}

Recall that $C([0,1], \mathrm{X})$ is the space of continuous curves from $[0,1]$ to $\mathrm{X}$ equipped with the sup norm, which is complete and separable. For $t \in[0,1]$, the evaluation map $\mathrm{e}_t: C([0,1], \mathrm{X}) \rightarrow \mathrm{X}$ is given by
$$
\mathrm{e}_t(\gamma):=\gamma_t, \quad \forall \gamma \in C([0,1], \mathrm{X}) .
$$
For $s, t \in[0,1]$, the restriction operator $\operatorname{restr}_s^t: C([0,1], \mathrm{X}) \rightarrow C([0,1], \mathrm{X})$ is defined by
$$
(\operatorname{restr}_s^t(\gamma))_r:=\gamma_{s+r(t-s)}, \quad \forall \gamma \in C([0,1], \mathrm{X}) .
$$
A curve $\gamma \in C([0,1], \mathrm{X})$ is said absolutely continuous if there exists $f \in L^1([0,1])$ such that
\begin{equation}\label{2.2}
\mathrm{d}(\gamma_t, \gamma_s) \leq \int_s^t f(r) \mathrm{d} r, \quad \forall \; 0\leq s<t\leq 1.
\end{equation}
The set of absolutely continuous curves from $[0,1]$ to $\mathrm{X}$ will be denoted by $\mathrm{AC}([0,1], \mathrm{X})$. More generally, if the function $f$ in $(\ref{2.2})$ belongs to $L^q([0,1])$, $q \in[1, \infty]$, $\gamma$ is said $q$-absolutely continuous, and $\mathrm{AC}^q([0,1], \mathrm{X})$ is the corresponding set of $q$-absolutely continuous curves. It turns out (see for instance \cite[Theorem 1.2.5]{1}) that for $\gamma \in \mathrm{AC}^q([0,1], \mathrm{X})$, the limit
$$
\lim_{h \rightarrow 0} \frac{\mathrm{d}(\gamma_{t+h}, \gamma_t)}{|h|},
$$
exists for $\mathscr{L}^1$-a.e. $t \in[0,1]$, and defines an $L^q$ function. Such function is called metric speed or metric derivative, is denoted by $|\dot{\gamma}_t|$ and is the minimal (in the $\mathscr{L}^1$-a.e. sense) $L^q$ function which can be chosen as $f$ in the right hand side of $(\ref{2.2})$.

\begin{definition}[$q$-test plan]
Let $\pi \in \mathscr{P}(C([0,1], \mathrm{X}))$. We say that $\pi$ has bounded compression if there exists $C>0$ such that
$$
(\mathrm{e}_t)_{\sharp} \pi \leq C \mathfrak{m}, \quad \forall t \in[0,1] .
$$
For $q \in(1, \infty)$, we say that $\pi$ is a $q$-test plan if it has bounded compression, is concentrated on $\mathrm{AC}^q([0,1], \mathrm{X})$ and
$$
\iint_0^1|\dot{\gamma}_t|^q \mathrm{d} t \mathrm{d} \pi(\gamma)<\infty.
$$
\end{definition}
\begin{definition}[Sobolev classes]
Let $p,q\in (1,\infty)$  be conjugate exponents. A Borel function $f:\mathrm{X}\rightarrow \mathbb{R}$ belongs to the Sobolev class $\mathrm{S}^p(\mathrm{X})$ if there exists a Borel function $G \in L^p(\mathrm{X})$ such that
\begin{equation}\label{2.10}
\int|f(\gamma_1)-f(\gamma_0)| \mathrm{d}\pi(\gamma) \leq \iint_0^1 G(\gamma_s)|\dot{\gamma}_s| \mathrm{d} s \mathrm{d} \pi(\gamma), \quad \forall \text { $q$-test plan }\pi.
\end{equation}
In this case, $G$ is called a $p$-weak upper gradient of $f$.
\end{definition}
Arguing as in \cite[Section 4.5]{2}, we get that for $f \in \mathrm{S}^p(\mathrm{X})$ there exists a minimal function $G \geq 0$, in the $\mathfrak{m}$-a.e. sense, in $L^p(\mathrm{X})$ such that $(\ref{2.10})$ holds. We will denote such minimal function by $|\nabla f|_{p}$.

The $p$-Cheeger energy, denoted by $\mathrm{Ch}_p: L^p(\mathrm{X}) \rightarrow[0, \infty]$, is defined by
\begin{equation*}
\operatorname{Ch}_p(f):=
\begin{cases}
\displaystyle{\frac{1}{p}\int |\nabla f|_{p}^p \mathrm{d}\mathfrak{m}}, &\text{ if }f\in  \mathrm{S}^p(\mathrm{X})\\
+\infty,&\text{ otherwise. }
\end{cases}
\end{equation*}
It is well known that $\mathrm{Ch}_p$ is a convex and lower semi-continuous functional on $L^p(\mathrm{X})$. We define the Sobolev space $W^{1, p}(\mathrm{X}):=L^p(\mathrm{X})\cap \mathrm{S}^p(\mathrm{X})$. $W^{1,p}(\mathrm{X})$ becomes a Banach space, when endowed with the norm
$$
\|f\|_{W^{1, p}(\mathrm{X})} =\{\|f\|_{L^p(\mathrm{X})}^p+p \mathrm{Ch}_p(f)\}^{\frac{1}{p}}.
$$
\begin{definition}
A metric measure space $(\mathrm{X}, \mathrm{d}, \mathfrak{m})$ is called infinitesimally Hilbertian if the $2$-Cheeger energy $\mathrm{Ch}_2$ is quadratic. Moreover, we say that $(\mathrm{X}, \mathrm{d}, \mathfrak{m})$ is a $\mathrm{RCD}(K, N)$$(\mathrm{RCD}^*(K, N))$ space, $K \in \mathbb{R}$ and $N \in(1, \infty)$, if it is infinitesimally Hilbertian and a $\mathrm{CD}(K, N)$$(\mathrm{CD}^*(K, N))$ space.
\end{definition}
Let $K\in \mathbb{R}$, $N \in(1, \infty)$, it is well known that $(\mathrm{X}, \mathrm{d}, \mathfrak{m})$ is a $\mathrm{RCD}(K, N)$ space if and only if it is a $\mathrm{RCD}^*(K, N)$ space, see \cite{27}. Notice that several different notions of Sobolev spaces on metric measure spaces have been established in \cite{6,7,8,9}. They are equivalent to each other on $\mathrm{RCD}(K, N)$ metric measure spaces (see, for example \cite{2}).

\begin{remark}\cite[Theorem 3.3]{in}
A non-trivial fact about a locally doubling space $(\mathrm{X}, \mathrm{d}, \mathfrak{m})$ supporting a weak local Poincar\'{e} inequality (and $\mathrm{RCD}(K, N)$ spaces are so) is that they have $p$-independent minimal weak upper gradients in the following sense: for any $p_1, p_2 \in(1, \infty)$, \\
(i) $W^{1, p_1}(\mathrm{X}) \cap W^{1, p_2}(\mathrm{X})$ is dense in both $W^{1, p_1}(\mathrm{X})$ and $W^{1, p_2}(\mathrm{X})$;\\
(ii) for any $f \in W^{1, p_1}(\mathrm{X}) \cap W^{1, p_2}(\mathrm{X})$ it holds $|\nabla f|_{p_1}=|\nabla f|_{p_2}$ $\mathfrak{m}$-a.e.;\\
(iii) any $f \in W^{1, p_1}(\mathrm{X})$ with $f,|\nabla f|_{p_1} \in L^{p_2}(\mathfrak{m})$ belongs to $W^{1, p_2}(\mathrm{X})$.	

Hence, in the following, we drop the dependence on $p$ on the notation $|\nabla f|_{p}$ and write $|\nabla f|$.
\end{remark}

It is possible to check that the object $|\nabla f|$ is local in the sense that
\begin{equation}\label{2.14}
\forall f \in \mathrm{S}^p(\mathrm{X}) \text { it holds }|\nabla f|=0, \quad \mathfrak{m}\text{-a.e. on } f^{-1}(N), \quad \forall N \subset \mathbb{R}, \text { s.t. } \mathcal{L}^1(N)=0
\end{equation}
and
\begin{equation}\label{2.15}
|\nabla f|=|\nabla g|, \quad \mathfrak{m}\text{-a.e. } \quad \text { on }\{f=g\}, \quad \forall f, g \in \mathrm{S}^p( \mathrm{X}).
\end{equation}
Also, for $f \in \mathrm{S}^p(\mathrm{X})$ and $\varphi: \mathbb{R} \rightarrow \mathbb{R}$ Lipschitz, the function $\varphi \circ f$ belongs to $\mathrm{S}^p(\mathrm{X})$ as well and it holds
\begin{equation}\label{2.16}
|\nabla(\varphi \circ f)|=|\varphi^{\prime} \circ f| | \nabla f| \quad \mathfrak{m}\text {-a.e.. }
\end{equation}

\subsection{Local Sobolev spaces and BV functions}
Let $(\mathrm{X}, \mathrm{d}, \mathfrak{m})$ be a $\mathrm{RCD}(K, N)$ space, $K \in \mathbb{R}$ and $N \in(1, \infty)$. Fix an open set $\Omega$ in $\mathrm{X}$, $\operatorname{Lip}(\Omega)$ will denote the space of Lipschitz functions on $\Omega$, while $\operatorname{Lip}_\mathrm{c}(\Omega)$ will be the subspace of Lipschitz functions on $\mathrm{X}$ with compact support in $\Omega$.
\begin{definition}
Let $p \in(1,\infty)$, we define:\\
$(a)$ The space $W_0^{1, p}(\Omega)$ as the closure of $\operatorname{Lip}_c(\Omega)$ with respect to the norm of $W^{1, p}(\mathrm{X})$.\\
$(b)$ The space $\hat{W}_0^{1,p}(\Omega)$ as the set of all $f \in W^{1,p}(\mathrm{X})$ such that $f=0$ $\mathfrak{m}$-a.e. in $\Omega^c$.
\end{definition}
It is trivial that $W_0^{1,p}(\Omega)$ and $\hat{W}_0^{1,p}(\Omega)$ are closed subspaces of $W^{1, p}(\mathrm{X})$, with $W_0^{1,p}(\Omega)\subset\hat{W}_0^{1,p}(\Omega)$. A kind of "reverse" inclusion is provided in the following lemma, whose proof is a slight modification of  \cite[Lemma 2.12]{41}.
\begin{lemma}
For all $x \in \mathrm{X}$ and $R>0$, one has
$$
\hat{W}_0^{1,p}(B_R(x))=\bigcap_{\epsilon>0} W_0^{1,p}(B_{R+\epsilon}(x)).
$$
In addition, for all $x \in \mathrm{X}$, the equality $\hat{W}_0^{1,p}(B_R(x))=W_0^{1,p}(B_R(x))$ holds with at most countably many exceptions. 
\end{lemma}

\begin{definition}[The local Sobolev class $\mathrm{S}^p(\Omega)$]\label{defsobo}
Let $p \in(1, \infty)$. The space $\mathrm{S}^p(\Omega)$ is the space of Borel functions $g: \Omega \rightarrow \mathbb{R}$ such that $g \chi \in \mathrm{S}^p(\mathrm{X})$ for any $\chi\in \operatorname{Lip}_c(\Omega)$, where $g \chi$ is taken $0$ by definition on $\mathrm{X} \backslash \Omega$.\\
It is immediate to verify that this is a good definition, and that for $g \in \mathrm{S}^p(\Omega)$ the function $|\nabla g| \in L^p(\Omega)$ is well defined, thanks to $(\ref{2.15})$, by
$$
|\nabla g|:=|\nabla(g \chi)|, \quad \mathfrak{m} \text {-a.e. on }\{\chi=1\}
$$
where $\chi\in \operatorname{Lip}_c(\Omega)$. \\
We say that $f \in L^p(\Omega)$ belongs to $W^{1,p}(\Omega)$ if
$|\nabla f| \in L^p(\Omega)$.
\end{definition}
\begin{definition}[$D^{\pm} f(\nabla g)$]
Let $f, g \in \mathrm{S}^p(\mathrm{X})$. The functions $D^{\pm} f(\nabla g):\mathrm{X} \rightarrow \mathbb{R}$ are $\mathfrak{m}$-a.e. well defined by
\begin{align*}
D^{+} f(\nabla g) & :=\mathrm{ess} \inf_{\epsilon>0} \frac{|\nabla(g+\epsilon f)|^p-|\nabla g|^p}{p \epsilon|\nabla g|^{p-2}}, \\
D^{-} f(\nabla g) & :=\mathrm{ess} \sup_{\epsilon<0} \frac{|\nabla(g+\epsilon f)|^p-|\nabla g|^p}{p \epsilon|\nabla g|^{p-2}},
\end{align*}
on $\{x:|\nabla g|(x) \neq 0\}$, and are taken $0$ by definition on $\{x:|\nabla g|(x)=0\}$.		
\end{definition}
Obviously, for any $f, g \in \mathrm{S}^p(\mathrm{X})$, we have 
\begin{equation}\label{24}
|D^{\pm} f(\nabla g)|\leq |\nabla f||\nabla g|,\quad \mathfrak{m}\text{-a.e.}.
\end{equation}
\begin{definition}
Let $p,q\in (1,\infty)$  be conjugate exponents. We say that $(\mathrm{X}, \mathrm{d}, \mathfrak{m})$ is $q$-infinitesimally strictly convex, if
\begin{equation}\label{3.10}
\int D^{+} f(\nabla g)|\nabla g|^{p-2}\mathrm{d} \mathfrak{m}= \int D^{-} f(\nabla g)|\nabla g|^{p-2}\mathrm{d} \mathfrak{m},\;\;  \forall f, g \in \mathrm{S}^p(\mathrm{X}).
\end{equation}
\end{definition}
\begin{remark}
Since $D^{-} f(\nabla g)\leq D^{+} f(\nabla g)$ $\mathfrak{m}$-a.e. for any $f, g \in \mathrm{S}^p(\mathrm{X})$, we get that the integral equality $(\ref{3.10})$ is equivalent to the pointwise one:
$$
 D^{+} f(\nabla g)=D^{-} f(\nabla g), \quad \mathfrak{m}\text{-a.e. }, \quad \forall f, g \in \mathrm{S}^p(\mathrm{X}).
$$
On $q$-infinitesimally strictly convex spaces, we will denote by $D f(\nabla g)$ the common value of $D^{+} f(\nabla g)=D^{-} f(\nabla g)$, for $f,g\in \mathrm{S}^p(\mathrm{X})$.
\end{remark}

We recall the following proposition in \cite[Proposition 5.9]{in}, which allows us to define a notion of (local) $p$-Laplacian(Definition $\ref{Lap}$) on $\mathrm{RCD}(K,N)$ spaces.
\begin{proposition}\label{prop2.5}
Let $(\mathrm{X}, \mathrm{d}, \mathfrak{m})$ be a $\mathrm{RCD}(K,N)$ space with $K\in \mathbb{R}$, $N\in (1,\infty)$. Then there are unique couples $(L^0(T^*\mathrm{X}),\mathrm{d})$ and $(L^0(T\mathrm{X}),\nabla)$ of universal cotangent and tangent module, with associated linear universal differential and gradient, respectively. In particular, for any $p\in (1,\infty)$, $(\mathrm{X}, \mathrm{d}, \mathfrak{m})$ is $q$-infinitesimally strictly convex, where $q$ is the conjugate exponent of $p$.
\end{proposition}
We state here a collection of useful facts concerning local Sobolev functions: these can be found for instance in \cite[Chapter 3,4]{10}.
\begin{proposition}\label{propp}
Let $p,q\in (1,\infty)$  be conjugate exponents. Assume that $(\mathrm{X}, \mathrm{d}, \mathfrak{m})$ is $q$-infinitesimally strictly convex, then for any $f,g\in \mathrm{S}^p(\Omega)$.\\
$(i)$ The functions $D f(\nabla g):\Omega\rightarrow [-\infty,+\infty]$ is well defined by 
$$
D f(\nabla g):= D (f\chi)(\nabla g\chi), \quad \mathfrak{m}\text{-a.e. on }\{\chi=1\},
$$
where $\chi\in \operatorname{Lip}_c(\Omega)$.\\
$(ii)$ \underline{Chain rules.} Let $\varphi: \mathbb{R} \rightarrow \mathbb{R}$ be Lipschitz. Then $\varphi \circ f \in \mathrm{S}^p(\Omega)$ and it holds
$$
D(\varphi \circ f)(\nabla g)=\varphi^{\prime} \circ f D f(\nabla g), \quad \mathfrak{m}\text{-a.e.},
$$
where the value of $\varphi^{\prime} \circ f$ is defined arbitrarily at points $x$ where $\varphi$ is not differentiable at $f(x)$. \\
Similarly, assume that $\varphi: \mathbb{R} \rightarrow \mathbb{R}$ is Lipschitz. Then $\varphi \circ g \in \mathrm{S}^p(\Omega)$ and it holds
$$
D f(\nabla(\varphi \circ g))=\varphi^{\prime} \circ g D f(\nabla g), \quad \mathfrak{m}\text{-a.e.},
$$
where the value of $\varphi^{\prime} \circ g$ is defined arbitrarily at points $x$ where $\varphi$ is not differentiable at $g(x)$.\\
$(iii)$ \underline{Leibniz rule.} Let $f_1, f_2 \in \mathrm{S}^p\cap L^{\infty}(\Omega)$ and $g \in \mathrm{S}^p(\Omega)$.
Then it holds
\begin{align*}
D(f_1 f_2)(\nabla g) = f_1 D f_2(\nabla g)+f_2 D f_1(\nabla g),\quad \mathfrak{m}\text{-a.e.},
\end{align*}
$(iv)$ \underline{Linearity.} The map
$$
\mathrm{S}^p(\Omega) \ni f \quad \mapsto \quad D f(\nabla g)
$$
is linear $\mathfrak{m}$-a.e., i.e.
$$
D(\alpha_1 f_1+\alpha_2 f_2)(\nabla g)=\alpha_1 D f_1(\nabla g)+\alpha_2 D f_2(\nabla g), \quad \mathfrak{m}\text{-a.e., }
$$
for any $f_1, f_2 \in \mathrm{S}^p(\Omega), \alpha_1, \alpha_2 \in \mathbb{R}$.
\end{proposition}

\begin{definition}[The local $p$-Laplacian]\label{Lap}
Let $p,q\in (1,\infty)$  be conjugate exponents. Assume that the space $(\mathrm{X}, \mathrm{d}, \mathfrak{m})$ is $q$-infinitesimally strictly convex space and $\Omega \subset \mathrm{X}$ is an open set. The $p$-Laplacian on $\Omega$ is an operator $\mathscr{L}_p$ on $W^{1,p}(\Omega)$ defined as the follows. For each function $u \in W^{1,p}(\Omega)$, its $p$-Laplacian
$\mathscr{L}_p u$ is a functional acting on $W_0^{1,p}(\Omega)$ given by
$$
\mathscr{L}_p u(v):=-\int_{\Omega} Dv(\nabla u)|\nabla u|^{p-2} \mathrm{d} \mathfrak{m} \quad \forall\; v \in W_0^{1,p}(\Omega) .
$$
\end{definition} 
\begin{remark}
Using Proposition $\ref{propp}$(iv) and $(\ref{24})$, it is easy to see that $\mathscr{L}_p u$ is a bounded linear functional on $W_0^{1,p}(\Omega)$ for any $u \in W^{1,p}(\Omega)$.
\end{remark}

\begin{definition}[The Dirichlet problem]\label{defDiri}
Let $p,q\in (1,\infty)$  be conjugate exponents. Assume that the space $(\mathrm{X}, \mathrm{d}, \mathfrak{m})$ is $q$-infinitesimally strictly convex space and $\Omega \subset \mathrm{X}$ is an open set. Let $f \in L^q(\Omega)$. We say that a function $u \in W^{1,p}(\Omega)$ is a weak solution to the Dirichlet problem
\begin{equation*}
\begin{cases}
-\mathscr{L}_p u=f & \text { in } \Omega \\ 
u=0 & \text { on } \partial \Omega
\end{cases}
\end{equation*}
if $u \in W_0^{1,p}(\Omega)$ and
\begin{equation}\label{28}
-\mathscr{L}_p u(v)=\int_{\Omega} f v \mathrm{d}\mathfrak{m}, \quad \text { for any } v \in W_0^{1,p}(\Omega).
\end{equation}		
\end{definition} 
For $u\in W^{1,p}(\Omega)$ and $f\in L^q(\Omega)$, we say that $u$ is a weak solution of $-\mathscr{L}_p u=f$ if $(\ref{28})$ holds.

\vspace{0.5cm}

Finally, we recall the notions of functions of bounded variation and sets of finite perimeter. We start by recalling the notion of slope of a real valued function. Let $u: \mathrm{X} \rightarrow \mathbb{R}$ be a real valued function. We define the slope of $u$ at the point $x \in \mathrm{X}$ as
$$
\operatorname{lip} (u)(x) = \begin{cases}
\displaystyle{\limsup_{y \rightarrow x} \frac{|u(x)-u(y)|}{\mathrm{d}(x, y)}} & \text { if } x \text { is not isolated } \\ 
0 & \text { otherwise. }\end{cases}
$$
\begin{definition}[BV functions and perimeter]
Given $f \in L^1(\Omega)$ we define
$$
\|D f\|(A):=\inf \bigg\{\liminf_i \int_A \operatorname{lip} f_i \mathrm{d}\mathfrak{m}: f_i \in \operatorname{Lip}_{\mathrm{loc}}(A), f_i \rightarrow f \text { in } L^1(A)\bigg\}
$$
for any open set $A \subset \Omega$. A function $f \in L^1(\Omega)$ is said to belong to the space of bounded variation functions $\mathrm{BV}(\Omega)$ if $\|D f\|(\Omega)<+\infty$. If $E \subset \mathrm{X}$ is an $\mathfrak{m}$-measurable set and $A \subset \mathrm{X}$ is open, we define the perimeter $\operatorname{Per}(E, A)$ of $E$ in $A$ by
$$
\operatorname{Per}(E, A):=\|D \chi_E\|(A).
$$
We say that $E$ has finite perimeter if $\operatorname{Per}(E, \mathrm{X})<+\infty$, and we denote $\operatorname{Per}(E):=\operatorname{Per}(E, \mathrm{X})$.
\end{definition}
Here, $\operatorname{Lip}_\mathrm{loc}(A)$ will denote the space of locally Lipschitz functions, i.e. for all $x\in A$ there exists $r>0$ such that $B_r(x)\subset A$ and the restriction of $f$ to $B_r(x)$ is Lipschitz. Let us remark that when $f \in \mathrm{BV}(\mathrm{X})$, the set function $\|D f\|$ above is the restriction to open sets of a Borel measure that we still denote by $\|D f\|$, see \cite{14,AD14}.

In the sequel, we shall frequently make use of the following coarea formula.
\begin{theorem}[Coarea formula]\cite{14,15}
Let $f \in L^1(\Omega)$ be given. Then for any $\mathfrak{m}$-measurable function $g:\Omega\rightarrow [0,+\infty]$ and any $t\in \mathbb{R}$, it holds that 
$$
\int_{\{f>t\}} g\mathrm{d}\|D f\|=\int_t^\infty\int g \mathrm{d}\operatorname{Per}(\{f>r\}) \mathrm{d}r .
$$
Moreover, if $f\in W_0^{1,p}(\Omega)$ is a non-negative function, then $\|Df\|=|\nabla f|\mathfrak{m}$, in other words, for any $\mathfrak{m}$-measurable function $g:\Omega\rightarrow [0,+\infty]$ and any $t>0$, it holds that
\begin{equation}\label{coarea}
\int_{\{f>t\}} g|\nabla f|\mathfrak{m}=\int_t^\infty\int g \mathrm{d}\operatorname{Per}(\{f>r\}) \mathrm{d} r .
\end{equation}
\end{theorem}

\begin{definition}[Isoperimetric profile]
Let $(\mathrm{X}, \mathrm{d}, \mathfrak{m})$ be a metric measure space. The isoperimetric profile $\mathcal{I}_{(\mathrm{X}, \mathrm{d}, \mathfrak{m})}:[0,1] \rightarrow[0,+\infty]$ is defined as
$$
\mathcal{I}_{(\mathrm{X}, \mathrm{d}, \mathfrak{m})}(v)= \inf \{\operatorname{Per}(E) : E\subset \mathrm{X}\text{ Borel }, \mathfrak{m}(E)=v\}, \quad v \in[0,1]
$$	
\end{definition}

\subsection{Measured Gromov-Hausdorff convergence and stability of $\operatorname{RCD}(K, N)$}
In this subsection, let us recall the notion of measured Gromov-Hausdorff convergence (mGH for short). The readers can refer to \cite{19} for a discussion on the various notions of convergence of metric measure spaces and their relations. Since in this paper we will apply it to compact metric measure space, we will restrict to this framework for simplicity. From now on, we set $\overline{\mathbb{N}}_+=\mathbb{N}_+\cup \{\infty\}$.
\begin{definition}
A sequence $(\mathrm{X}_j, \mathrm{d}_j, \mathfrak{m}_j)$ of compact metric measure space is said to converge in the $\mathrm{mGH}$ sense to a compact metric measure space $(\mathrm{X}_\infty, \mathrm{d}_\infty, \mathfrak{m}_\infty)$ if there exists a complete and separable metric space $(\mathrm{Y}, \mathrm{d})$ and isometric embeddings $\{\iota_j:(\mathrm{X}_j, \mathrm{d}_j) \rightarrow(\mathrm{Y}, \mathrm{d})\}_{j \in \overline{\mathbb{N}}_+}$ such that, for every $\epsilon>0$, there exists $j_0$ such that, for every $j>j_0$,
$$
\iota_{\infty}(\mathrm{X}_{\infty}) \subset B_{\epsilon}[\iota_j(\mathrm{X}_j)] \text { and } \iota_j(\mathrm{X}_j) \subset B_{\epsilon}[\iota_{\infty}(\mathrm{X}_{\infty})],
$$
where $B_{\epsilon}[A]:=\{z \in \mathrm{Y}: \mathrm{d}(z, A)<\epsilon\}$ for every subset $A \subset Y$, and
$$
\int_{\mathrm{Y}} \varphi\mathrm{d}(\iota_j)_{\sharp}(\mathfrak{m}_j) \rightarrow \int_{\mathrm{Y}} \varphi\mathrm{d}(\iota_{\infty})_{\sharp}(\mathfrak{m}_{\infty}) \quad \text { for all } \varphi \in C_{\mathrm{b}}(\mathrm{Y}),
$$
where $C_{\mathrm{b}}(\mathrm{Y})$ denotes the set of real-valued bounded continuous functions in $\mathrm{Y}$.
\end{definition}
The following theorem summarizes the compactness/stability properties we will use in the proof of the almost rigidity result.
\begin{theorem}[metrizability and compactness]\label{com}
Let $K>0$ and $N\in (1,\infty)$ be fixed. Then the $\mathrm{mGH}$ convergence restricted to $($isomorphism classes of$)$ $\operatorname{RCD}(K, N)$ spaces is metrizable by a distance function $\mathrm{d}_{\mathrm{mGH}}$. Furthermore, every sequence $(\mathrm{X}_j, \mathrm{d}_j, \mathfrak{m}_j)$ of $\operatorname{RCD}(K, N)$ spaces admits a subsequence which converges in the $\mathrm{mGH}$ sense to a $\operatorname{RCD}(K, N)$ space.
\end{theorem}

\subsection{1-dimensional model spaces}
In this subsection we recall the $1$-dimensional "model" metric measure spaces with Ricci curvature bounded below by $K>0$ and dimension bounded above by $N \in(1, \infty)$ on which we will construct the needed symmetrizations. Let $K>0$ and $N \in(1, \infty)$. Let $J_{K, N}$ be the interval $
J_{K, N} =[0, \pi \sqrt{(N-1)/K}]$
and define the following probability density function on $J_{K, N}$:
$$
h_{K, N}(t) = \frac{1}{c_{K, N}} \sin ^{N-1}\bigg(t \sqrt{\frac{K}{N-1}}\bigg),\text{ where }c_{K, N} = \int_{J_{K, N}} \sin ^{N-1}\bigg(t \sqrt{\frac{K}{N-1}}\bigg) \mathrm{d}t.
$$
\begin{definition}[Model spaces]
Given $K>0$, $N \in(1, \infty)$. We define the one dimensional model space with curvature parameter $K$ and dimension parameter $N$ as $(J_{K, N}, \mathrm{d}_{eu}, \mathfrak{m}_{K, N})$, where $\mathfrak{m}_{K, N}= h_{K, N} \mathscr{L}^1\llcorner J_{K, N}$, $\mathrm{d}_{eu}$ is the restriction to $J_{K,N}$ of the Euclidean distance over the real line and $\mathscr{L}^1$ is the standard $1$-dimensional Lebesgue measure.			
\end{definition}
We will also denote by $H_{K,N}$ the cumulative distribution function of $\mathfrak{m}_{K,N}$, i.e. $H_{K,N}(x)=\mathfrak{m}_{K,N}([0,x])$ for any $x\in J_{K, N}$.\\
The following Lemma is an elementary consequence of the definitions of $h_{K,N}$ and $H_{K,N}$:
\begin{lemma}\label{lem2.11}
Let $K>0$ and $N \in(1, \infty)$ be fixed. Then:\\
$(i)$ If $\gamma_1(K, N) = c_{K, N}^{-1}(K/(N-1))^{\frac{N-1}{2}}$, then
$$
\lim _{t \rightarrow 0^{+}} \frac{h_{K, N}(t)}{t^{N-1}}=\gamma_1(K, N), \quad \text { and } \quad h_{K, N}(t) \leq \gamma_1(K, N) t^{N-1} \quad \forall t \in J_{K, N} .
$$
Moreover, for any $r_1 \in(0, \pi \sqrt{(N-1)/K})$ there exists $C=C(r_1, K, N)>0$ such that
$$
h_{K, N}(t) \geq C t^{N-1}, \quad \forall t \in[0,r_1].
$$
$(ii)$ $H_{K, N}$ is invertible on $J_{K, N}$; moreover, if $\gamma_2(K, N) = \gamma_1(K, N)/N$:
$$
\begin{array}{ll}
\displaystyle{\lim_{t \rightarrow 0^{+}} \frac{H_{K, N}(t)}{t^N}=\gamma_2(K, N)}, \quad & \text { and } \quad H_{K, N}(t) \leq \gamma_2(K, N) t^N \quad \forall t \in J_{K, N} ; \\
\quad\\
\displaystyle{\lim_{t \rightarrow 0^{+}} \frac{H_{K, N}^{-1}(t)}{t^{\frac{1}{N}}}=\frac{1}{\gamma_2(K, N)^{\frac{1}{N}}}}, \quad& \text { and } \quad H_{K, N}^{-1}(t) \geq \displaystyle{\frac{t^{\frac{1}{N}}}{\gamma_2(K, N)^{\frac{1}{N}}}} \quad \forall t \in[0,1].
\end{array}
$$	
\end{lemma}
For the model space $J_{K,N}$ defined before, we can find an almost explicit expression for the isoperimetric profile.
\begin{lemma}[Isoperimetric profile of $(J_{K, N},\mathrm{d}_{eu}, \mathfrak{m}_{K, N})$]\cite[Lemma 2.12]{11}\label{lem2.12}
The isoperimetric profile $\mathcal{I}_{K, N}$ of the model space $(J_{K, N},\mathrm{d}_{eu}, \mathfrak{m}_{K, N})$
is given by the following formula:
$$
\mathcal{I}_{K, N}(v)=h_{K, N}(H_{K, N}^{-1}(v)), \quad v \in[0,1] .
$$
\end{lemma}
Finally, we recall the L\'{e}vy-Gromov isoperimetric inequality in $\mathrm{RCD}$ spaces, as obtained by Cavalletti and Mondino in \cite{12} (for the Minkowski content) and in \cite{13} (for the perimeter).
\begin{proposition}[L\'{e}vy-Gromov inequality]\label{levy}
Let $(\mathrm{X}, \mathrm{d}, \mathfrak{m})$ be a $\operatorname{RCD}(K, N)$ space with $K>0$ and $N\in (1,\infty)$. Then for any Borel set $E\subset \mathrm{X}$,
$$ \mathrm{Per}(E)\geq \mathcal{I}_{K, N}(\mathfrak{m}(E)).$$
\end{proposition}
\begin{remark}\label{rmk2}
Let $(\mathrm{X}, \mathrm{d}, \mathfrak{m})$ be a $\operatorname{RCD}(K, N)$ space with $K>0$ and $N\in (1,\infty)$. There exists $\omega:(0, \infty) \rightarrow (0,1 / 2]$, depending on $K,N$, such that for all $\epsilon>0$ one has the implication
$$
\mathfrak{m}(E) \leq \omega(\epsilon) \quad \Longrightarrow \quad \mathfrak{m}(E) \leq \epsilon\mathrm{Per}(E)
$$
for any Borel set $E \subset \mathrm{X}$.
\end{remark}

\subsection{Rearrangements and symmetrizations}
Throughout this subsection, $(\mathrm{X}, \mathrm{d}, \mathfrak{m})$ will be a metric measure space and $\Omega \subset \mathrm{X}$ an open subset.
\begin{definition}[Distribution function and Decreasing rearrangement]
Let $u: \Omega \rightarrow \mathbb{R}$ be an $\mathfrak{m}$-measurable function. We define $\mu=\mu_u:[0,+\infty) \rightarrow[0, \mathfrak{m}(\Omega)]$, the distribution function of $u$, as
$$
\mu(t):=\mathfrak{m}(\{|u|>t\}) .
$$
For $u$ and $\mu$ as above, we define $u^{\sharp}:[0, \mathfrak{m}(\Omega)]\rightarrow [0,+\infty]$, the decreasing rearrangement of $u$, as
$$
u^{\sharp}(s):= \begin{cases}\operatorname{ess} \sup |u| & \text { if } \mathrm{s}=0 \\ \inf \{t: \mu(t)<s\} & \text { if } \mathrm{s}>0.\end{cases}
$$	
\end{definition}
It is not difficult to check that $u^{\sharp}$ is non-increasing and left-continuous and $u^{\sharp}$ is right-continuous at $0$.

Given $K>0,N\in (1,\infty)$ and $u: \Omega \rightarrow\mathbb{R}$ an $\mathfrak{m}$-measurable function, we define the Schwarz symmetrization as follows: first, we consider $r>0$ such that $\mathfrak{m}(\Omega)=\mathfrak{m}_{K,N}([0,r])$, then we define the Schwarz symmetrization $u_{K,N}^\star=u^\star: [0,r]\rightarrow [0,\infty]$ as
$$
u^\star(x):=u^{\sharp}(\mathfrak{m}_{K,N}([0, x])), \quad \forall x \in[0, r].
$$

\begin{proposition}[Some properties of $u^\sharp$ and $u^*$]\label{3}
With the notations above, we have\\
$(a)$ $u, u^{\sharp}$ and $u^{\star}$ are equimeasurable, in the sense that
$$
\mathfrak{m}(\{|u|>t\})=\mathscr{L}^1(\{u^{\sharp}>t\})=\mathfrak{m}_{K, N}(\{u^{\star}>t\})
$$
for all $t>0$. \\
$(b)$ If $u \in L^p(\Omega)$ for some $1 \leq p \leq \infty$, then $u^{\sharp} \in L^p([0, \mathfrak{m}(\Omega)], \mathscr{L}^1)$ and $u^{\star} \in L^p([0, r], \mathfrak{m}_{K, N})$. The converse implications also hold. In that case, moreover,
$$
\|u\|_{L^p(\Omega,\mathfrak{m})}=\|u^{\sharp}\|_{L^p([0,\mathfrak{m}(\Omega)], \mathscr{L}^1)}=\|u^{\star}\|_{L^p([0, r], \mathfrak{m}_{K, N})} .
$$
$(c)$ If $u\in L^1(\Omega)$ and let $E\subset \Omega$ be $\mathfrak{m}$-measurable, then 
\begin{equation}
\int_E u\mathrm{d}\mathfrak{m}\leq \int_0^{\mathfrak{m}(E)} u^\sharp(s) \mathrm{d}s.
\end{equation}
Moreover, if $u$ is non-negative, equality holds if and only if $(u|_E)^\sharp=(u^\sharp)|_{[0,\mathfrak{m}(E)]}$.\\
$(d)$ If $\phi$ is a strictly increasing and continuous function, then $(\phi\circ |u|)^\sharp=\phi\circ u^\sharp$ and $(\phi\circ |u|)^\star=\phi\circ u^\star$.\\
$(e)$ If $\psi:J_{K,N}\rightarrow [0,\infty)$ be a nonincreasing and non-negative function. Then $\psi^\star=\psi$ for all $x\in J_{K,N}\backslash L$, where $L$ is a countable set.
\end{proposition}

\subsection{Poisson problem on the model space}
As already said in the introduction, one of the main results of this paper is a comparison between the Schwarz symmetrization of the solution of an elliptic problem on $(\mathrm{X},\mathrm{d},\mathfrak{m})$ and the solution of a symmetrized problem on the model space. We discuss here the "model problem" on $(J_{K, N},\mathrm{d}_{eu},\mathfrak{m}_{K,N})$. Let $I =[0, r_1)$ with $0<r_1<\pi \sqrt{(N-1)/K}$.
\begin{proposition}[Sobolev space on $I$]\label{defsobo1}
Let $p\in (1,\infty)$. \\
$(i)$ If $v\in W^{1,p}(I, \mathfrak{m}_{K, N})$, then $v^\prime$ exists and $|\nabla v|_{\mathfrak{m}_{K,N}}=|v^\prime|$ $\mathfrak{m}_{K,N}$-a.e., where $v^{\prime}$ is the distributional derivative defined by
\begin{equation}\label{11}
\int_I v \phi^{\prime} \mathrm{d} \mathscr{L}^1=-\int_I v^{\prime} \phi \mathrm{d} \mathscr{L}^1 \quad \forall \phi \in C_c^{\infty}((0,r_1)).
\end{equation}
$(ii)$ It holds that
$$
W^{1,p}(I, \mathfrak{m}_{K, N}) =\big\{v \in L^p(I, \mathfrak{m}_{K, N}) : v^{\prime}\text{ exists and belongs to } L^p(I, \mathfrak{m}_{K, N})\big\}.
$$
\begin{proof}
To prove (i), given $v\in W^{1,p}(I, \mathfrak{m}_{K, N})$, since $\mathfrak{m}_{K,N}:=h_{K,N}\mathscr{L}^1\llcorner J_{K, N}$ with $h_{K,N}$ locally bounded away from $0$ out of the two end points of $J_{K,N}$, it follows that $v\in W^{1,p}((\epsilon,r_1), \mathscr{L}^1\llcorner J_{K, N})$ and $|\nabla v|_{\mathscr{L}^1\llcorner J_{K, N}}\leq |\nabla v|_{\mathfrak{m}_{K,N}}$ $\mathscr{L}^1$-a.e. on $(\epsilon,r_1)$ for any $\epsilon\in (0,r_1)$. By the very definition of $W^{1,p}((\epsilon,r_1), \mathscr{L}^1\llcorner J_{K, N})$ and the arbitrariness of $\epsilon\in (0,r_1)$, we deduce that $v^\prime$ exists, $|v^\prime|=|\nabla v|_{\mathscr{L}^1\llcorner J_{K, N}}\leq |\nabla v|_{\mathfrak{m}_{K,N}}$ $\mathfrak{m}_{K,N}$-a.e. on $[0,r_1)$ and thus $v^\prime\in L^p(I, \mathfrak{m}_{K, N})$. On the other hand, the facts that $v\in W^{1,p}((\epsilon,r_1), \mathscr{L}^1\llcorner J_{K, N})$ for any $\epsilon\in (0,r_1)$ and $h_{K,N}$ is bounded on $J_{K,N}$ imply that $|\nabla v|_{\mathfrak{m}_{K,N}}\leq |v^\prime|$ $\mathfrak{m}_{K,N}$-a.e..

For part (ii), $\subset$ is a direct consequence of (i). The other direction is also easy. Indeed, given $v\in L^p(I, \mathfrak{m}_{K, N})$ with $v^\prime \in L^p(I, \mathfrak{m}_{K, N})$, where $v^\prime$ is defined as in $(\ref{11})$. Obviously, $v$ belongs to $W^{1,p}((\epsilon,r_1), \mathscr{L}^1\llcorner J_{K, N})$ for any $\epsilon\in (0,r_1)$. Therefore, $v\in W^{1,p}((\epsilon,r_1), \mathfrak{m}_{K, N})$ for any $\epsilon\in (0,r_1)$ because $h_{K,N}$ is bounded on $J_{K,N}$ and $|\nabla v|_{\mathfrak{m}_{K,N}}\leq |v^\prime|$ $\mathfrak{m}_{K,N}$-a.e.. Hence we can conclude that $v\in W^{1,p}(I,  \mathfrak{m}_{K, N})$. 
\end{proof}
\end{proposition}

\begin{remark}
Let $p,q\in (1,\infty)$  be conjugate exponents. We consider the metric measure space $(J_{K, N}, \mathrm{d}_{eu}, \mathfrak{m}_{K, N})$ and open set $[0,r_1)$. Then $(J_{K, N}, \mathrm{d}_{eu}, \mathfrak{m}_{K, N})$ is $q$-infinitesimally strictly convex, for all $p\in (1,\infty)$ and 
$$
\Delta_{p,K, N} w(\phi):=\mathscr{L}_p w(\phi)=-\int_I |w^{\prime}|^{p-2}w^\prime \phi^{\prime} \mathrm{d}\mathfrak{m}_{K, N}.
$$	

The corresponding Dirichlet problem: let $f \in L^q(I, \mathfrak{m}_{K, N})$. We say that $w \in W^{1,p}(I, \mathfrak{m}_{K, N})$ is a weak solution to
\begin{equation}\label{20}
\begin{cases}
-\Delta_{p,K, N} w & =f \quad \text { in } I \\
w(r_1) & =0
\end{cases}
\end{equation}
if:\\
$(i)$ for any $\phi \in W_0^{1,p}(I,  \mathfrak{m}_{K, N})$,
\begin{equation}\label{we}
\int_{I} |w^{\prime}|^{p-2}w^\prime \phi^{\prime} \mathrm{d}\mathfrak{m}_{K, N}=\int_{I} f \phi \mathrm{d}\mathfrak{m}_{K, N};
\end{equation}
$(ii)$ Boundary condition: $w \in W_0^{1,p}(I,  \mathfrak{m}_{K, N})$.
\end{remark}

In the next proposition, we give an explicit solution to the problem in $(\ref{20})$.
\begin{proposition}[An explicit solution]\label{prop1}
Let $p,q\in (1,\infty)$  be conjugate exponents. Let $f \in L^q(I, \mathfrak{m}_{K, N})$ be a nonnegative function, the problem in $(\ref{20})$ admits a unique weak solution $w \in W_0^{1,p}\cap L^\infty(I,  \mathfrak{m}_{K, N})$, which can be represented as
\begin{equation}\label{21}
w(\rho)=\int_\rho^{r_1} \bigg(\frac{1}{h_{K, N}(r)} \int_0^r f \mathrm{d}\mathfrak{m}_{K, N}\bigg)^{\frac{1}{p-1}}\mathrm{d}r, \quad \forall \rho \in[0, r_1],
\end{equation}
or equivalently as
\begin{equation}\label{22}
w(\rho)=\int_{H_{K, N}(\rho)}^{H_{K, N}(r_1)}\frac{1}{\mathcal{I}_{K, N}(\sigma)} \bigg(\frac{1}{\mathcal{I}_{K, N}(\sigma)}\int_0^\sigma f\circ H_{K,N}^{-1}(t)\mathrm{d}t\bigg)^{\frac{1}{p-1}} \mathrm{d} \sigma, \quad \forall \rho \in[0, r_1].
\end{equation}
\begin{proof}
\textbf{Step 1:} we show that the two expressions $(\ref{21})$, $(\ref{22})$ are actually equivalent. Indeed, using the fact $\mathcal{I}_{K, N}=h_{K, N} \circ H_{K, N}^{-1}$ by Lemma $\ref{lem2.12}$, we deduce that
\begin{align}
\int_\rho^{r_1} \bigg(\frac{1}{h_{K, N}(r)} \int_0^r f \mathrm{d}\mathfrak{m}_{K, N}\bigg)^{\frac{1}{p-1}}\mathrm{d}r & =\int_\rho^{r_1} \frac{1}{h_{K, N}(r)}\bigg(\frac{1}{h_{K, N}(r)} \int_0^r f \mathrm{d}\mathfrak{m}_{K, N}\bigg)^{\frac{1}{p-1}} h_{K, N}(r) \mathrm{d} r\notag \\
& =\int_{H_{K, N}(\rho)}^{H_{K, N}(r_1)}\frac{1}{\mathcal{I}_{K, N}(\sigma)} \bigg(\frac{1}{\mathcal{I}_{K, N}(\sigma)}\int_0^{H_{K,N}^{-1}(\sigma)} f(s)\mathrm{d}\mathfrak{m}_{K, N}(s)\bigg)^{\frac{1}{p-1}} \mathrm{d} \sigma\notag\\
&=\int_{H_{K, N}(\rho)}^{H_{K, N}(r_1)}\frac{1}{\mathcal{I}_{K, N}(\sigma)} \bigg(\frac{1}{\mathcal{I}_{K, N}(\sigma)}\int_0^\sigma f\circ H_{K,N}^{-1}(t)\mathrm{d}t\bigg)^{\frac{1}{p-1}} \mathrm{d} \sigma,\label{1}
\end{align}
where we have used the change of variables $\sigma=H_{K, N}(r)$ in the external integral, and then the change of variables $t=H_{K, N}(s)$ in the internal integral. \\
\textbf{Step 2:} we show that a weak solution to $(\ref{20})$ must coincide with the function in  $(\ref{21})$; in particular, we have the uniqueness of weak solutions to $(\ref{20})$. Let $\bar{w} \in W^{1,p}(I, \mathfrak{m}_{K, N})$ be a weak solution to $(\ref{20})$. First, we prove that the weak derivative as in $(\ref{11})$ of $\bar{w}$ coincides $\mathfrak{m}_{K, N}$-a.e. with the function $-g^{\frac{1}{p-1}}$, where
\begin{equation*}
g(x)=\frac{1}{h_{K, N}(x)} \int_0^x f \mathrm{d}\mathfrak{m}_{K, N},\quad x\in [0,r_1].
\end{equation*} Indeed, for any test function $\phi \in {C}_c^{\infty}(I)$ one has $(x,s)\mapsto\chi_{[0, x]}(s) f(s)\phi^{\prime}(x)h_{K, N}(s)\in L^1([0,r_1]^2,(\mathscr{L}^1)^2)$, and thus by the Fubini Theorem:
\begin{align}
\int_I g(x) \phi^{\prime}(x) \mathrm{d} \mathfrak{m}_{K, N}(x) & =\int_0^{r_1}\bigg(\int_0^{r_1} \chi_{[0, x]}(s) f(s) \frac{\phi^{\prime}(x)}{h_{K, N}(x)} \mathrm{d}\mathfrak{m}_{K, N}(s)\bigg) \mathrm{d}\mathfrak{m}_{K, N}(x) \notag\\
& =\int_0^{r_1} f(s)\bigg(\int_s^{r_1} \phi^{\prime}(x) \mathrm{d}x\bigg) \mathrm{d} \mathfrak{m}_{K, N}(s) \notag\\
& =-\int_0^{r_1} f\phi\mathrm{d} \mathfrak{m}_{K, N}\label{23}.
\end{align}
Thus, since $\bar{w}$ is a weak solution to $(\ref{20})$, for any $\phi \in {C}_c^{\infty}(I)$
$$
\int_I[g(x)+|\bar{w}^{\prime}|^{p-2}(x)\bar{w}^{\prime}(x)] h_{K, N}(x) \phi^{\prime}(x) \mathrm{d} x=0 .
$$
By a classical result (see for example \cite[Lemma 8.1]{16}), there exists a constant $C \in \mathbb{R}$ such that $|\bar{w}^{\prime}|^{p-2}(x)\bar{w}^{\prime}(x) h_{K, N}(x)+g(x) h_{K, N}(x)=C$ for $\mathscr{L}^1$-a.e. $x \in I$. This however implies that for any $\phi \in {C}_c^{\infty}(I)$,
$$
0=C \int_0^{r_1} \phi^{\prime}(x) \mathrm{d}x=-C \phi(0)
$$
hence $C=0$. Moreover, we have $\bar{w}^{\prime}=-g^{\frac{1}{p-1}}$ (recall that $g\geq 0$). Since 
$\bar{w} \in W^{1,p}_0(I,  \mathfrak{m}_{K, N})$ and $h_{K,N}$ is bounded by below on $[\epsilon,r_1]$, thus in particular it belongs to $W^{1,p}((\epsilon, r_1),  \mathscr{L}^1)$ for any $\epsilon>0$. Thus, by well known results about Sobolev functions on intervals (see \cite[Theorem 8.2]{16}) and trace theorem (see \cite{17}), we obtain that 
\begin{equation*}
\bar{w}(\rho)=-\int_\rho^{r_1}\bar{w}^{\prime}(s)\mathrm{d}s=\int_\rho^{r_1}g^{\frac{1}{p-1}}(s)\mathrm{d}s=w(\rho),\text{ for $\mathfrak{m}_{K,N}$-a.e. }\rho.
\end{equation*}
\textbf{Step 3.} we show that $w$ defined by $(\ref{21})$ is actually a solution to $(\ref{20})$. Since the integrand is continuous on $(0, r_1]$, $w$ is $C^1$ on $(0, r_1]$ and continuous on $[0,r_1]$(with $w(r_1)=0$).

By straightforward computations, we show that $w$ and $w^{\prime}$ are $L^p(I, \mathfrak{m}_{K, N})$ functions. Indeed, by H\"{o}lder inequality we have that
\begin{equation}\label{25}
\int_0^r|f(s)| \mathrm{d}\mathfrak{m}_{K, N}(s) \leq\|f\|_{L^q(I, \mathfrak{m}_{K, N})} H_{K, N}(r)^{\frac{1}{p}},\text{ for any }r\in (0,r_1],
\end{equation}
thus, by Lemma $\ref{lem2.11}$, for any $\rho\in (0,r_1)$,
\begin{align*}
|w(\rho)| & \leq\|f\|_{L^q}^{\frac{1}{p-1}} \int_\rho^{r_1} \frac{H_{K, N}(r)^{\frac{1}{p(p-1)}}}{h_{K, N}(r)^{\frac{1}{p-1}}} \mathrm{d} r\leq C_1\|f\|_{L^q}^{\frac{1}{p-1}} \int_\rho^{r_1} \frac{r^{\frac{N}{p(p-1)}}}{r^{\frac{N-1}{p-1}}} \mathrm{d} r \\
&=\begin{cases}
C_2\|f\|_{L^q}^{\frac{1}{p-1}}\bigg(r_1^{\frac{N}{p(p-1)}-\frac{N-1}{p-1}+1}-\rho^{\frac{N}{p(p-1)}-\frac{N-1}{p-1}+1}\bigg),&\quad \frac{N}{p(p-1)}-\frac{N-1}{p-1}> -1,\\
C_1 \|f\|_{L^q}^{\frac{1}{p-1}} (\log r_1 -\log \rho), &\quad \frac{N}{p(p-1)}-\frac{N-1}{p-1}= -1,\\
C_2\|f\|_{L^q}^{\frac{1}{p-1}}\bigg(\rho^{\frac{N}{p(p-1)}-\frac{N-1}{p-1}+1}-r_1^{\frac{N}{p(p-1)}-\frac{N-1}{p-1}+1}\bigg),&\quad \frac{N}{p(p-1)}-\frac{N-1}{p-1}<-1.
\end{cases}
\end{align*}
Hence, $w$ is bounded on $[0,r_1]$. Consequently, by Lemma $\ref{lem2.11}$ again, if $\frac{N}{p(p-1)}-\frac{N-1}{p-1}\neq -1$:
\begin{equation}\label{es1}
\int_0^{r_1}|w|^p \mathrm{d}\mathfrak{m}_{K, N} \leq C_3\|f\|_{L^q}^{\frac{p}{p-1}} \int_0^{r_1}(r_1^{\frac{N}{p-1}-\frac{(N-1)p}{p-1}+p}+\rho^{\frac{N}{p-1}-\frac{(N-1)p}{p-1}+p})\rho^{N-1} \mathrm{d} \rho<\infty,
\end{equation}
if instead $\frac{N}{p(p-1)}-\frac{N-1}{p-1}=-1$,
\begin{equation}\label{es3}
\int_0^{r_1}|w|^p \mathrm{d}\mathfrak{m}_{K, N} \leq C_3 \|f\|_{L^q}^{\frac{p}{p-1}} \int_0^{r_1}(\log r_1-\log \rho)^p \rho^{N-1} \mathrm{d} \rho<\infty.
\end{equation}
Moreover, exploiting $(\ref{25})$ and by Lemma $\ref{lem2.11}$ again,
\begin{equation}\label{es2}
\int_0^{r_1}|w^{\prime}|^p \mathrm{d} \mathfrak{m}_{K, N} \leq\|f\|_{L^q}^{\frac{p}{p-1}} \int_0^{r_1} \frac{H_{K, N}(r)^{\frac{1}{p-1}}}{h_{K, N}(r)^{\frac{1}{p-1}}} \mathrm{d} r \leq C_4\|f\|_{L^q}^{\frac{p}{p-1}} \int_0^{r_1} \frac{r^{\frac{N}{p-1}}}{r^{\frac{N-1}{p-1}}} \mathrm{d} r<\infty.
\end{equation}
Note that all constants $C_1,C_2,C_3$ and $C_4$ depend only on $r_1,K,N$ and $p$.

It is well known that $C^\infty(\overline{\Omega})$ is dense in $W^{1,p}(\Omega)$ for any $p\in [1,\infty)$ and $\Omega\subset \mathbb{R}^n$ with $C^1$ boundary, see \cite{17}. The proofs contained there still work in the case of $W^{1,p}(I,  \mathfrak{m}_{K, N})$ with very few straightforward modifications because $h_{K,N}$ is bounded. By the fact that $w(r_1)=0$, we can conclude that $w\in W_0^{1,p}(I,\mathfrak{m}_{K,N})$.

Finally, by tracing back the identity in $(\ref{23})$, the very same argument shows that $w$ satisfies $(\ref{we})$ and thus it is a weak solution to $(\ref{20})$.
\end{proof}
\end{proposition}

\section{The Talenti-type comparison theorem}
\subsection{Proof of the Talenti-type comparison theorem}
Throughout this section, let $(\mathrm{X}, \mathrm{d}, \mathfrak{m})$ be a $\operatorname{RCD}(K, N)$ space, $K>0$, $N \in(1, \infty)$ and $\Omega \subset \mathrm{X}$ an open set. Let $p,q\in (1,\infty)$  be conjugate exponents.

Before passing to the proof of the main comparison theorem $\ref{thm}$, we establish few auxiliary lemmas.
\begin{lemma}\cite[Lemma 3.5]{11}\label{lem3.5}
Let $u \in L^p(\Omega), f\in L^q(\Omega)$. Define
$$
F(t) = \int_{\{u>t\}}(u-t) f \mathrm{d}\mathfrak{m}, \quad \forall t \in \mathbb{R} .
$$
Then $F$ is differentiable out of a countable set $C \subset \mathbb{R}$, and
$$
F^{\prime}(t)=-\int_{\{u>t\}} f \mathrm{d}\mathfrak{m}, \quad \forall t \in \mathbb{R} \backslash C .
$$	
\end{lemma}

\begin{lemma}\label{4}
Let $f \in L^q(\Omega)$. Let $u \in W_0^{1,p}(\Omega)$ be a weak solution of $-\mathscr{L}_p u=f$. Then for $\mathscr{L}^1$-a.e. $t>0$, it holds:
$$
-\frac{\mathrm{d}}{\mathrm{d} t} \int_{\{|u|>t\}}|\nabla u|\mathrm{d} \mathfrak{m} \leq  (-\mu^{\prime}(t))^{1/q} \bigg(\int_{\{|u|>t\}}|f|\mathrm{d} \mathfrak{m}\bigg)^{1/p},
$$
where $\mu$ is the distribution function of $u$.
\begin{proof}
Let $t>0$ be fixed, and consider the following function: 
\begin{equation}\label{30}
v_t =\varphi_t\circ u, \text{ where }\varphi_t(x)=(x-t)^{+}-(x+t)^{-}\text{ for any }x\in \mathbb{R}.
\end{equation}
Since $u\in W_0^{1,p}(\Omega)$ and $\varphi_t(0)=0$, by Proposition  $\ref{propp}$(ii), $v_t$ still belongs to the space $W_0^{1,p}(\Omega)$, and thus it can be used as a test function in $(\ref{eq})$ to obtain
$$
-\mathscr{L}_p u(v_t)=\int_{\Omega} f v_t\mathrm{d} \mathfrak{m}=\int_{\{u>t\}}(u-t) f \mathrm{d} \mathfrak{m}-\int_{\{-u>t\}}(-u-t) f \mathrm{d} \mathfrak{m}.
$$
By applying Lemma $\ref{lem3.5}$ we obtain that, for $\mathscr{L}^1$-a.e. $t>0$, $t \mapsto \mathscr{L}_p u(v_t)$ is differentiable with
\begin{equation}\label{31}
\frac{\mathrm{d}}{\mathrm{d} t} \mathscr{L}_p u(v_t)=\int_{\{u>t\}} f \mathrm{d} \mathfrak{m}-\int_{\{u<-t\}} f \mathrm{d} \mathfrak{m}\leq \int_{\{|u|>t\}}|f| \mathrm{d} \mathfrak{m} .
\end{equation}
For fixed $t>0$ and $h>0$, we can explicitly write
\begin{align}
v_{t+h}-v_t
= \begin{cases}h & \text { if } u<-t-h \\
-(u+t) & \text { if }-t-h \leq u<-t \\
0 & \text { if }|u| \leq t \\
-(u-t) & \text { if } t<u \leq t+h \\
-h & \text { if } u>t+h.\label{33}
\end{cases}
\end{align}
Notice that, by linearity of $\mathscr{L}_p u$ and locality $(\ref{2.15})$ of minimal weak upper gradient, it follows from $(\ref{33})$ that 
\begin{align*}
 \frac{\mathscr{L}_p u(v_{t+h})-\mathscr{L}_p u(v_t)}{h}=\frac{\mathscr{L}_p u(v_{t+h}-v_t)}{h}=\frac{1}{h}\int_{\{t<|u| \leq t+h\}}|\nabla u|^p \mathrm{d} \mathfrak{m}.
\end{align*}
Consequently, for all $t>0$ and $h>0$, it holds that 
\begin{align*}
\frac{1}{h} \int_{\{t<|u| \leq t+h\}}|\nabla u| \mathrm{d} \mathfrak{m} & \leq\bigg(\frac{1}{h} \int_{\{t<|u| \leq t+h\}}|\nabla u|^p \mathrm{d} \mathfrak{m}\bigg)^{1/p}\bigg(\frac{\mathfrak{m}(\{t<|u| \leq t+h\})}{h}\bigg)^{1/q}\\
& =\bigg(\frac{\mathscr{L}_p u(v_{t+h})-\mathscr{L}_p u(v_t)}{h}\bigg)^{1/p}\bigg(-\frac{\mu(t+h)-\mu(t)}{h}\bigg)^{1/q}.
\end{align*}
Hence, letting $h \rightarrow 0$ and using $(\ref{31})$ we get exactly the desired result.
\end{proof}
\end{lemma}

Finally, we turn to the proof of Theorem $\ref{thm}$. The proof is along the lines of the approach in \cite{11}.
\begin{proof}[Proof of Theorem $\ref{thm}$]
Without loss of generality, we can assume that $u\neq 0$.
(i) Combining coarea formula $(\ref{coarea})$, Lemma $\ref{4}$, Propositions $\ref{3}(c)$, $\ref{levy}$ and Chain rule $(\ref{2.16})$ of minimal weak upper gradient, we obtain the following chain of inequalities:
\begin{align}
\mathcal{I}_{K, N}(\mu(t)) &\leq \mathrm{Per}(\{|u|>t\})\leq-\frac{\mathrm{d}}{\mathrm{d} t} \int_{\{|u|>t\}}|\nabla |u|| \mathrm{d}\mathfrak{m} = -\frac{\mathrm{d}}{\mathrm{d} t} \int_{\{|u|>t\}}|\nabla u| \mathrm{d}\mathfrak{m} \notag\\
& \leq(-\mu^{\prime}(t))^{1/q} \bigg(\int_{\{|u|>t\}}|f|\mathrm{d} \mathfrak{m}\bigg)^{1/p}\leq(-\mu^{\prime}(t))^{1/q} \bigg(\int_0^{\mu(t)} f^{\sharp}(s) \mathrm{d} s\bigg)^{1/p}\label{41}
\end{align}
for $\mathscr{L}^1$-a.e. $t>0$, which can be rewritten as
\begin{equation}\label{42}
1 \leq-\frac{\mu^{\prime}(t)}{ \mathcal{I}_{K, N}(\mu(t))^{\frac{p}{p-1}}} \bigg(\int_0^{\mu(t)} f^{\sharp}(s) \mathrm{d} s \bigg)^{\frac{1}{p-1}}=\frac{-\mu^{\prime}(t)}{\mathcal{I}_{K, N}(\mu(t))} \bigg(\frac{1}{\mathcal{I}_{K, N}(\mu(t))}\int_0^{\mu(t)} f^{\sharp}(s) \mathrm{d} s\bigg)^{\frac{1}{p-1}}
\end{equation}
for $\mathscr{L}^1$-a.e. $t \in(0, M)$, where $M=\operatorname{ess} \sup |u|\in (0,\infty]$. For simplicity, let
$$
F(\xi) = \int_0^{\xi} f^{\sharp}(s) \mathrm{d} s \text{ for any }\xi>0.
$$
Let now $0 \leq \tau^{\prime}<\tau \leq M$. Integrating $(\ref{42})$ from $\tau^{\prime}$ to $\tau$ we get
$$
\tau-\tau^{\prime} \leq  \int_{\tau^{\prime}}^\tau \frac{-\mu^{\prime}(t)}{\mathcal{I}_{K, N}(\mu(t))} \bigg(\frac{F(\mu(t))}{\mathcal{I}_{K, N}(\mu(t))}\bigg)^{\frac{1}{p-1}} \mathrm{d} t, \quad 0 \leq \tau^{\prime}<\tau \leq M .
$$
Using the change of variables $\xi=\mu(t)$ on the intervals where $\mu$ is continuous(see \cite[p.108]{40}), and observing that the integrand is nonnegative, we obtain
\begin{equation}\label{14}
\tau-\tau^{\prime} \leq \int_{\mu(\tau)}^{\mu(\tau^{\prime})}\frac{1}{\mathcal{I}_{K, N}(\xi)} \bigg(\frac{F(\xi)}{\mathcal{I}_{K, N}(\xi)}\bigg)^{\frac{1}{p-1}} \mathrm{d} \xi, \quad 0 \leq \tau^{\prime}<\tau \leq M.
\end{equation}
Let us fix $s \in(0, \mu(0))$ and let $\eta>0$ be a small enough parameter such that $u^{\sharp}(s)>\eta$; consider $\tau^{\prime}=0$ and $\tau=u^{\sharp}(s)-\eta$. Notice that, since $u^{\sharp}(s)$ is the infimum of the $\tilde{\tau}$ such that $\mu(\tilde{\tau})<s$, we have that $\mu(\tau) \geq s$. Using again the nonnegativity of the integrand, we obtain that
$$
u^{\sharp}(s)-\eta \leq \int_s^{\mu(0)} \frac{1}{\mathcal{I}_{K, N}(\xi)} \bigg(\frac{F(\xi)}{\mathcal{I}_{K, N}(\xi)}\bigg)^{\frac{1}{p-1}} \mathrm{d} \xi, \quad \forall s \in(0, \mu(0)) .
$$
Letting $\eta \downarrow 0$, enlarging the integration interval and notice that on $(\mu(0), \mathfrak{m}(\Omega)]$ the function $u^{\sharp}$ vanishes and $u^{\sharp}$ is left continuous on $[0,\mathfrak{m}(\Omega)]$ and right-continuous at $0$, we get that 
\begin{equation}\label{9}
u^{\sharp}(s) \leq
\int_s^{\mathfrak{m}(\Omega)}\frac{1}{\mathcal{I}_{K, N}(\xi)} \bigg(\frac{1}{\mathcal{I}_{K, N}(\xi)}\int_0^\xi f^{\sharp}(t)\mathrm{d}t\bigg)^{\frac{1}{p-1}} \mathrm{d} \xi
\quad \forall s \in[0, \mathfrak{m}(\Omega)].
\end{equation}
Finally, by the definition of the symmetrized function $u^{\star}=u^{\sharp} \circ H_{K, N}$, we obtain
$$
u^{\star}(x) \leq\int_{H_{K, N}(x)}^{\mathfrak{m}(\Omega)}\frac{1}{\mathcal{I}_{K, N}(\xi)} \bigg(\frac{1}{\mathcal{I}_{K, N}(\xi)}\int_0^\xi f^{\star}\circ H_{K,N}^{-1}(t)\mathrm{d}t\bigg)^{\frac{1}{p-1}} \mathrm{d} \xi
\quad \forall x \in[0,r_v].
$$
Now we can recognize that the right hand side coincides with the characterization of $w$ we obtained in $(\ref{22})$, since $r_v$ was chosen so that $H_{K, N}(r_v)=\mathfrak{m}(\Omega)$. In particular, $u^\star(0)=\operatorname{ess} \sup_{\Omega}|u|\leq w(0)<\infty$. 

Finally, if $(\mathrm{X},\mathrm{d},\mathfrak{m})=(J_{K,N},\mathrm{d}_{eu},\mathfrak{m}_{K,N})$, and assume that $f:J_{K,N}\rightarrow [0,\infty)$ is a nonincreasing and non-negative function, then, by Proposition $\ref{3}(e)$, we obtain $f^\star=f$ $\mathfrak{m}_{K,N}$-a.e. and $w^\star=w$ on $J_{K,N}$ since $w$ is continuous. That is to say, the Talenti comparison inequality is sharp. We complete the proof of (i).

(ii) We start by noticing that for $1\leq r\leq p$,
$$
\int_{\Omega}|\nabla u|^r \mathrm{d}\mathfrak{m}=\int_{\{|u|>0\}}|\nabla u|^r \mathrm{d}\mathfrak{m}
$$
since $|\nabla u|=0$ $\mathfrak{m}$-a.e. on $\{u=0\}$, by $(\ref{2.14})$. Let $M:=\operatorname{ess} \sup_{\Omega}|u|\in (0,+\infty)$; fix $t,h>0$. Using the H\"{o}lder inequality one gets
\begin{equation}\label{45}
\frac{1}{h} \int_{\{t<|u| \leq t+h\}}|\nabla u|^r \mathrm{d} \mathfrak{m} \leq\bigg(\frac{1}{h} \int_{\{t<|u| \leq t+h\}}|\nabla u|^p \mathrm{d} \mathfrak{m}\bigg)^{\frac{r}{p}}\bigg(\frac{\mathfrak{m}(\{t<|u| \leq t+h\})}{h}\bigg)^{\frac{p-r}{p}}
\end{equation}
By the very same computations we already performed in Lemma $\ref{4}$, letting $h$ tend to zero in  $(\ref{45})$, we obtain
$$
-\frac{\mathrm{d}}{\mathrm{d} t} \int_{\{|u|>t\}}|\nabla u|^r \mathrm{d}\mathfrak{m} \leq\bigg( \int_{\{|u|>t\}} |f| \mathrm{d} \mathfrak{m}\bigg)^{\frac{r}{p}}(-\mu^{\prime}(t))^{\frac{p-r}{p}},
$$
for $\mathscr{L}^1$-a.e. $t$. Let us now adopt again the notation
$$
F(\xi) = \int_0^{\xi} f^{\sharp}(s) \mathrm{d} s.
$$
Using Propositions $\ref{3}(c)$ again, we get that
\begin{equation}\label{47}
-\frac{\mathrm{d}}{\mathrm{d} t} \int_{\{|u|>t\}}|\nabla u|^r \mathrm{d}\mathfrak{m} \leq F(\mu(t))^{\frac{r}{p}}(-\mu^{\prime}(t))^{\frac{p-r}{p}}
\end{equation}
for $\mathscr{L}^1$-a.e. $t$. In order to obtain a clean term $\mu^{\prime}(t)$ at the right hand side, we multiply both sides of $(\ref{47})$ with the respective sides of $(\ref{42})$ raised at the power $r/p$. This gives, for $\mathscr{L}^1$-a.e. $t \in(0, M)$:
$$
-\frac{\mathrm{d}}{\mathrm{d} t} \int_{\{|u|>t\}}|\nabla u|^r \mathrm{d}\mathfrak{m} \leq\bigg(\frac{F(\mu(t))}{ \mathcal{I}_{K, N}(\mu(t))}\bigg)^{\frac{r}{p-1}}(-\mu^{\prime}(t)) .
$$
Combining this and coarea formula $(\ref{coarea})$, and changing the variables as usual with $\xi=\mu(t)$, the following estimate holds that
\begin{equation}\label{48}
\int_{\Omega}|\nabla u|^r \mathrm{d}\mathfrak{m} \leq \int_0^{\mathfrak{m}(\Omega)}\bigg(\frac{F(\xi)}{ \mathcal{I}_{K, N}(\xi)}\bigg)^{\frac{r}{p-1}} \mathrm{d} \xi .
\end{equation}
Finally, we recall that $w$ has an explicit expression we can differentiate: by differentiating $(\ref{21})$ (with datum $f^{\star}$), we have, for all $\rho \in(0, r_v)$,
$$
w^{\prime}(\rho)=-\bigg(\frac{1}{h_{K, N}(\rho)} \int_0^{H_{K, N}(\rho)}  f^{\star}(H_{K, N}^{-1}(t)) \mathrm{d} t\bigg)^{\frac{1}{p-1}}=-\bigg(\frac{F(H_{K, N}(\rho))}{ h_{K, N}(\rho)}\bigg)^{\frac{1}{p-1}} .
$$
Thus, the following identity holds that
\begin{equation}\label{49}
\int_0^{r_v}|w^{\prime}|^r \mathrm{d}\mathfrak{m}_{K, N}=\int_0^{r_v}\bigg(\frac{F(H_{K, N}(\rho))}{ h_{K, N}(\rho)}\bigg)^{\frac{r}{p-1}} h_{K, N}(\rho) \mathrm{d} \rho=\int_0^{\mathfrak{m}(\Omega)}\bigg(\frac{F(\xi)}{ \mathcal{I}_{K, N}(\xi)}\bigg)^{\frac{r}{p-1}} \mathrm{d} \xi,
\end{equation}
where we have used the change of variables $\xi=H_{K, N}(\rho)$ and the fact that $\mathcal{I}_{K, N}(\xi)=$ $h_{K, N}(H_{K, N}^{-1}(\xi))$. Comparing with$(\ref{48})$, we obtain the desired $L^r$-gradient estimate.
\end{proof}

\subsection{Rigidity in the Talenti-type theorem}
Let $u \in W_0^{1,p}(\Omega)$ and $w \in W_0^{1,p}([0, r_v),  \mathfrak{m}_{K, N})$ be as in Theorem $\ref{thm}$. The next problem we want to approach is the equality case, that is, what we can say about the original metric measure space when $u^{\star}=w$; in fact, we will obtain that if the equality is attained at least at one point, then the metric measure space is forced to have a particular structure, namely it is a spherical suspension, which is Theorem $\ref{thm5.6}$. Before passing to its statement
we need some definitions and results.

Let $(B, \mathrm{d}_B, \mathfrak{m}_B)$ and $(F, \mathrm{d}_F, \mathfrak{m}_F)$ be geodesic metric measure spaces and $f: B \rightarrow[0, \infty)$ be a Lipschitz function. Let $\mathrm{d}$ be the pseudo-distance on $B \times F$ defined by
$$
\mathrm{d}((p, x),(q, y)) = \inf \{L(\gamma) \mid \gamma(0)=(p, x), \gamma(1)=(q, y)\},
$$
where, for any absolutely continuous curve $\gamma=(\gamma_B, \gamma_F):[0,1] \rightarrow B \times F$,
$$
L(\gamma) = \int_0^1(|\dot{\gamma}_B|^2+(f \circ \gamma_B)^2|\dot{\gamma}_F|^2)^{\frac{1}{2}} \mathrm{d} t .
$$
Given $N>1$, we define $B \times{ }_f^N F$ to be the metric measure space
$$
((B \times F) / \sim, \mathrm{d}, \mathfrak{m}),
$$
where $\sim$ is the equivalence relation associated to the pseudo-distance $\mathrm{d}$ and $\mathfrak{m} = f^N \mathfrak{m}_B \times \mathfrak{m}_F$.\\
Given $N>2$, we say that a $\operatorname{RCD}(N-1, N)$ space $(\mathrm{X}, \mathrm{d}, \mathfrak{m})$ is a \textbf{spherical suspension} if it is isomorphic to $[0, \pi] \times{ }_{\sin }^{N-1} \mathrm{Y}$ for a $\operatorname{RCD}^*(N-2, N-1)$ space $(\mathrm{Y}, \mathrm{d}_{\mathrm{Y}}, \mathfrak{m}_{\mathrm{Y}})$ with $\mathfrak{m}_{\mathrm{Y}}(\mathrm{Y})=1$.

\begin{theorem}[Rigidity for L\'{e}vy-Gromov]\cite{13}
Let $(\mathrm{X}, \mathrm{d}, \mathfrak{m})$ be a $\operatorname{RCD}(N-1, N)$ space with $N>2$. Assume there exists $\bar{v} \in(0,1)$ such that $\mathcal{I}_{(\mathrm{X}, \mathrm{d}, \mathfrak{m})}(\bar{v})= \mathcal{I}_{N-1, N}(\bar{v})$. Then $(\mathrm{X}, \mathrm{d}, \mathfrak{m})$ is a spherical suspension.	
\end{theorem}

\begin{theorem}[P\'{o}lya-Szeg\"{o} for $\operatorname{RCD}(N-1, N)$ spaces]\cite{18}
Let $(\mathrm{X}, \mathrm{d}, \mathfrak{m})$ be a $\operatorname{RCD}(N-1, N)$ space with $N >2$. Let $\Omega \subset \mathrm{X}$ be an open set with  $\mathfrak{m}(\Omega)=v \in(0,1)$ and let $r_v \in(0, \pi)$ such that $\mathfrak{m}_{N-1, N}([0, r_v])=v$. Then, for every $p \in(1, \infty)$, the following hold that\\
$(i)$ \underline{P\'{o}lya-Szeg\"{o} comparison}: for any $u \in W_0^{1, p}(\Omega)$, it holds that $u^{\star}(r_v)=0$ and
\begin{equation}\label{50}
\int_0^{r_v}|\nabla u^{\star}|^p \mathrm{d}\mathfrak{m}_{N-1, N} \leq \int_{\Omega}|\nabla u|^p \mathrm{d}\mathfrak{m} .
\end{equation}
$(ii)$ \underline{Rigidity}: if there exists $u \in W_0^{1,p}(\Omega)$ with $u \not \equiv 0$, achieving equality in $(\ref{50})$, then $(\mathrm{X}, \mathrm{d}, \mathfrak{m})$ is a spherical suspension.\\
$(iii)$ \underline{Rigidity for Lipschitz functions}: if there exists $u \in W_0^{1,p}(\Omega) \cap \operatorname{Lip}(\Omega)$ with $u \not \equiv 0$ and $|\nabla u| \neq 0$ $\mathfrak{m}$-a.e. in $\operatorname{supp}(u)$, achieving equality in $(\ref{50})$, then $(\mathrm{X}, \mathrm{d}, \mathfrak{m})$ is a spherical suspension and $u$ is radial: that is, $u$ is of the form $u=g(\mathrm{d}(\cdot, x_0))$, with $x_0$ being the tip of a spherical suspension structure of $\mathrm{X}$, and $g:[0, \pi] \rightarrow \mathbb{R}$ satisfying $|g|=u^{\star}$.	
\end{theorem}
We are ready to state the following rigidity result which build on top of the rigidity in the L\'{e}vy-Gromov and P\'{o}lya-Szeg\"{o} inequalities.
\begin{theorem}[Rigidity for Talenti]\label{thm5.6}
Let $(\mathrm{X}, \mathrm{d}, \mathfrak{m})$ be a $\operatorname{RCD}(N-1, N)$ space with $N>2$, and let $\Omega \subset \mathrm{X}$ be an open set with $\mathfrak{m}(\Omega)=v \in(0,1)$. Let $p,q\in (1,\infty)$  be conjugate exponents.

Let $f \in L^q(\Omega)$, with $f \neq 0$. Assume that $u \in W_0^{1,p}(\Omega)$ is a weak solution to $(\ref{eq})$. Let also $w \in W_0^{1,p}(I, \mathfrak{m}_{N-1, N})$ be a solution to $(\ref{eq00})$. Assume that $u^{\star}(\bar{x})=w(\bar{x})$ for a point $\bar{x} \in[0, r_v)$. Then:\\
$(i)$ $u^{\star}=w$ in the whole interval $[\bar{x}, r_v]$;\\
$(ii)$ $(\mathrm{X}, \mathrm{d}, \mathfrak{m})$ is a spherical suspension;\\
$(iii)$ if $\bar{x}=0, u \in \operatorname{Lip}(\Omega)$ and $|\nabla u| \neq 0$ $\mathfrak{m}$-a.e. in $\operatorname{supp}(u)$, then $u$ is radial.
\end{theorem}
Since the proof of the above Theorem is a direct adaptation of the proof in \cite[Theorem 4.4]{11}, we omit it.

\subsection{Almost rigidity in the Talenti-type theorem}\label{sec}
Our aim in this subsection is to prove the almost rigidity in the Talenti-type theorem and the continuity of the local Poisson problem. 

Throughout this section, let $p,q\in (1,\infty)$ be conjugate exponents, and the following assumption will be made:
\begin{assumption}\label{assu}
Spaces: $\{\mathscr{X}_i\}_{i \in \mathbb{N}_+}=\{(\mathrm{X}_i, \mathrm{d}_i, \mathfrak{m}_i)\}_{i \in \mathbb{N}_+}$ and $\mathscr{X}=(\mathrm{X}, \mathrm{d},  \mathfrak{m})$ will be $\operatorname{RCD}(N-1, N)$ spaces with $N>2$; \\
Convergence of spaces: we will assume that $\mathscr{X}_i$ converge in the $\mathrm{mGH}$ sense to $\mathscr{X}$ and the following conditions hold:\\
$(\mathrm{GH1})$ $\mathrm{X}_i$ and $\mathrm{X}$ are all contained in a common metric space $(\mathrm{Y}, \mathrm{d})$, with $\mathrm{d}_i=\mathrm{d}|_{\mathrm{X}_i \times \mathrm{X}_i}$;\\
$(\mathrm{GH2})$ The measures $\mathfrak{m}_i$ weakly converge to $\mathfrak{m}$:
$$
\lim_{i \rightarrow \infty} \int_{\mathrm{Y}} \phi \mathrm{d}\mathfrak{m}_i=\int_{\mathrm{Y}} \phi \mathrm{d}\mathfrak{m} \text { for all } \phi \in {C}_{\mathrm{b}}(\mathrm{Y}).
$$
Base points and Convergence: $x_i\in \mathrm{X}_i$ and $x\in \mathrm{X}$, $x_i\rightarrow x$, meaning that $\lim_{i\rightarrow \infty}\mathrm{d}(x_i,x)=0$.
\end{assumption}

Let $B_{R_i}(x_i)$ and $B_R(x)$ be metric balls in $\mathrm{X}_i$ and $\mathrm{X}$ respectively. Let $f_i \in L^p(B_{R_i}(x_i),\mathfrak{m}_i),f \in L^p(B_R(x),\mathfrak{m})$; by extending such functions to be $0$ out of the balls on which they are defined, we can equivalently assume $f_i \in L^p(\mathrm{X}_i,\mathfrak{m}_i)$ and $f \in L^p(\mathrm{X},\mathfrak{m})$; by the assumption that the spaces $\mathrm{X}_i$ and $\mathrm{X}$ are contained in $\mathrm{Y}$, up to a further extension we actually have $f_i \in L^p(\mathrm{Y},\mathfrak{m}_i)$ and $f \in L^p(\mathrm{Y}, \mathfrak{m})$. In addition, by abuse of notations, we also consider $B_{R_i}(x_i)$ and $B_R(x)$ as open balls in $\mathrm{Y}$.

The following lemma is elementary but often useful.
\begin{lemma}\label{lem4.13}
$(i)$ Let $R_i \rightarrow R$ with $R_i,R\in (0,\infty)$. Then $\mathfrak{m}_i(B_{R_i}(x_i))\rightarrow \mathfrak{m}(B_R(x))$.	\\
$(ii)$ Let $(0,\infty)\ni R_i \rightarrow R\in [0,\infty]$ be such that $\mathfrak{m}_i(B_{R_i}(x_i))=v \in(0,1)$ for all $i \in \mathbb{N}_+$. Then $R\in (0,\infty)$ and $\mathfrak{m}(B_R(x))=v$.	
\begin{proof}
For part (i), the fact that $\mathfrak{m}_i$ weakly converge to $\mathfrak{m}$ implies 
\begin{align*}
\mathfrak{m}(U)&\leq \varliminf_{i\rightarrow \infty}\mathfrak{m}_i(U)\text{ for any }U\subset \mathrm{Y}\text{ open.}\\
\mathfrak{m}(F)&\geq \varlimsup_{i\rightarrow \infty}\mathfrak{m}_i(F)\text{ for any }F\subset \mathrm{Y}\text{ closed.}
\end{align*}
Let $R>0$. Given any $\epsilon\in (0,R)$, the fact that $x_i\rightarrow x$ and $R_i\rightarrow R$ implies $\mathrm{d}(x_i,x)<\epsilon/2$ and $|R_i-R|<\epsilon/2$ for enough large $i$. Thus 
\begin{align*}
\varlimsup_{i \rightarrow \infty}\mathfrak{m}_i(B_{R_i}(x_i))&\leq \varlimsup_{i \rightarrow \infty}\mathfrak{m}_i(\bar{B}_{R_i}(x_i))\leq \varlimsup_{i \rightarrow \infty}\mathfrak{m}_i(\bar{B}_{R+\epsilon}(x))\leq \mathfrak{m}(\bar{B}_{R+\epsilon}(x)),\\
\varliminf_{i \rightarrow \infty}\mathfrak{m}_i(B_{R_i}(x_i))&\geq \varliminf_{i \rightarrow \infty}\mathfrak{m}_i(B_{R-\epsilon}(x))\geq \mathfrak{m}(B_{R-\epsilon}(x)).
\end{align*}
Combining these and the fact that $\mathfrak{m}(\partial B_R(x))$ for every $R>0$, and letting $\epsilon\rightarrow 0$, we obtain that
\begin{equation*}
\lim_{i \rightarrow \infty}\mathfrak{m}_i(B_{R_i}(x_i))=\mathfrak{m}(B_{R}(x)).
\end{equation*}
To prove (ii), we first note that if $R\in (0,\infty)$, (i) implies (ii). Now, we argue by contradiction: assume first that $R=\infty$. Given $\epsilon,M>0$, the fact that $x_i\rightarrow x$ and $R_i\rightarrow \infty$ implies $\mathrm{d}(x_i,x)<\epsilon$ and $R_i>M$ for enough large $i$. Thus 
\begin{align*}
v=\varliminf_{i \rightarrow \infty}\mathfrak{m}_i(B_{R_i}(x_i))&\geq \varliminf_{i \rightarrow \infty}\mathfrak{m}_i(B_{M-\epsilon}(x))\geq \mathfrak{m}(B_{M-\epsilon}(x)).
\end{align*}
Letting $M\rightarrow \infty$, this means that $v\geq 1$, and thus we get a contradiction. Similarly, we also can prove that $R=0$ is impossible, since $v>0$. We complete the proof.
\end{proof}
\end{lemma}

Now, we recall the suitable notions of weak and strong convergence for $f_i,f$.
\begin{definition}[$L^p$-Convergence]\label{def1}
Let $f_i \in L^p(\mathrm{X}_i,\mathfrak{m}_i)$ and $f \in L^p(\mathrm{X},\mathfrak{m})$. We say that:\\
$(a)$ $f_i$ $L^p$-weakly converge to $f$ if
\begin{equation}\label{6}
\lim_{i \rightarrow \infty} \int_{\mathrm{Y}} \phi f_i \mathrm{d}\mathfrak{m}_i=  \int_{\mathrm{Y}} \phi f \mathrm{d}\mathfrak{m} \text { for all } \phi \in {C}_{\mathrm{b}}(\mathrm{Y}) \text{ and } \sup_i\|f_i\|_{L^p}<\infty .
\end{equation}
$(b)$ $f_i$ $L^p$-strongly converge to $f$ if, in addition,
$$
\limsup_{i \rightarrow \infty}\|f_i\|_{L^p}\leq\|f\|_{L^p} .
$$
\end{definition}

\begin{remark}\cite[Theorem 5.4.4.]{23}\label{rmk3}
Let $f_i \in L^p(\mathrm{X}_i,\mathfrak{m}_i)$ and $f \in L^p(\mathrm{X},\mathfrak{m})$. $f_i$ $L^p$-strongly converge to $f$ if and only if for every $\zeta \in \mathrm{C}(\mathrm{Y} \times \mathbb{R})$ with $|\zeta(y, r)| \leq C(1+|r|^p),$ for some $C \geq 0$, we have
\begin{equation}\label{5}
\int \zeta(y, f_i(y)) \mathrm{d} \mathfrak{m}_i(y) \rightarrow \int \zeta(y, f(y)) \mathrm{d} \mathfrak{m}(y).
\end{equation}
In particular, $f_i$ $L^p$-strongly converge to $f$ implies that $f_i$ $L^{p^\prime}$-strongly converge to $f$, for any $p^\prime\in (1,p]$.
\end{remark}
\begin{proposition}[Lower semicontinuity and Compactness]\label{prop}
If $f_i \in L^{p}(\mathrm{X}_i, \mathfrak{m}_i)$ $L^{p}$-weakly converge to $f \in L^p(\mathrm{X}, \mathfrak{m})$, then
$$
\|f\|_{L^p} \leq \liminf _{i \rightarrow \infty}\|f_i\|_{L^{p}} .
$$
Moreover, any sequence $\{f_i\}_{i\in \mathbb{N}_+}$ with $f_i \in L^p(\mathrm{X}_i, \mathfrak{m}_i)$ such that $\sup_i\|f_i\|_{L^p}<\infty$ holds admits a $L^p$-weakly convergent subsequence.	
\end{proposition}
\begin{proposition}\cite{0.3,20}\label{prop0}
Let $f_i\in L^p(\mathrm{X}_i,\mathfrak{m}_i),f \in L^p(\mathrm{X},\mathfrak{m})$, $g_i\in L^q(\mathrm{X}_i,\mathfrak{m}_i),g \in L^q(\mathrm{X},\mathfrak{m})$. Assume that $f_i$ $L^p$-weakly converge to $f$ and that $g_i$ $L^q$-strongly converge to $g$. Then
$$
\lim_{i \rightarrow \infty} \int f_i g_i \mathrm{d} \mathfrak{m}_i=\int f g \mathrm{d} \mathfrak{m}.
$$
\end{proposition}
\begin{definition}[$W^{1,p}$-Convergence]
Let $f_i \in W^{1,p}(\mathrm{X}_i,\mathfrak{m}_i)$ and $f \in W^{1,p}(\mathrm{X},\mathfrak{m})$. We say that:\\
$(a)$ $f_i$ $W^{1,p}$-weakly converge to $f$ if $f_i$ $L^p$-weakly converge to $f$ and $\sup_i \int |\nabla f_i|^p\mathrm{d}\mathfrak{m}_i$ is finite.\\
$(b)$ $f_i$ $W^{1,p}$-strongly converge to $f$ if $f_i$ $L^p$-strongly converge to $f$ and
$$\int |\nabla f|^p\mathrm{d}\mathfrak{m}=  \lim_{i\rightarrow \infty}\int |\nabla f_i|^p\mathrm{d}\mathfrak{m}_i.$$		
\end{definition}
\begin{definition}[local $W^{1,p}$-Convergence]
Let $f_i \in W^{1,p}(B_{R_i}(x_i),\mathfrak{m}_i)$ and $f \in W^{1,p}(B_R(x),\mathfrak{m})$. We say that:\\
$(a)$ $f_i$ $W^{1,p}$-weakly converge to $f$ if $f_i$ $L^p$-weakly converge to $f$ and $\sup_i \int_{B_{R_i}(x_i)} |\nabla f_i|^p\mathrm{d}\mathfrak{m}_i$ is finite.\\
$(b)$ $f_i$ $W^{1,p}$-strongly converge to $f$ if $f_i$ $L^p$-strongly converge to $f$ and 
$$\int_{B_R(x)} |\nabla f|^p\mathrm{d}\mathfrak{m}=  \lim_{i\rightarrow \infty}\int_{B_{R_i}(x_i)} |\nabla f_i|^p\mathrm{d}\mathfrak{m}_i.$$	
\end{definition}

Let us recall the two global results in \cite{21} about $W^{1,p}$ convergence. Notice that the assumption in \cite[Proposition 7.5, Theorem 8.1]{21} is satisfies, by Remark $\ref{rmk2}$. 
\begin{lemma}\label{lem1}
$(i)$ $[$Compactness$]$. Assume that $f_i \in W^{1,p}(\mathrm{Y},\mathfrak{m}_i)$ satisfy
$\sup_i\|f_i\|_{W^{1,p}}<\infty$. Then $(f_i)$ has a subsequence $(f_{i_j})$ such that $f_{i_j}$ $L^p$-strongly converge to $f\in W^{1,p}(\mathrm{Y},\mathfrak{m})$. In particular, if $f_i$ $W^{1,p}$-weakly converge to $f$, then $f_i$ $L^{p}$-strongly converge to $f$.\\
$(ii)$ $[$Mosco convergence of $p$-Cheeger energies$]$. $(a)$ For every sequence $i\mapsto f_i\in L^p(\mathrm{Y}, \mathfrak{m}_i)$ $L^p$-weakly converging to some $f\in L^p(\mathrm{Y},\mathfrak{m})$, we have
$$
\mathrm{Ch}_p(f) \leq \liminf_{i \rightarrow \infty} \mathrm{Ch}_{p}^i(f_i) .
$$
$(b)$ For every $f\in L^p(\mathrm{Y},\mathfrak{m})$, there exists a sequence $i\mapsto f_i\in L^p(\mathrm{Y},\mathfrak{m}_i)$ $L^{p}$-strongly convergent to $f$ such that
$$
\mathrm{Ch}_p(f)\geq \limsup_{i \rightarrow \infty} \mathrm{Ch}_{p}^i(f_i).
$$	
\end{lemma}

\begin{lemma}[Sobolev inequality]
For $f\in W^{1,p}(B_R(x),\mathfrak{m})$, we have
\begin{equation}\label{2.7}
\bigg(\fint_{B_R(x)}\bigg|f-\fint_{B_R(x)} f \mathrm{d}\mathfrak{m}\bigg|^{p^*} \mathrm{d}\mathfrak{m}\bigg)^{1 / p^*} \leq C\bigg(\fint_{B_R(x)} |\nabla f|^p \mathrm{d}\mathfrak{m}\bigg)^{1 / p},
\end{equation}
where $p^*=p N /(N-p)$ if $N>p$, $p^*$ can be any power in $(p, \infty)$ if $N \in(1,p]$ and $C:=C(K, N, p^*, R)>0$. 
\begin{proof}
First, $(\ref{2.7})$ can be proved for locally Lipschitz functions starting from the local Poincar\'{e} inequality $(\ref{poin})$, applying then \cite[Theorem 5.1]{8}. By density, it extends to global $W^{1,p}$-functions. Since $W^{1,p}(B_R(x),\mathfrak{m})$ locally coincide with global $W^{1,p}$-functions, a simple monotone approximation then provides the result in the class $W^{1,p}(B_R(x),\mathfrak{m})$.
\end{proof}
\end{lemma}

\begin{lemma}\label{lem2.9}
If $f_i \in W^{1,p}(\mathrm{Y},\mathfrak{m}_i)$ $W^{1,p}$-weakly converge to $f\in W^{1,p}(\mathrm{Y},\mathfrak{m})$, then
\begin{equation}\label{0}
\liminf_{i \rightarrow \infty} \int g|\nabla f_i| \mathrm{d} \mathfrak{m}_i \geq \int g|\nabla f| \mathrm{d} \mathfrak{m}
\end{equation}
for any lower semicontinuous $g:\mathrm{Y}\rightarrow[0, \infty]$.
\begin{proof}
For $p=2$, $(\ref{0})$ has been proved in \cite[Lemma 5.8]{21}. Now we turn to the general case. Since any lower semicontinuous function is the monotone limit of a sequence of Lipschitz functions with bounded support, we can assume $g \in \operatorname{Lip}_{\mathrm{bs}}(\mathrm{Y})$.

Assume first that all $f_i,f$ are bounded in $L^{\infty}$ (and thus in $L^2$, recall that all $\mathfrak{m}_i,\mathfrak{m}$ are probability measures), and $f_i$ $L^2$-strongly converge to $f$. Let us fix $t>0$ and consider the functions $h_t^i f_i$, which are bounded in $L^\infty$ and equi-Lipschitz, (see \cite{29,30}). By \cite[Theorem 6.11]{19}, we obtain that $h_t^i f_i$ $L^{2}$-strongly converge to $h_t f \in W^{1,2}(\mathrm{Y},\mathfrak{m})$ and thus $h_t^i f_i$ $W^{1,2}$-weakly converge to $h_t f$. Hence, by \cite[Lemma 5.8]{21}, we have
\begin{equation}\label{00}
\liminf_{i \rightarrow \infty} \int g|\nabla h_t^i f_i| \mathrm{d} \mathfrak{m}_i \geq \int g|\nabla h_t f| \mathrm{d} \mathfrak{m}.
\end{equation}
Taking into account the inequality $|\nabla h_t^i f_i| \leq e^{-K t} h_t^i|\nabla f_i|$(which can be obtained by applying \cite[Corollary 3.5]{31} in combination with a simple approximate argument), we can estimate
\begin{align*}
\liminf_{i \rightarrow \infty} \int g|\nabla f_i| \mathrm{d} \mathfrak{m}_i & \geq \liminf_{i \rightarrow \infty} \int h_t^i g|\nabla f_i| \mathrm{d} \mathfrak{m}_i-\limsup_{i \rightarrow \infty} \int |h_t^i g-g| |\nabla f_i| \mathrm{d}\mathfrak{m}_i \\
& \geq e^{K t} \liminf_{i \rightarrow \infty} \int g|\nabla h_t^i f_i| \mathrm{d} \mathfrak{m}_i-C \limsup _{i \rightarrow \infty}\|h_t^i g-g\|_{L^q(\mathrm{Y}, \mathfrak{m}_i)},
\end{align*}
with $C=\sup_i(p\operatorname{Ch}_p^i(f_i))^{1 / p}$. The fact that
$\limsup_{t \rightarrow 0} \limsup_{i \rightarrow \infty} \int|h_t^i g-g|^q \mathrm{d}\mathfrak{m}_i=0$, (by \cite[Proposition 4.6]{21})
combining $(\ref{00})$ implies that 
\begin{equation*}
\liminf_{i \rightarrow \infty} \int g|\nabla f_i| \mathrm{d} \mathfrak{m}_i \geq \int g|\nabla f| \mathrm{d} \mathfrak{m}.
\end{equation*}

Eventually we consider the general case $f_i$. We consider the truncation $1$-Lipschitz functions $\varphi_N(t):=\max\{\min\{t,N\},-N\}$ and $f_i^N:=\varphi_N \circ f_i$. The fact that $f_i$ $W^{1,p}$-weakly converge to $f$ and Lemma $\ref{lem1}$(i) implies that $f_i$ $L^p$-strongly converge to $f$. Using Remark $\ref{rmk3}$, it is easy to check that $f_i^N$ $L^{p}$-strongly converge to $f^N:=\varphi_N \circ f$ and hence $f_i^N\mathfrak{m}_i$ weakly converge to $f^N\mathfrak{m}$. Combining this and $\|f_i^N\|_{L^\infty}\leq N$ for all $i$, one can prove that $f_i^N$ $L^2$-strongly converge to $f^N$. Hence, by what we previously proved, we obtain that
$$
\int g|\nabla f^N| \mathrm{d} \mathfrak{m} \leq \liminf_{i \rightarrow \infty} \int g|\nabla f_i^N| \mathrm{d} \mathfrak{m}_i  \leq \liminf_{i \rightarrow \infty}\int g|\nabla f_i| \mathrm{d} \mathfrak{m}_i.
$$
Letting $N \rightarrow \infty$, we complete the proof.
\end{proof}
\end{lemma}

\begin{proposition}\label{prop4}
Let $R_i,R\in (0,\infty)$ with $R_i\rightarrow R$ and $f_i \in W^{1,p}(B_{R_i}(x_i),\mathfrak{m}_i)$ with $\sup_i\|f_i\|_{W^{1,p}}<\infty$. \\
$(i)$ $[$Compactness$]$. There exist $f \in W^{1,p}(B_R(x),\mathfrak{m})$ and a subsequence $\{f_{i_j}\}$ such that $f_{i_j}$ $L^p$-strongly converge to $f$. In particular, if $f_i$ $W^{1,p}$-weakly converge to $f$, then $f_i$ $L^{p}$-strongly converge to $f$.\\
$(ii)$ $[$Lower semicontinuity$]$. If $f_i$ $W^{1,p}$-weakly converge to $f$, then $f\in W^{1,p}(B_R(x),\mathfrak{m})$ and
\begin{equation}\label{000}\liminf_{i \rightarrow \infty} \int_{B_{R_i}(x_i)} |\nabla f_i|^p\mathrm{d}\mathfrak{m}_i \geq \int_{B_R(x)} |\nabla f|^p\mathrm{d}\mathfrak{m}.
\end{equation}	
$(iii)$	If $f_i$ $W^{1,p}$-strongly converge to $f$, then $|\nabla f_i|$ $L^p$-strongly converge to $|\nabla f|$.
\begin{proof}
For any $r\in (0,R)$, fix $\varphi^r \in \operatorname{Lip}_c(B_R(x))$ with $0 \leq \varphi^r \leq 1$ and that $\varphi^r \equiv 1$ on $B_r(x)$. Here, we remark that $B_r(x),B_R(x)$ are open balls in $\mathrm{Y}$.\\
\textbf{Step 1}: there exist $f \in L^{p}(B_R(x),\mathfrak{m})$ and a subsequence $\{f_{i_j}\}$ such that $f_{i_j}$ $L^p$-strongly converge to $f$.
Note that by $(\ref{2.7})$, Lemma $\ref{lem4.13}$ and the fact that $\sup_i\|f_i\|_{W^{1,p}}<\infty$, we have
$$
\sup_i\|f_i\|_{L^{p^*}}<\infty.
$$
Thus, by the $L^p$-weak compactness(see Proposition $\ref{prop})$, up to a subsequence, we can assume that $f_i$ $L^p$-weakly converge to $f \in L^{p^*}(B_R(x),\mathfrak{m})$.

Since $\varphi^r f_i$ $L^p$-weakly converge to $\varphi^r f$ with $\sup_i\|\varphi^r f_i\|_{W^{1,p}}<\infty$, Lemma $\ref{lem1}$(i) yields that $\varphi^r f_i$ $L^p$-strongly converge to $\varphi^r f$, and that $\varphi^r f \in W^{1,p}(\mathrm{Y},\mathfrak{m})$. The H\"{o}lder inequality yields
\begin{equation}\label{7}
\|\varphi^r f_i-f_i\|_{L^p} \leq \mathfrak{m}_i(B_{R_i}(x_i) \backslash B_r(x))^{1/p-1 / p^*}\|f_i\|_{L^{p^*}},
\end{equation}
for enough large $i$. Combining this and the facts that  $\lim_i\mathfrak{m}_i(B_{R_i}(x_i) \backslash B_r(x))=\mathfrak{m}(B_R(x) \backslash B_r(x))$, $0\leq \varphi^r\leq 1$ and  $\varphi^r f_i$ $L^p$-strongly converge to $\varphi^r f$, we obtain that
\begin{align*}
\varlimsup_{i \rightarrow \infty}\|f_i\|_{L^p}&
\leq \varlimsup_{i \rightarrow \infty} \|\varphi^rf_i-f_i\|_{L^p}+\|\varphi^rf_i\|_{L^p}\\
& \leq \mathfrak{m}(B_R(x) \backslash B_r(x))^{1/p-1/p^*}\sup_i \|f_i\|_{L^{p^*}}+\|\varphi^rf\|_{L^p}\\
&\leq \mathfrak{m}(B_R(x) \backslash B_r(x))^{1/p-1/p^*}\sup_i \|f_i\|_{L^{p^*}}+\|f\|_{L^p}.
\end{align*}
Letting $r\rightarrow R$ and recalling that $p^*>p$, we obtain that $\varlimsup_{i \rightarrow \infty}\|f_i\|_{L^p}\leq \|f\|_{L^p}$ and thus $f_i$ $L^p$-strongly converge to $f$.\\
\textbf{Step 2}: we prove (ii), and thus the $L^p$-strongly limit in \textbf{Step 1} belongs to $ W^{1,p}(B_R(x),\mathfrak{m})$. Assume without loss of generality that $|\nabla f_i|$ $L^p$-weakly converge to $g \in L^p(B_R(x),\mathfrak{m})$. For any $\varphi\in \mathrm{Lip}_c(B_R(x))$, as usual, one can prove $\varphi f\in W^{1,p}(\mathrm{Y},\mathfrak{m})$, and thus $|\nabla f|$ is well defined. Next, we fix an open set $A \Subset B_R(x)$, then there exists $s \in(0, R)$ with $A \Subset B_s(x)$. Since the sequence $f_i \varphi^s$ $W^{1,p}$-weakly converges to $f \varphi^s$, from Lemma $\ref{lem2.9}$ with $g$ equal to the characteristic function of $A$ we get
$$
\int_{\bar{A}} g \mathrm{d}\mathfrak{m}\geq \limsup _i \int_A |\nabla f_i| \mathrm{d} \mathfrak{m}_i\geq \liminf_i \int_A |\nabla (\varphi^s f_i)| \mathrm{d} \mathfrak{m}_i \geq \int_A |\nabla (\varphi^s f)|\mathrm{d}\mathfrak{m}=\int_A |\nabla f|\mathrm{d}\mathfrak{m} .
$$
Now, since any open set $B \subset B_R(x)$ can be written as $\cup_i A_i$, with $A_i \Subset A_{i+1}$, the previous inequality gives $\int_B g \mathrm{d}\mathfrak{m} \geq \int_B |\nabla f| \mathrm{d}\mathfrak{m}$. Since $B$ is arbitrary, this proves the inequality $g \geq |\nabla f|$ and we obtain $(\ref{000})$ immediately, by the $L^p$-weak lower semicontinuity(see Proposition $\ref{prop})$. \\
\textbf{Step 3}: we prove (iii). Assume that $f_i$ $W^{1,p}$-strongly converge to $f$ and $|\nabla f_i|$ $L^p$-strongly converge to $g$. Using the $L^p$-weak lower semicontinuity again, we have that 
\begin{equation*}
\int |\nabla f|^p\mathrm{d}\mathfrak{m} =\lim_{i\rightarrow \infty} \int |\nabla f_i|^p\mathrm{d}\mathfrak{m}_i\geq \int g^p\mathrm{d}\mathfrak{m}.
\end{equation*}
Since $g\geq |\nabla f|$ by \textbf{Step 2}, we must have $g=|\nabla f|$ $\mathfrak{m}$-a.e. and thus $|\nabla f_i|$ $L^p$-strongly converge to $|\nabla f|$. 
\end{proof}
\end{proposition} 

\begin{lemma}\label{lem}
$(a)$ For any $f \in W_0^{1,p}(B_R(x),\mathfrak{m})$, up to  a subsequence, there exist $f_i \in W_0^{1,p}(B_{R_i}(x_i),\mathfrak{m}_i)$ such that $f_i$ $W^{1,p}$-strongly converge to $f$.	\\
$(b)$ If $f_i\in W_{0}^{1,p}(B_{R_i}(x_i),\mathfrak{m}_i)$ $L^p$-weakly converge to $f\in W^{1,p}(B_{R}(x),\mathfrak{m})$, then 
$f \in \hat{W}_0^{1,p}(B_R(x),\mathfrak{m})$.
\begin{proof}
For part (a), assume first that $f\in \mathrm{Lip}_c(B_R(x))\subset W^{1,p}(\mathrm{Y},\mathfrak{m})$. Then there exists a sequence of uniformly bounded functions $\{f_n\}\subset \mathrm{Lip}_{\mathrm{bs}}(\mathrm{Y})$ such that $f_n\rightarrow f$ and $\mathrm{Lip}_a(f_n) \rightarrow|\nabla f|$ in $L^p(\mathrm{Y},\mathfrak{m})$, (see for instance \cite{25}), where $\mathrm{Lip}_a(f_n)$ is the asymptotic Lipschitz constant of $f_n$.  For any $n$, the fact that the upper semicontinuity of the asymptotic Lipschitz constant implies that
$$
\limsup_{i \rightarrow \infty}  \operatorname{Ch}_{p}^i(f_n) \leq \limsup_{i \rightarrow \infty} \frac{1}{p}\int \operatorname{Lip}_a^{p}(f_n) \mathrm{d} \mathfrak{m}_i \leq\frac{1}{p}\int \operatorname{Lip}_a^p(f_n) \mathrm{d} \mathfrak{m}.
$$
Since $f_n|_{\mathrm{X}_i}$ $L^{p}$-strongly converge to $f_n|_{\mathrm{X}}$, there exists a subsequence $\{n_k\}$ such that $f_k|_{\mathrm{X}_{n_k}}$ $L^{p}$-strongly converge to $f$ and $\limsup_k \operatorname{Ch}_{p}^{n_k}(f_k) \leq \mathrm{Ch}_p(f)$.

Let $\epsilon>0$ with $\operatorname{supp} f \subset B_{R-3 \epsilon}(x)$ and $(0,\epsilon)\ni\epsilon_i\rightarrow 0$. The fact that $\epsilon_i\rightarrow 0, R_i\rightarrow R$ implies that there exists $\varphi_i \in D(\Delta_i)$ be satisfying
$$
0 \leq \varphi_i \leq 1,\;\; \varphi_i|_{B_{R_i-3\epsilon_i}(x_i)} \equiv 1, \quad \operatorname{supp} \varphi_i \subset B_{R_i-2\epsilon_i}(x_i), |\nabla \varphi_i|+|\Delta_i \varphi_i| \leq C(K, N,R):=C,
$$
(see \cite[Lemma 3.1]{26}). Hence, without loss of generality we can assume that $\varphi_i$ $W^{1,2}$-strongly converge to some $\varphi \in D(\Delta)$(see \cite[Corollary 5.5, Proposition 7.5]{21}). Then it is not difficult to check that $g_k:=\varphi_{n_k}f_k|_{\mathrm{X}_{n_k}}$ $L^p$-weakly converge to $\varphi f=f$, and, using Proposition $\ref{prop4}$, it remains only to prove that 
\begin{equation*}
\varlimsup_{k\rightarrow \infty}\int |\nabla g_k|^p\mathrm{d}\mathfrak{m}_{n_k}\leq \int |\nabla f|^p\mathrm{d}\mathfrak{m}.
\end{equation*}
This is obtained by a direct argument:
\begin{align*}
\bigg(\int |\nabla g_k|^p\mathrm{d}\mathfrak{m}_{n_k} \bigg)^{1/p}&\leq \bigg(\int |\nabla \varphi_{n_k}|^p|f_k|^p\mathrm{d}\mathfrak{m}_{n_k} \bigg)^{1/p}+\bigg(\int |\nabla f_k|^p\mathrm{d}\mathfrak{m}_{n_k} \bigg)^{1/p}\\
&\leq C\sup_k \|f_k\|_{L^\infty} \mathfrak{m}_{n_k}(B_{R_{n_k}-2\epsilon_{n_k}}(x_{n_k})\backslash B_{R_{n_k}-3\epsilon_{n_k}}(x_{n_k}))^{1/p}+\bigg(\int |\nabla f_k|^p\mathrm{d}\mathfrak{m}_{n_k} \bigg)^{1/p}\\
&\rightarrow \bigg(\int |\nabla f|^p\mathrm{d}\mathfrak{m} \bigg)^{1/p}\text{ as }k\rightarrow \infty.
\end{align*}

For the general case: $f\in W_0^{1,p}(B_R(x),\mathfrak{m})$, we can obtain the conclusion by a diagonal argument.

To prove (b), let $z \in \bar{B}_R(x)^c$ and let $r>0$ with $\bar{B}_r(z) \subset\bar{B}_R(x)^c$ and let $\varphi \in \operatorname{Lip}(\mathrm{X})$
with $\operatorname{supp} \varphi \subset B_r(z)$. Then, applying the proof of (a) to $f:=\varphi$ shows that, up to a subsequence, there exists $\varphi_i \in \operatorname{Lip}(\mathrm{X}_i)$ $W^{1,q}$-strongly converge to $\varphi$ such that $\operatorname{supp} \varphi_i \subset B_r(z_i)$, where $\mathrm{X}_i\ni z_i  \rightarrow z$, i.e. $\mathrm{d}(z_i,z)=0$. In particular, using Proposition $\ref{prop0}$, we deduce that 
$$
\int \varphi f \mathrm{d}\mathfrak{m}=\lim_{i \rightarrow \infty} \int \varphi_i f_i \mathrm{d}\mathfrak{m}_i=0 .
$$
Since $z,r,\varphi$ is arbitrary, this proves that $f \in \hat{W}_0^{1,p}(B_R(x),\mathfrak{m})$.
\end{proof}
\end{lemma}

In order to obtain the almost rigidity result, we state two auxiliary lemmas. The case $p=2$ is already covered by \cite[Lemma 4.10, Lemma 4.12]{11} and the proofs contained therein can be straightforward adapted to the present one. First, $L^p$-strong convergence of maps implies the almost pointwise convergence of the distribution functions to the distribution function of the limit.
\begin{lemma}[Convergence of distribution functions]\label{lem4.10}
Let $g_i \in L^p(\mathrm{X}_i,\mathfrak{m}_i)$ and $g \in L^p(\mathrm{X},\mathfrak{m})$. Let $\mu_i:=\mu_{g_i}$ and $\mu_g$ be the distribution functions of $g_i$ and $g$ respectively. If $g_i$ $L^p$-strongly converge to $g$ , then $\mu_i(t)$ converges to $\mu(t)$ for every $t \in(0,+\infty) \backslash C$, where $C$ is a countable set.	
\end{lemma}
Secondly, $W^{1,p}$-weakly convergence of functions implies $L^p$-strong convergence of the Schwarz symmetrizations.
\begin{lemma}\label{lem4.12}
Let $R_i,R\in (0,\infty)$ with $R_i\rightarrow R$. Let $g_i \in W_0^{1,p}(B_{R_i}(x_i),\mathfrak{m}_i)$, $g\in W^{1,p}(B_R(x),\mathfrak{m})$. Assume that $g_i$ $W^{1,p}$-weakly converge to $g$. Then, up to a subsequence, $g_i^{\star}$ converge to $g^{\star}$ in $L^p(J_{N-1,N},\mathfrak{m}_{N-1,N})$.	
\end{lemma}

\begin{theorem}[Almost rigidity]\label{thm7}
Let $p,q\in (1,\infty)$  be conjugate exponents. For every $\epsilon>0$, $N>2$, $v \in (0,1)$, $0<c_l \leq c_u<\infty$, there exists $\delta=\delta(\epsilon, p,N, v, c_l/c_u)>0$ such that the following statement holds. Assume that:\\
$(i)$ $(\mathrm{X}, \mathrm{d}, \mathfrak{m})$ is a $\mathrm{RCD}(N-1, N)$ space and $\Omega=B_R(x) \subset \mathrm{X}$ is an open ball with $\mathfrak{m}(\Omega)=v$;\\
$(ii)$ $f \in W_0^{1,q}(\Omega, \mathfrak{m})$ with $c_l \leq\|f\|_{L^q(\Omega, \mathfrak{m})} \leq\|f\|_{W^{1,q}(\Omega, \mathfrak{m})} \leq c_u$;\\
$(iii)$ $u \in W_0^{1,p}(\Omega, \mathfrak{m})$ is a weak solution of $-\mathscr{L}_p u=f$;\\
$(iv)$ $w \in W_0^{1,p}([0, r_v), \mathfrak{m}_{N-1, N})$ is a weak solution of $-\Delta_{p,N-1, N} w=f^{\star}$.

If $\|u^{\star}-w\|_{L^p([0, r_v), \mathfrak{m}_{N-1, N})}<\delta$, then there exists a spherical suspension $(Z, \mathrm{d}_Z, \mathfrak{m}_Z)$ such that
$$
\mathrm{d}_{\mathrm{mGH}}((\mathrm{X}, \mathrm{d}, \mathfrak{m}),(Z, \mathrm{d}_Z, \mathfrak{m}_Z))<\epsilon.
$$		
\begin{proof}
\textbf{Step 1}. We argue by contradiction: assume that there exist $\bar{\epsilon}, \bar{N}, \bar{v}, \bar{c}$ such that for any $i \in \mathbb{N}_+$ we can find a $\operatorname{RCD}(\bar{N}-1, \bar{N})$ space $(\mathrm{X}_i, \mathrm{d}_i, \mathfrak{m}_i)$, a ball $\Omega_i=B_{R_i}(x_i) \subset \mathrm{X}_i$ with $\mathfrak{m}_i(\Omega_i)=\bar{v}$, and functions $f_i \in W_0^{1,q}(\Omega_i,\mathfrak{m}_i)$, $u_i \in W_0^{1,p}(\Omega_i, \mathfrak{m}_i)$, $w_i \in W_0^{1,p}([0, r_{\bar{v}}), \mathfrak{m}_{\bar{N}-1, \bar{N}})$ such that for any $i \in \mathbb{N}_+$,
\begin{align*}
& -\mathscr{L}_p u_i=f_i \text { weakly, }-\Delta_{p,\bar{N}-1, \bar{N}} w_i=f_i^{\star} \text { weakly, }  \\
& \bar{c} \leq\|f_i\|_{L^q(\Omega, \mathfrak{m})} \leq\|f_i\|_{W^{1,q}(\Omega, \mathfrak{m})} \leq 1, \quad\|u_i^{\star}-w_i\|_{L^p([0, r_{\bar{v}}), \mathfrak{m}_{\bar{N}-1, \bar{N}})}<\frac{1}{i},
\end{align*}
and moreover
\begin{equation}\label{55}
\inf \{\mathrm{d}_{\mathrm{mGH}}((\mathrm{X}_i, \mathrm{d}_i, \mathfrak{m}_i),(Z, \mathrm{d}_{Z}, \mathfrak{m}_{Z})):(Z, \mathrm{d}_{\mathrm{Z}}, \mathfrak{m}_{Z}) \text { is a spherical suspension }\} \geq \bar{\epsilon}
\end{equation}
Up to subsequences, we can assume that $(\mathrm{X}_i, \mathrm{d}_i, \mathfrak{m}_i)$ converge to a $\operatorname{RCD}(N-1, N)$ space $(\mathrm{X}, \mathrm{d}, \mathfrak{m})$(by Theorem $\ref{com}$) and $x_i\rightarrow x\in \mathrm{X}$, and Assumption $\ref{assu}$ is satisfied; moreover, $\mathfrak{m}(B_R(x))=\bar{v}$ by Lemma $\ref{lem4.13}$;\\
\textbf{Step 2}:
We claim that $f_i$ $L^q$-strongly converge to a function $f \in L^q(B_R(x),\mathfrak{m})$ and $f_i^{\star}$ converges to $f^\star$ in $L^q([0, r_{\bar{v}}), \mathfrak{m}_{\bar{N}-1, \bar{N}})$, $u_i$ $L^p$-strongly converge to a function $u\in \hat{W}_0^{1,p}(B_R(x),\mathfrak{m})$ and $u_i^{\star}$ converges to $u^\star$ in $L^p([0, r_{\bar{v}}), \mathfrak{m}_{\bar{N}-1, \bar{N}})$, and $w_i$ converge to a function $w\in W_0^{1,p}([0, r_{\bar{v}}), \mathfrak{m}_{\bar{N}-1, \bar{N}})$ in $L^p([0, r_{\bar{v}}), \mathfrak{m}_{\bar{N}-1, \bar{N}})$. Indeed, first we note that the following identity was proved in $(\ref{22})$:
\begin{equation}\label{eq5}
w_i(\rho)=\int_{H_{\bar{N}-1, \bar{N}}(\rho)}^{H_{\bar{N}-1, \bar{N}}(r_{\bar{v}})}\frac{1}{\mathcal{I}_{\bar{N}-1, \bar{N}}(\sigma)} \bigg(\frac{1}{\mathcal{I}_{\bar{N}-1, \bar{N}}(\sigma)}\int_0^\sigma f_i^\star\circ H_{\bar{N}-1,\bar{N}}^{-1}(t)\mathrm{d}t\bigg)^{\frac{1}{p-1}} \mathrm{d} \sigma, \quad \forall \rho \in[0, r_{\bar{v}}].
\end{equation}
From the fact that $\sup_i\|f_i\|_{W^{1,q}}\leq 1$ by assumption, it follows that $\sup_i\|f_i^\star\|_{L^{q}}\leq 1$ by Proposition $\ref{3}$(b). The estimates $(\ref{es1})$,$(\ref{es3})$ and $(\ref{es2})$ imply that $w_i$ are also uniformly bounded in $W^{1,p}([0, r_{\bar{v}}), \mathfrak{m}_{\bar{N}-1, \bar{N}})$.
Using the Talenti-type Theorem $\ref{thm}$ and Proposition $\ref{3}$(b) again, we deduce that $\|u_i\|_{W^{1,p}}$ are uniformly bounded as well. Thus, by Proposition $\ref{prop4}$ and Lemma $\ref{lem}$(b), (and up to subsequences), $u_i$ $L^p$-strongly converge to a function $u\in \hat{W}_0^{1,p}(B_R(x),\mathfrak{m})$, $f_i$ $L^q$-strongly converge to a function $f\in \hat{W}_0^{1,q}(B_R(x),\mathfrak{m})$ and $w_i$ $L^p$-strongly converge to a function $w\in W_0^{1,p}([0,r_{\bar{v}}),\mathfrak{m}_{\bar{N}-1,\bar{N}})$. In particular, $w_i$ converge to $w$ in $L^p(J_{\bar{N}-1,\bar{N}},\mathfrak{m}_{\bar{N}-1,\bar{N}})$ by reflexivity of $L^p$ spaces. Moreover, by Lemma $\ref{lem4.12}$, $u_i^{\star}$ converge to $u^{\star}$ in $L^p(J_{\bar{N}-1,\bar{N}},\mathfrak{m}_{\bar{N}-1,\bar{N}})$ and $f_i^{\star}$ converge to $f^{\star}$ in $L^q(J_{\bar{N}-1,\bar{N}},\mathfrak{m}_{\bar{N}-1,\bar{N}})$. \\
\textbf{Step 3}: We claim that $w$ is a weak solution to $-\Delta_{p,\bar{N}-1,\bar{N}}w=f^\star$. The facts that $w_i$ converge to $w$ in $L^p(J_{\bar{N}-1,\bar{N}},\mathfrak{m}_{\bar{N}-1,\bar{N}})$, $f_i^{\star}$ converge to $f^{\star}$ in $L^q(J_{\bar{N}-1,\bar{N}},\mathfrak{m}_{\bar{N}-1,\bar{N}})$ and $(\ref{eq5})$ imply that $w$ can be represented as
\begin{equation}\label{eq6}
w(\rho)=\int_{H_{\bar{N}-1, \bar{N}}(\rho)}^{H_{\bar{N}-1, \bar{N}}(r_{\bar{v}})}\frac{1}{\mathcal{I}_{\bar{N}-1, \bar{N}}(\sigma)} \bigg(\frac{1}{\mathcal{I}_{\bar{N}-1, \bar{N}}(\sigma)}\int_0^\sigma f^\star\circ H_{\bar{N}-1,\bar{N}}^{-1}(t)\mathrm{d}t\bigg)^{\frac{1}{p-1}} \mathrm{d} \sigma, \quad \forall \rho \in[0, r_{\bar{v}}].
\end{equation}
We obtain the claim by Proposition $\ref{prop1}$. \\
\textbf{Step 4}. Notice that
$$
\mathrm{d}_{\mathrm{mGH}}((\mathrm{X}, \mathrm{d}, \mathfrak{m}),(Z, \mathrm{d}_{Z}, \mathfrak{m}_{Z})) \geq \bar{\epsilon}
$$
for any spherical suspension $(Z, \mathrm{d}_{Z}, \mathfrak{m}_{Z})$, by $(\ref{55})$. However, by the $L^p$-strong convergence of $u_i^{\star}$ to $u^{\star}$ and the $L^p$-strong convergence of $w_i$ to $w$, one has
$$
\|u^{\star}-w\|_{L^p([0, r_{\bar{v}}), \mathfrak{m}_{\bar{N}-1, \bar{N}})}=\lim_{i \rightarrow \infty}\|u_i^{\star}-w_i\|_{L^p([0, r_{\bar{v}}), \mathfrak{m}_{\bar{N}-1, \bar{N}})}=0,
$$
which implies that $u^{\star}=w$. This implies that $\mu_u(t)=\nu(t)$ for any $t>0$, where $\nu$ is the distribution function of $w$. Moreover, since $f$ is the $L^q$-strong limit of the $f_i$'s, it has $L^q$-norm bounded from below by $\bar{c}$, thus it is different from $0$ on a non-negligible set. This implies that $w\neq 0$. Recall that $(\ref{41})$ says that 
\begin{equation*}
\mathcal{I}_{\bar{N}-1, \bar{N}}(\mu_u(t))^{\frac{p}{p-1}}\leq (\mathrm{Per}(\mu_u(t)))^{\frac{p}{p-1}}\leq -\mu_u^{\prime}(t) \bigg(\int_0^{\mu_u(t)} f_i^{\sharp}(s) \mathrm{d} s\bigg)^{\frac{1}{p-1}},\text{ for $\mathscr{L}^1$-a.e. }t>0.
\end{equation*}
Letting $i\rightarrow \infty$ and using $(\ref{eq6})$, we have, for $\mathscr{L}^1$-a.e. $t\in (0,\mathrm{ess}\sup |u|)$,
\begin{equation*}
1\leq \frac{(\mathrm{Per}(\mu_u(t)))^{\frac{p}{p-1}}}{\mathcal{I}_{\bar{N}-1, \bar{N}}(\mu_u(t))^{\frac{p}{p-1}}}\leq \frac{-\mu_u^{\prime}(t)}{\mathcal{I}_{\bar{N}-1, \bar{N}}(\mu_u(t))^{\frac{p}{p-1}}}\bigg(\int_0^{\mu_u(t)} f^{\sharp}(s) \mathrm{d} s\bigg)^{\frac{1}{p-1}}=\nu^{\prime}(t)w^\prime(\nu(t))=1.
\end{equation*}
Hence, for any such $t$, the superlevel set $\{|u|>t\}$ satisfies $\mathcal{I}_{\bar{N}-1, \bar{N}}(\mathfrak{m}(\{|u|>t\}))=\operatorname{Per}(\{|u|>t\})$ and $\mathfrak{m}(\{|u|>t\})\in (0,1)$. By the rigidity in the L\'{e}vy-Gromov inequality, this implies that $(\mathrm{X}, \mathrm{d}, \mathfrak{m})$ is a spherical suspension.		
\end{proof}
\end{theorem}

\subsection{The continuity of the local Poisson problem}
\begin{theorem}\label{thm2}
Let $p,q\in (1,\infty)$  be conjugate exponents. Under the Assumption $\ref{assu}$, let $R_i,R\in (0,\infty)$ be such that $R_i\rightarrow R$ and $\mathfrak{m}_i(B_{R_i}(x_i))= v \in(0,1)$ and 
\begin{equation}\label{wo}
W_0^{1,p}(B_R(x),\mathfrak{m})=\hat{W}_0^{1,p}(B_R(x),\mathfrak{m})
.
\end{equation} 
Let $f_i \in W_0^{1,q}(B_{R_i}(x_i),\mathfrak{m}_i)$ with $\sup_i\|f_i\|_{W^{1,q}}<\infty$. Assume that $u_i \in W_0^{1,p}(B_{R_i}(x_i),\mathfrak{m}_i)$ are weak solutions to
$$
\begin{cases}
-\mathscr{L}_p u_i=f_i & \text { in } B_{R_i}(x_i) \\ u_i=0 & \text { on } \partial B_{R_i}(x_i)
\end{cases}
$$
and $w_i \in W_0^{1,p}([0,r_{v}), \mathfrak{m}_{N-1, N})$ are weak solutions to
$$
\begin{cases}
-\Delta_{p,N-1, N} w_i=f_i^{\star} \text { in }[0, r_{v}) \\ 
w_i(r_{v})=0.\end{cases}
$$
where $r_{v} = H_{N-1, N}^{-1}(v)$. Then, up to a subsequence:\\
$(i)$ $f_i$ $L^q$-strongly converge to a function $f \in L^q(B_R(x),\mathfrak{m})$ with $\mathfrak{m}(B_R(x))=v$; $f_i^{\star}$ converge to $f^{\star}$ in $L^q(J_{N-1,N},\mathfrak{m}_{N-1,N})$;\\
$(ii)$ $w_i$ $W^{1,p}$-strongly converge to a weak solution $w$ to
\begin{equation}\label{eq2}
\begin{cases}
-\Delta_{p,N-1, N} w=f^{\star}  \text { in }[0, r_{v}) \\ 
w(r_{v})=0.
\end{cases}
\end{equation}
If $p\geq 2$, in addition, then, up to a subsequence, \\
$(iii)$ $u_i$ $W^{1,p}$-strongly converge to a weak solution $u$ to
\begin{equation}\label{eq1}
\begin{cases}
-\mathscr{L}_p u=f & \text { in } B_R(x) \\ 
u=0 & \text { on } \partial B_R(x);
\end{cases}
\end{equation}
\begin{proof}
As in the proof of Theorem $\ref{thm7}$, up to subsequences, we can assume that $f_i$ $L^q$-strongly converge to a function $f \in L^q(B_R(x),\mathfrak{m})$ and $f_i^{\star}$ converges to $f^\star$ in $L^q([0, r_{v}), \mathfrak{m}_{N-1, N})$, $u_i$ $L^p$-strongly converge to a function $u\in \hat{W}_0^{1,p}(B_R(x),\mathfrak{m})$ and $u_i^{\star}$ converges to $u^\star$ in $L^p([0, r_{v}), \mathfrak{m}_{N-1, N})$, and $w_i$ $L^p$-strongly converge to a weak solution $w\in W_0^{1,p}([0, r_{v}), \mathfrak{m}_{N-1, N})$ of $-\Delta_{p,N-1, N} w=f^{\star}$. 

We claim that $u_i$ $W^{1,p}$-strongly converge to $u$. It suffices to prove that 
\begin{equation}\label{12}
\varlimsup_{i\rightarrow \infty}\int |\nabla u_i|^p\mathrm{d}\mathfrak{m}_i\leq \int |\nabla u|^p\mathrm{d}\mathfrak{m}.
\end{equation}
Indeed, since $u\in W_0^{1,p}(B_R(x),\mathfrak{m})$, Lemma $\ref{lem}$(a) says that up to a subsequence, there exist $v_i\in$ $W_0^{1,p}(B_{R_i}(x_i),\mathfrak{m}_i)$ such that $v_i$ $W^{1,p}$-strongly converge to $u$. Then, from the facts that $u_i$ is a weak solution to $(\ref{eq1})$ and $u_i,v_i\in W_0^{1,p}(B_{R_i}(x_i),\mathfrak{m}_i)$ can be used as a test function, we deduce that 
\begin{align}
\int f_i u_i-f_i v_i \mathrm{d}\mathfrak{m}_i & =-\mathscr{L}_p u_i(u_i)+\mathscr{L}_p u_i(v_i)\label{eq0}\\
&=\int|\nabla u_i|^{p} \mathrm{d}\mathfrak{m}_i-\int|\nabla u_i|^{p-2} Dv_i(\nabla u_i) \mathrm{d} \mathfrak{m}_i\notag\\
& \geq \int|\nabla u_i|^{p} \mathrm{d} \mathfrak{m}_i-\int|\nabla u_i|^{p-1}|\nabla v_i| \mathrm{d} \mathfrak{m}_i \notag\\
& \geq \int|\nabla u_i|^{p}\mathrm{d}\mathfrak{m}_i-\int\bigg(\frac{p-1}{p}|\nabla u_i|^{p}+\frac{1}{p}|\nabla v_i|^{p}\bigg) \mathrm{d}\mathfrak{m}_i \notag\\
& = \int \frac{1}{p}|\nabla u_i|^{p} \mathrm{d}\mathfrak{m}_i-\int \frac{1}{p}|\nabla v_i|^{p} \mathrm{d}\mathfrak{m}_i.\notag
\end{align}
Since $f_i$ $L^{q}$-strongly converge to $f$ and $u_i-v_i$ $L^{p}$-strongly converge to $0$, the integral in the left of $(\ref{eq0})$ converges to $0$, by Proposition $\ref{prop0}$. Thus, letting $i\rightarrow \infty$ in $(\ref{eq0})$, we obtain the desired inequality $(\ref{12})$. Similarly, we can prove that $w_i$ $W^{1,p}$-strongly converge to $w$.

To prove (iii), we first claim that if $g_i$ $W^{1,p}$-strongly converge to $g$ and $h_i$ $W^{1,p}$-strongly converge to $h$, then 
it holds
\begin{equation}\label{2}
\lim_{i\rightarrow \infty}\mathscr{L}_p g_i(h_i)=\mathscr{L}_p g(h).
\end{equation}
For $p=2$, $(\ref{2})$ has been proved in \cite[Theorem 5.4]{21}. We can assume that $p>2$. Indeed, on the one hand, by Remark $\ref{rmk3}$ and Proposition $\ref{prop4}$(iii), the fact that $g_i$ $W^{1,p}$-strongly converge to $g$ implies that $g_i$ $W^{1,2}$-strongly converge to $g$, similarly for $h_i,h$. Since $Dh_i(\nabla g_i)$ is bounded in $L^{p/2}$,  \cite[Theorem 5.7]{21} implies that $Dh_i(\nabla g_i)$ $L^{p/2}$-weakly converge to $Dg(\nabla h)$.
On the other hand, by \cite[Proposition 2.18(vii)]{0.3}, the fact that $|Dg_i|$ $L^p$-strongly converge to $|Dg|$ implies that $|Dg_i|^{p-2}$ $L^{\frac{p}{p-2}}$-strongly converge to $|Dg|^{p-2}$. Using Proposition $\ref{prop0}$, we obtain $(\ref{2})$ immediately.

Now we fix a $\psi \in W_0^{1,p}(B_R(x),\mathfrak{m})$. By Lemma $\ref{lem}$(a) again, there exist $\psi_i \in W_0^{1,p}(B_{R_i}(x_i),\mathfrak{m}_i)$ such that $\psi_i$ $W^{1,p}$-strongly converge to $\psi$. Since $u_i$ is a weak solution to $(\ref{eq1})$ again, we have
\begin{equation}\label{13}
-\mathscr{L}_p u_i(\psi_i)=\int  f_i \psi_i \mathrm{d} \mathfrak{m}_i.
\end{equation}
Since $\psi_i$ and $u_i$ $W^{1,p}$-strongly converge to $\psi$ and $u$, respectively, and $\psi_i$(resp. $f_i$) $L^{p}$(resp. $L^q$)-strongly converge to $\psi$(resp. $f$), we can use Propositions $\ref{prop0}$ and $(\ref{2})$ to deduce that 
\begin{equation*}
-\mathscr{L}_p u(\psi)=\int  f \psi \mathrm{d} \mathfrak{m},
\end{equation*}
this implies that $u$ is a weak solution to $(\ref{eq1})$. 
\end{proof}
\end{theorem}
\begin{remark}
We emphasize that the condition $(\ref{wo})$ is necessary. We consider a trivial mGH convergent sequence and a convergent sequence of points: 
$$
(J_{N-1,N}, \mathrm{d}_{eu}, \mathfrak{m}_{N-1,N}) \stackrel{\mathrm{mGH}}{\longrightarrow}(J_{N-1,N}, \mathrm{d}_{eu}, \mathfrak{m}_{N-1,N})\text{ and }x_i:=\pi/4-\epsilon_i\rightarrow x:=\pi/4
$$
as $\epsilon_i \downarrow 0$. For simplicity, we set $s_i=\pi/2-\epsilon_i$. Let $0\neq f\in \mathrm{Lip}_c((0,\pi/4))$ be a nonnegative function. Let $R_i=R=\pi/4$ such that $B_{R_i}(x_i)=[0,s_i)$ and  $B_R(x)=(0,\pi/2)$, and let $f_i=f$. Proposition $\ref{prop1}$ says that the unique weak solution of the Poisson problem
\begin{equation*}
\begin{cases}
- \Delta_{p,N-1, N} w =f_i \text { in } [0,s_i), \\ 
w(s_i)=0.
\end{cases}
\end{equation*}
can be represented as
$$
w_i(\rho)=\int_{H_{N-1, N}(\rho)}^{H_{N-1, N}(s_i)}\frac{1}{\mathcal{I}_{N-1, N}(\sigma)} \bigg(\frac{1}{\mathcal{I}_{N-1, N}(\sigma)}\int_0^\sigma f\circ H_{N-1,N}^{-1}(t)\mathrm{d}t\bigg)^{\frac{1}{p-1}} \mathrm{d} \sigma,\quad \rho\in [0,s_i].
$$
The $W^{1,p}$ strong limit function of $w_i$
$$
w(\rho)=\int_{H_{N-1, N}(\rho)}^{H_{N-1, N}(\pi/2)}\frac{1}{\mathcal{I}_{N-1, N}(\sigma)} \bigg(\frac{1}{\mathcal{I}_{N-1, N}(\sigma)}\int_0^\sigma f\circ H_{N-1,N}^{-1}(t)\mathrm{d}t\bigg)^{\frac{1}{p-1}} \mathrm{d} \sigma,\quad \rho\in [0,\pi/2].
$$
is continuous on $[0,\pi/2]$ and $w(0)\neq 0$, so it is not in the space $W_0^{1,p}((0,\pi/2),\mathfrak{m}_{N-1,N})$.
\end{remark}

\section{Applications}
\subsection{Improved Sobolev embeddings}
\begin{theorem}
Let $(\mathrm{X}, \mathrm{d}, \mathfrak{m})$ be a $\operatorname{RCD}(K, N)$ space with $K>0$, $N \in(1, \infty)$ and $\Omega \subset \mathrm{X}$  an open set with $v=\mathfrak{m}(\Omega) \in(0,1)$. Let $p,q\in (1,\infty)$  be conjugate exponents.

Let $u\in W_0^{1,p}(\Omega)$ and $f \in L^q(\Omega)$. Assume that $u$ is a weak solution of $-\mathscr{L}_pu=f$. Then the following statements hold:\\
$(i)$ If $f \in L^s(\Omega)$ with $\frac{N}{p}<s \leq \infty$, then $\|u\|_{L^{\infty}(\Omega)} \leq c_1(K, N, v, p,s)\|f\|_{L^s(\Omega)}^{\frac{1}{p-1}}$, where
\begin{align}
c_1(K, N, v, p,s) = \int_0^{\mathfrak{m}(\Omega)} \frac{\xi^{\frac{s-1}{s(p-1)}}}{\mathcal{I}_{K, N}(\xi)^{\frac{p}{p-1}}} \mathrm{d} \xi<\infty,
\end{align}
where we adopt the convention that $\frac{1}{s}=0$ if $s=\infty$.\\
$(ii)$ If $f \in L^s(\Omega)$ with $0<s \leq \frac{N}{p}$, and $t \geq 1$ is such that $\frac{t}{p-1}(\frac{1}{s}-\frac{p}{N})<1$, then $\|u\|_{L^t(\Omega)}\leq c_2(K, N, v, p,s,t)\|f\|_{L^s(\Omega)}^{\frac{1}{p-1}}$, where
\begin{align*}
 c_2(K, N, v, p,s,t)  =\bigg(\int_0^v\bigg(\int_x^{v} \frac{\xi^{\frac{s-1}{s(p-1)}}}{\mathcal{I}_{K, N}(\xi)^{\frac{p}{p-1}}} \mathrm{d} \xi\bigg)^t \mathrm{d} x\bigg)^{\frac{1}{t}}<\infty .
\end{align*}
\begin{proof}
Recall that $u^{\sharp}$ satisfies the following inequality (see $(\ref{9})$):
$$
u^{\sharp}(x) \leq
\int_x^{\mathfrak{m}(\Omega)}\frac{1}{\mathcal{I}_{K, N}(\xi)} \bigg(\frac{1}{\mathcal{I}_{K, N}(\xi)}\int_0^\xi f^{\sharp}(t)\mathrm{d}t\bigg)^{\frac{1}{p-1}}\mathrm{d} \xi
\quad \forall x \in[0, \mathfrak{m}(\Omega)].
$$
If $f \in L^s(\Omega)$, then by H\"{o}lder inequality, by equimeasurability of $f$ and $f^{\sharp}$, and by the characterization of the isoperimetric profile on $J_{K, N}$(Lemma $\ref{lem2.12}$),
\begin{equation}\label{10}
u^{\sharp}(x) \leq \|f\|_{L^s(\Omega)}^{\frac{1}{p-1}} \int_x^{\mathfrak{m}(\Omega)} \frac{\xi^{\frac{s-1}{s(p-1)}}}{h_{K, N}(H_{K, N}^{-1}(\xi))^{\frac{p}{p-1}}} \mathrm{d} \xi,
\end{equation}
with the convention that $1/s=0$ if $s=\infty$. Now by the estimates on $h_{K, N}$(Lemma $\ref{lem2.11}$), there exists a constant $C_1>0$ only depending on $K>0, N \in(1, \infty)$ and $v=\mathfrak{m}(\Omega) \in(0,1)$ such that for all $\xi \in[0, \mathfrak{m}(\Omega)]$
$$
h_{K, N}(H_{K, N}^{-1}(\xi)) \geq C_1 \xi^{\frac{N-1}{N}} .
$$
We can thus draw the following conclusions:\\
\textbf{Case 1}: If $\frac{N}{p}<s \leq \infty$, then
$$
\int_x^{\mathfrak{m}(\Omega)} \frac{\xi^{\frac{s-1}{s(p-1)}}}{h_{K, N}(H_{K, N}^{-1}(\xi))^{\frac{p}{p-1}}} \mathrm{d} \xi \leq \frac{1}{C_1^{\frac{p}{p-1}}} \int_0^{\mathfrak{m}(\Omega)} \xi^{\frac{1}{p-1}(\frac{p}{N}-\frac{1}{s}+1-p)} \mathrm{d} \xi<\infty, \quad \forall x \in[0, \mathfrak{m}(\Omega)] .
$$
By $(\ref{10})$ and by equimeasurability of $u$ and $u^{\sharp}$, this implies (i).\\
\textbf{Case 2}: If $s=\frac{N}{p}$ and $t \geq 1$, then
$$
\bigg(\int_x^{\mathfrak{m}(\Omega)} \frac{\xi^{\frac{s-1}{s(p-1)}}}{h_{K, N}(H_{K, N}^{-1}(\xi))^{\frac{p}{p-1}}} \mathrm{d} \xi\bigg)^t \leq C_1^{-\frac{pt}{p-1}}
(\log v-\log x)^t
$$
and thus
\begin{align*}
\|u\|_{L^t(\Omega)}^t=\|u^{\sharp}\|_{L^t([0, v))}^t & \leq \|f\|_{L^s(\Omega)}^{\frac{t}{p-1}} \int_0^v\bigg(\int_x^{\mathfrak{m}(\Omega)} \frac{\xi^{\frac{s-1}{s(p-1)}}}{h_{K, N}(H_{K, N}^{-1}(\xi))^{\frac{p}{p-1}}} \mathrm{d} \xi\bigg)^t \mathrm{d}x \\
& \leq \frac{\|f\|_{L^s(\Omega)}^{\frac{t}{p-1}}}{C_1^{\frac{pt}{p-1}}} \int_0^1(-\log x)^t\mathrm{d} x<\infty,
\end{align*}
where we used $t\geq 1$.\\
\textbf{Case 3}: If $0< s<\frac{N}{p}$ and $t \geq 1$, with $\frac{t}{p-1}(\frac{1}{s}-\frac{p}{N})<1$, then
$$
\bigg(\int_x^{\mathfrak{m}(\Omega)} \frac{\xi^{\frac{s-1}{s(p-1)}}}{h_{K, N}(H_{K, N}^{-1}(\xi))^{\frac{p}{p-1}}} \mathrm{d} \xi\bigg)^t\leq \frac{1}{C_1^{\frac{pt}{p-1}}\frac{1}{(p-1)^t}(\frac{1}{s}-\frac{p}{N})^t}(x^{\frac{1}{p-1}(\frac{p}{N}-\frac{1}{s})}-v^{\frac{1}{p-1}(\frac{p}{N}-\frac{1}{s})})^t
$$
and thus
\begin{align*}
\|u\|_{L^t(\Omega)}^t=\|u^{\sharp}\|_{L^t([0, v))}^t & \leq \|f\|_{L^s(\Omega)}^{\frac{t}{p-1}} \int_0^v\bigg(\int_x^{\mathfrak{m}(\Omega)} \frac{\xi^{\frac{s-1}{s(p-1)}}}{h_{K, N}(H_{K, N}^{-1}(\xi))^{\frac{p}{p-1}}} \mathrm{d} \xi\bigg)^t\mathrm{d}x \\
& \leq \frac{\|f\|_{L^s(\Omega)}^{\frac{t}{p-1}}}{C_1^{\frac{pt}{p-1}}\frac{1}{(p-1)^t}(\frac{1}{s}-\frac{p}{N})^t} \int_0^v(x^{\frac{1}{p-1}(\frac{p}{N}-\frac{1}{s})}-v^{\frac{1}{p-1}(\frac{p}{N}-\frac{1}{s})})^t\mathrm{d} x<\infty.
\end{align*}
\end{proof}
\end{theorem} 

\subsection{Rayleigh–Faber–Krahn–B\'{e}rard–Meyer comparison theorem}
Let $(\mathrm{X}, \mathrm{d}, \mathfrak{m})$ be a metric measure space. Let $\Omega \subset \mathrm{X}$ be an open set and $p\in (1,\infty)$.
\begin{definition}
For any non-zero function $w \in W^{1,p}(\Omega)$ we define the Rayleigh quotient to be
$$
\mathcal{R}_{\Omega}(w) = \frac{{\int_{\Omega}|\nabla w|^p \mathrm{d}\mathfrak{m}}}{{\int_{\Omega} |w|^p \mathrm{d}\mathfrak{m}}} .
$$
We say that:\\
$(i)$ $\lambda_{\mathrm{X}}^p(\Omega) = \inf \{\mathcal{R}_{\Omega}(w) \mid w \in W_0^{1,p}(\Omega), w \neq 0\}$ is the first eigenvalue of the $p$-Laplacian in $\Omega$ $($with Dirichlet homogeneous conditions$)$;\\
$(ii)$ $u \in W_0^{1,p}(\Omega)$ is a first eigenfunction of the $p$-Laplacian in $\Omega$ $($with Dirichlet homogeneous conditions$)$ if it minimizes $\mathcal{R}_{\Omega}$ among functions $w \in W_0^{1,p}(\Omega), w \neq 0$.
\end{definition}
When $(\mathrm{X}, \mathrm{d}, \mathfrak{m})=(J_{K, N}, \mathrm{d}_{eu}, \mathfrak{m}_{K, N})$ with $K>0$, $N \in(1, \infty)$, $v \in(0,1)$, and $\Omega=[0, H_{K, N}^{-1}(v))$, we will denote the first eigenvalue with $\lambda_{K, N, v}^p$.

The following result was proved by Mondino and Semola in \cite{18}. We give below an alternative proof based instead on Talenti's comparison theorem.
\begin{theorem}\label{thm4}
Let $(\mathrm{X}, \mathrm{d}, \mathfrak{m})$ be a $\operatorname{RCD}(K, N)$ space with $K>0$, $N \in(1, \infty)$, and let $\Omega \subset \mathrm{X}$ be an open set with measure $v = \mathfrak{m}(\Omega) \in(0,1)$. Then:\\
$(i)$ $\lambda_{\mathrm{X}}^p(\Omega) \geq \lambda_{K, N, v}^p$;\\
$(ii)$ There exists a first eigenfunction of the  $p$-Laplacian in $\Omega$; such an eigenfunction can be chosen to be strictly positive and continuous in $\Omega$;\\
$(iii)$ If $u$ is a positive first eigenfunction of the  $p$-Laplacian in $\Omega$, then $0<u^{\star} \leq w$ in $[0, r_v)$, where $r_v=H_{K, N}^{-1}(v)$ and $w$ is a solution to $-\Delta_{p,K, N} w= \lambda_{\mathrm{X}}^p(\Omega) (u^{\star})^{p-1}$ in $[0, r_v)$ with $w(r_v)=0$.
\begin{proof}
\textbf{Step 1}: A first eigenfunction exists. Let $\{u_n\}_n$ be a minimizing sequence for $\mathcal{R}_{\Omega}$ with $u_n \in W_0^{1,p}(\Omega)$, $\|u_n\|_{L^p}=1$ and $\int_{\Omega}|\nabla u_n|^p \mathrm{d}\mathfrak{m} \searrow \lambda_{\mathrm{X}}^p(\Omega)$. Without loss of generality, we can assume that $\sup_n \|u_n\|_{W^{1,p}}<\infty$ and thus, by Lemma $\ref{lem1}$(i), up to a subsequence, the sequence $u_n$ converges to a function $u \in W_0^{1,p}(\Omega)$ in $L^p(\Omega)$. Thus $\|u\|_{L^p}=1$; moreover, from the very definition of $\lambda_{\mathrm{X}}^p(\Omega)$ and by Lemma $\ref{lem1}$(ii), it holds that
$$
\lambda_{\mathrm{X}}^p(\Omega) \leq \int_{\Omega}|\nabla u|^p \mathrm{d}\mathfrak{m} \leq \liminf_{n \rightarrow \infty} \int_{\Omega}|\nabla u_n|^p \mathrm{d}\mathfrak{m}=\lambda_{\mathrm{X}}^p(\Omega).
$$
Thus $u$ is a first eigenfunction.\\
\textbf{Step 2}: Any first eigenfunction $u \in W_0^{1,p}(\Omega)$ is a weak solution to
\begin{equation}\label{eq3}
\left\{\begin{array}{rlrl}
-\mathscr{L}_pu & =\lambda_{\mathrm{X}}^p(\Omega) |u|^{p-2}u & & \text { in } \Omega \\
u & =0 & & \text { on } \partial \Omega.
\end{array}\right.
\end{equation}
This relies on a standard variational argument: for any $\psi\in W_0^{1,p}(\Omega)$, we can explicitly compute the derivative of $\epsilon \mapsto \mathcal{R}_{\Omega}(u+\epsilon \psi)$ at $\epsilon=0$:
$$
\frac{\mathrm{d}}{\mathrm{d} \epsilon} \mathcal{R}_{\Omega}(u+\epsilon \psi)\bigg|_{\epsilon=0}=p \frac{-\mathscr{L}_pu(\psi)-\lambda_{\mathrm{X}}^p(\Omega) \int_{\Omega} |u|^{p-2}u \psi \mathrm{d}\mathfrak{m}}{\|u\|_{L^p}^p} .
$$
Since this derivative must vanish, the identity $-\mathscr{L}_pu(\psi)=\lambda_{\mathrm{X}}^p(\Omega) \int_{\Omega} |u|^{p-2}u \psi \mathrm{d}\mathfrak{m}$ needs to hold for any $\psi \in W_0^{1,p}(\Omega)$, which proves that $u$ is a weak solution to $(\ref{eq3})$.\\
\textbf{Step 3}. Since $\operatorname{RCD}(K, N)$ spaces are locally doubling and support a weak local Poincar\'{e} inequality(see Proposition $\ref{prop5}$), the Harnack type results proved in \cite[Theorem 5.1]{32} hold in these spaces and hence any first eigenfunction is continuous, and is positive in $\Omega$ up to multiplying by a constant(Corollary 5.7 and Corollary 5.8 therein). From now on, we always assume that $u$ is positive.\\
\textbf{Step 4}. Let now $w$ be a solution to $-\Delta_{p,K, N} w= (\lambda_{\mathrm{X}}^p(\Omega) u^{p-1})^\star=\lambda_{\mathrm{X}}^p(\Omega) (u^{\star})^{p-1}$(by Proposition $\ref{3}$(d)) in $[0, r_v)$ with $w(r_v)=0$. By the definition of $w$ and by using $w$ itself as a test function, it holds that $\int_0^{r_v}|w^\prime|^{p} \mathrm{d}\mathfrak{m}_{K, N}=\lambda_{\mathrm{X}}^p(\Omega) \int_0^{r_v}  (u^{\star})^{p-1} w \mathrm{d}\mathfrak{m}_{K, N}$; by the Talenti-type theorem and the fact that $u$ is positive, it holds that $0<u^{\star} \leq w$ on $[0,r_v)$. Thus
$$
\lambda_{K, N, v}^p \leq \frac{\int_0^{r_v}|w^\prime|^p \mathrm{d}\mathfrak{m}_{K, N}}{\int_0^{r_v} |w|^p \mathrm{d}\mathfrak{m}_{K, N}}=\lambda_{\mathrm{X}}^p(\Omega) \frac{\int_0^{r_v} ( u^{\star})^{p-1}w \mathrm{d}\mathfrak{m}_{K, N}}{\int_0^{r_v} |w|^p \mathrm{d}\mathfrak{m}_{K, N}} \leq \lambda_{\mathrm{X}}^p(\Omega) \frac{\int_0^{r_v} |w|^p \mathrm{d}\mathfrak{m}_{K, N}}{\int_0^{r_v} |w|^p \mathrm{d}\mathfrak{m}_{K, N}}=\lambda_{\mathrm{X}}^p(\Omega).
$$
\end{proof}
\end{theorem}

Finally, we recall that the following rigid result which will be used in the next subsection.
\begin{theorem}[Rigidity for the $p$-spectral gap]\label{thm4.5}\cite{18}
Let $(\mathrm{X}, \mathrm{d}, \mathfrak{m})$ be an $\operatorname{RCD}(N-1, N)$ space with $N>2$. Let $\Omega \subset \mathrm{X}$ be an open set such that $\mathfrak{m}(\Omega)=v \in(0,1)$ and suppose that  $\lambda_{\mathrm{X}}^p(\Omega) = \lambda_{N-1, N, v}^p$.

Then $(\mathrm{X},\mathrm{d},\mathfrak{m})$ is a spherical suspension, and the first eigenfunction of the $p$-Laplacian in $\Omega$ is unique up to scalar factor $($and it coincides with the radial one$)$.
\end{theorem}

\subsection{A reverse H\"{o}lder inequality for the first Dirichlet eigenfunctions}
Let $(\mathrm{X}, \mathrm{d}, \mathfrak{m})$ be a $\operatorname{RCD}(N-1, N)$ space with $N>2$, and $\Omega \subset \mathrm{X}$ an open set such that $\mathfrak{m}(\Omega)=v \in(0,1)$. Let $p>1$. Theorem $\ref{thm4}$ says that:
$$
\lambda_{N-1, N, v}^p \leq \lambda_{\mathrm{X}}^p(\Omega).
$$
From the variational characterization of the first eigenvalue, the fact that $\lambda_{N-1, N, t}^p \rightarrow \infty$ as $t \rightarrow 0$, and that $t \mapsto \lambda_{N-1, N, t}^p$ is a continuous function, it follows that there exists $\alpha \in(0, v]$ such that $\lambda_{N-1, N, \alpha}^p=\lambda_{\mathrm{X}}^p(\Omega)$. Moreover, Theorem $\ref{thm4.5}$ says that the first eigenfunction of the $p$-Laplacian in $[0, r_\alpha)$, which we will denote by $[0,r_\alpha]\ni s\mapsto z(s)$, is unique up to scalar factor.
\begin{theorem}\label{thm3}
Let $p>1,r>0$ and let $(\mathrm{X}, \mathrm{d}, \mathfrak{m})$ be a $\operatorname{RCD}(N-1, N)$ space with $N>2$, $\Omega \subset \mathrm{X}$ an open set such that $\mathfrak{m}(\Omega)=v \in(0,1)$, and $u\in W_0^{1,p}(\Omega)$ a positive first eigenfunction of the $p$-Laplacian in $\Omega$. Suppose that $\alpha \leq v$ is such that $\lambda_{N-1, N, \alpha}^p=\lambda_{\mathrm{X}}^p(\Omega)$. Let $z:[0, r_\alpha] \rightarrow [0,\infty)$ be the first eigenfunction of the $p$-Laplacian in $[0, r_\alpha)$, scaled such that:
$$
\int_{\Omega} u^r \mathrm{d} \mathfrak{m}=\int_0^{r_\alpha} z^r \mathrm{d} \mathfrak{m}_{N-1, N} .
$$
Then, for all $t \geq r$, it holds that:
\begin{equation}\label{1.6}
\frac{(\int_{\Omega} u^t \mathrm{d} \mathfrak{m})^{\frac{1}{t}}}{(\int_{\Omega} u^r \mathrm{d} \mathfrak{m})^{\frac{1}{r}}} \leq \frac{(\int_0^{r_\alpha} z^t\mathrm{d} \mathfrak{m}_{N-1, N})^{\frac{1}{t}}}{(\int_0^{r_\alpha} z^r \mathrm{d} \mathfrak{m}_{N-1, N})^{\frac{1}{r}}}.
\end{equation}
Furthermore, if the equality holds for some $t>\max\{1,r\}$ then $(\mathrm{X}, \mathrm{d}, \mathfrak{m})$ is a spherical suspension.
\end{theorem}

\begin{lemma}
Let $p>1$ and let $(\mathrm{X}, \mathrm{d}, \mathfrak{m})$ be a $\operatorname{RCD}(N-1, N)$ space with $N>2$, $\Omega \subset \mathrm{X}$ an open set such that $\mathfrak{m}(\Omega)=v \in(0,1)$, and $u\in W_0^{1,p}(\Omega)$ a positive first eigenfunction of the $p$-Laplacian in $\Omega$. Suppose that $\alpha \leq v$ is such that $\lambda_{N-1, N, \alpha}^p=\lambda_{\mathrm{X}}^p(\Omega)$. Let $z:[0, r_\alpha] \rightarrow [0,\infty)$ be the first eigenfunction of the $p$-Laplacian in $[0, r_\alpha)$, scaled such that:
$$
z(0)=u^\star(0).
$$	
Then:\\
- either $\alpha=v$ and $z(s)=u^\star(s)$ for every $s \in[0, r_v]$;\\
- or $\alpha<v$ and $z(s) \leq u^\star(s)$ for every $s \in[0, r_\alpha]$.
\begin{proof}
If $\alpha=v$, we have $\lambda_{\mathrm{X}}^p(\Omega)=\lambda_{N-1, N, v}^p$. Theorem $\ref{thm4}$, $\ref{thm4.5}$ say that $\mathrm{X}$ is a spherical suspension and $z(s)=u^\star(s)$ for every $s \in[0, r_v]$.

Now we turn to the case $\alpha<v$. $(\ref{14})$ says that for any $0 \leq t^{\prime}<t < \sup_\Omega u$ it holds:
$$
t-t^{\prime} \leq \int_{\mu(t)}^{\mu(t^{\prime})} \frac{(\lambda_{N-1, N, \alpha}^p)^{\frac{1}{p-1}}}{\mathcal{I}_{N-1, N}(\xi)^{\frac{p}{p-1}}}\bigg(\int_0^{\xi} (u^{\sharp})^{p-1}(t) \mathrm{d} t\bigg)^{\frac{1}{p-1}} \mathrm{d} \xi .
$$
Hence, letting $0 \leq s<s^{\prime} \leq \mu(0)=v$, setting $t=u^{\sharp}(s)-\epsilon, t^{\prime}=u^{\sharp}(s^{\prime})$ for $\epsilon>0$ small enough and then taking $\epsilon \rightarrow 0$ (noting that by manipulating equi-measurability, $\mu(u^{\sharp}(s^{\prime})) \leq s^{\prime}$ and $\mu(u^{\sharp}(s)-\epsilon) \geq s)$:
$$
u^{\sharp}(s)-u^{\sharp}(s^{\prime}) \leq \int_s^{s^{\prime}}\frac{(\lambda_{N-1, N, \alpha}^p)^{\frac{1}{p-1}}}{\mathcal{I}_{N-1, N}(\xi)^{\frac{p}{p-1}}}\bigg(\int_0^{\xi} (u^{\sharp})^{p-1}(t) \mathrm{d} t\bigg)^{\frac{1}{p-1}} \mathrm{d} \xi.
$$
Combining this and the fact that $u$ is bounded, it follows that $u^{\sharp}$ is Lipschitz continuous on $[0, v]$ and thus $\mathscr{L}^1$-a.e. on $(0, v)$. Hence,  we have that the following estimate holds for $\mathscr{L}^1$-a.e. $s \in(0, v)$:
\begin{equation}\label{3.3}
-\frac{\mathrm{d}}{\mathrm{d}s} u^{\sharp}(s) \leq \frac{(\lambda_{N-1, N, \alpha}^p)^{\frac{1}{p-1}}}{\mathcal{I}_{N-1, N}(s)^{\frac{p}{p-1}}}\bigg(\int_0^{s} (u^{\sharp})^{p-1}(t) \mathrm{d} t\bigg)^{\frac{1}{p-1}}.
\end{equation}
Furthermore, by Proposition $\ref{prop1}$, we have the following equality,
\begin{equation}\label{3.4}
-\frac{\mathrm{d}}{\mathrm{d} s} z^{\sharp}(s)=\frac{(\lambda_{N-1, N, \alpha}^p)^{\frac{1}{p-1}}}{\mathcal{I}_{N-1, N}(s)^{\frac{p}{p-1}}}\bigg(\int_0^{s} (z^{\sharp})^{p-1}(t) \mathrm{d} t\bigg)^{\frac{1}{p-1}} \quad \text { for every } s \in(0, \alpha) .
\end{equation}
Now, from the above estimates we know that $u^{\sharp}$ and $z^{\sharp}$ are continuous functions on $[0, \alpha]$, with $u^{\sharp}>0$ on $[0, \alpha]$. Thus, $\frac{z^{\sharp}(s)}{u^{\sharp}(s)}$ is also a continuous function on $[0, \alpha]$ and we can set:
$$
a:=\max _{s \in[0, \alpha]} \frac{z^{\sharp}(s)}{u^{\sharp}(s)} .
$$
We claim that $a\leq 1$. Assume by contradiction that $a>1$, and set:
$$
s_0:=\inf \{s \in[0, \alpha]: a u^{\sharp}(s)=z^{\sharp}(s)\} .
$$
Clearly it holds that $s_0>0$, since $u^{\sharp}(0)=z^{\sharp}(0)$ and both functions are right-continuous at $0$. Next, define:
$$
w^{\sharp}(s):= \begin{cases}a u^{\sharp}(s), & \text { for } 0 \leq s \leq s_0 \\ z^{\sharp}(s), & \text { for } s_0 \leq s \leq \alpha .\end{cases}
$$
Combining this definition and  $(\ref{3.3})$, $(\ref{3.4})$, we have:
\begin{equation}\label{3.7}
-\frac{\mathrm{d}}{\mathrm{d} s} w^{\sharp}(s) \leq \frac{(\lambda_{N-1, N, \alpha}^p)^{\frac{1}{p-1}}}{\mathcal{I}_{N-1, N}(s)^{\frac{p}{p-1}}}\bigg(\int_0^{s} (w^{\sharp})^{p-1}(t) \mathrm{d} t\bigg)^{\frac{1}{p-1}}, \quad \text { for $\mathscr{L}^1$-a.e. } s \in(0, s_0) \cup(s_0, \alpha).
\end{equation}
Letting $w:=w^{\sharp}\circ H_{N-1,N}$, we can see that:
\begin{align*}
\int_0^{r_\alpha}|w^{\prime}|^p \mathrm{d} \mathfrak{m}_{N-1, N}&=\int_0^{r_\alpha}(-w^\prime)(x)|w^{\sharp\prime}(H_{N-1,N}(x))|^{p-1} h_{N-1, N}(x)^p \mathrm{d} x\\
&\leq \lambda_{N-1, N, \alpha}^p\int_0^{r_\alpha}(-w^\prime)(x)\int_0^{H_{N-1,N}(x)} (w^{\sharp})^{p-1}(t) \mathrm{d} t \mathrm{d} x\\
&=\lambda_{N-1, N, \alpha}^p\int_0^{r_\alpha}|w|^p\mathfrak{m}_{N-1, N},
\end{align*}
where we have used the estimate $(\ref{3.7})$ and the fact $w^{\sharp{\prime}} \leq 0$ for the inequality, and integrated by parts for the first equality noting that we have $w^{\sharp}(\alpha)=0$. Thus, $w$ minimizes the Rayleigh quotient and therefore it is a first eigenfunction of the $p$-Laplacian in $[0, r_\alpha)$. Henceforth, $w$ is a multiple of $z$ and thus $a u^{\sharp}(s)=z^{\sharp}(s)$ on $s \in[0, s_0]$. Since $u^{\sharp}(0)=u^\star(0)=z(0)=z^{\sharp}(0)$, we conclude that $a=1$, contradicting that $a>1$. Hence, we prove the claim, in other words, $z(s) \leq u^\star(s)$ for every $s \in[0, r_\alpha]$. In particular, we complete the proof of Lemma.
\end{proof}
\end{lemma}

\begin{theorem}[Chiti's Comparison Theorem]
With the same assumptions as in Theorem $\ref{thm3}$. Then there exists $r_1 \in(0, r_\alpha)$ such that
$$
\begin{cases}u^\star(s) \leq z(s) & \text { for } 0 \leq s \leq r_1 \\ z(s) \leq u^\star(s) & \text { for } r_1 \leq s \leq r_\alpha.\end{cases}
$$
\end{theorem}

The proof of the Chiti's Comparison Theorem and Theorem $\ref{thm3}$ are slight modifizations of \cite[Theorem 1.2, 3.2]{42} by using $(\ref{3.3})$ and $(\ref{3.4})$ and we omit them. Next, we prove an almost stable version of Theorem $\ref{thm3}$.
\begin{theorem}[almost stability in the reverse H\"{o}lder inequality]\label{thm8}
Let $p>1$, $N \in(2, \infty)$, $v \in(0,1)$ and $\lambda>0$ be given. Then there exists $C=C(N, v, p,\lambda)>0$ and $\tilde{\delta}=\tilde{\delta}(N, v, p,\lambda)>0$ such that the following holds.

Let $(\mathrm{X}, \mathrm{d}, \mathfrak{m})$ be a $\operatorname{RCD}(N-1, N)$ space. Let $\Omega \subset \mathrm{X}$ be an open set with $\mathfrak{m}(\Omega)=v \in(0,1)$ and $\lambda_{\mathrm{X}}^p(\Omega)=\lambda$. Let $u$ and $z$ be as in the assumptions of Theorem $\ref{thm3}$, scaled such that:
\begin{equation}\label{1.75}
\int_{\Omega} u^{p-1} \mathrm{d} \mathfrak{m}=\int_0^{r_\alpha} z^{p-1} \mathrm{d}\mathfrak{m}_{N-1, N}.
\end{equation}
If there exist $\delta \in(0, \tilde{\delta})$ and an unbounded subset $Q \subset(p-1, \infty)$ such that
\begin{equation}\label{1.7}
\begin{cases}
\|z\|_{L^t([0, r_\alpha), \mathfrak{m}_{N-1, N})}^{p-1}-\|u\|_{L^t(\Omega, \mathfrak{m})}^{p-1}<\delta, \quad \text { for all } t \in Q,& \text{ if } p\geq 2,\\
(\|z\|_{L^t([0, r_\alpha), \mathfrak{m}_{N-1, N})}-\|u\|_{L^t(\Omega, \mathfrak{m})})^{p-1}<\delta, \quad \text { for all } t \in Q,& \text{ if } p\in (1,2),\\
\end{cases}
\end{equation}
then
$$
(\pi-\operatorname{diam}(\mathrm{X}))^N \leq C \delta^{1/p}.
$$
\end{theorem}
\begin{proof}
Assume that $(\mathrm{X}, \mathrm{d}, \mathfrak{m})$ is an $\operatorname{RCD}(N-1, N)$ space. Let $\Omega \subset \mathrm{X}$ be an open set with $\mathfrak{m}(\Omega)=v \in(0,1)$ and $\lambda_{\mathrm{X}}^p(\Omega)=\lambda$. Let $u$ and $z$ be as in the assumptions of Theorem $\ref{thm3}$, scaled such that $(\ref{1.75})$ holds. We extend $z^\sharp$ to $[0,\infty)$ by taking it to be $0$ on $(\alpha,\infty)$ and extend $u^\sharp$ to $[0,\infty)$ by taking it to be $0$ on $(v,\infty)$. Now by the assumptions of the theorem and Chiti's comparison theorem, there exists $s_1 \in(0, \alpha)$ such that
\begin{equation}\label{1.9}
u^{\sharp}(s) \leq z^{\sharp}(s) \text{ for } s \in[0, s_1] \text{ and } z^{\sharp}(s) \leq u^{\sharp}(s) \text{ for } s \in[s_1, v].
\end{equation}

Now if $s\in [0,s_1]$, we have 
$$
\int_0^s(u^{\sharp}(t))^{p-1} \mathrm{d} t \leq \int_0^s(z^{\sharp}(t))^{p-1} \mathrm{d} t.
$$
If instead $s \in[s_1, v]$, using $(\ref{1.75})$, we have
\begin{align*}
\int_0^s(u^{\sharp}(t))^{p-1} \mathrm{d} t & =\int_0^v(u^{\sharp}(t))^{p-1} \mathrm{d} t-\int_s^v(u^{\sharp}(t))^{p-1} \mathrm{d} t \leq \int_0^v(u^{\sharp}(t))^{p-1} \mathrm{d} t-\int_s^v(z^{\sharp}(t))^{p-1} \mathrm{d} t \\
& =\int_0^v(z^{\sharp}(t))^{p-1} \mathrm{d} t-\int_s^v(z^{\sharp}(t))^{p-1} \mathrm{d} t=\int_0^s(z^{\sharp}(t))^{p-1} \mathrm{d} t
\end{align*}
Hence, we obtain that
\begin{equation}\label{1.10}
\int_0^s(u^{\sharp}(t))^{p-1} \mathrm{d} t \leq \int_0^s(z^{\sharp}(t))^{p-1} \mathrm{d} t, \quad \text { for all } s \in[0, v].
\end{equation}

Recall that $(\ref{41})$ says that 
$$
\mathcal{I}_{N-1, N}(\mu(t))^q\leq \mathrm{Per}(\{u>t\})^q\leq (-\mu^{\prime}(t))\bigg(\int_0^{\mu(t)} \lambda (u^{\sharp})^{p-1}(s) \mathrm{d} s\bigg)^{q/p} ,\quad\text{ for }\mathscr{L}^1\text{-a.e. }t>0.
$$
Dividing the inequality by $\mathcal{I}_{N-1, N}(\mu(t))^q$, we obtain
$$
\frac{(\operatorname{Per}(\{u>t\}))^q}{\mathcal{I}_{N-1, N}(\mu(t))^q} \leq \frac{- \mu^{\prime}(t)}{\mathcal{I}_{N-1, N}(\mu(t))^q}\bigg(\int_0^{\mu(t)} \lambda (u^{\sharp})^{p-1}(s) \mathrm{d} s\bigg)^{q/p},\quad\text{ for }\mathscr{L}^1\text{-a.e. }t\in (0,\sup_\Omega u).
$$
Fix $\delta>0$ and $y\in (\alpha,\alpha+\delta^{1/p})$. Integrating from $u^{\sharp}(y)$ to $u^{\sharp}(0)$ yields
$$
\int_{u^{\sharp}(y)}^{u^{\sharp}(0)} \frac{(\operatorname{Per}(\{u>t\}))^q}{\mathcal{I}_{N-1, N}(\mu(t))^q} \mathrm{d} t \leq \int_{u^{\sharp}(y)}^{u^{\sharp}(0)} \frac{- \mu^{\prime}(t)}{\mathcal{I}_{N-1, N}(\mu(t))^q} \bigg(\int_0^{\mu(t)} \lambda (u^{\sharp})^{p-1}(s) \mathrm{d} s\bigg)^{q/p} \mathrm{d} t .
$$
Performing the change of variable $s=\mu(t)$, noting that $\mu(u^{\sharp}(y)) \leq y \leq \alpha+\delta^{1/p}$ and applying $(\ref{1.10})$ it follows that
\begin{equation}\label{3.26}
\int_{u^{\sharp}(y)}^{u^{\sharp}(0)} \frac{(\operatorname{Per}(\{u>t\}))^q}{\mathcal{I}_{N-1, N}(\mu(t))^q} \mathrm{d} t \leq \int_0^{\alpha+\delta^{1/p}} \frac{1}{\mathcal{I}_{N-1, N}(s)^q} \bigg(\int_0^{s} \lambda (z^{\sharp})^{p-1}(t) \mathrm{d} t\bigg)^{q/p} \mathrm{d}s.
\end{equation}
We next estimate the right hand side of $(\ref{3.26})$ by splitting the integral in the two contributions: on $[0, \alpha]$ and on $[\alpha, \alpha+\delta^{1/p}]$. Noting that $q=\frac{p}{p-1}$, the former contribution is controlled by using $(\ref{3.4})$
\begin{equation}\label{3.27}
\int_0^\alpha \frac{\lambda^{\frac{1}{p-1}}}{\mathcal{I}_{N-1, N}(s)^{\frac{p}{p-1}}}\bigg(\int_0^{s} (z^{\sharp})^{p-1}(t) \mathrm{d} t\bigg)^{\frac{1}{p-1}} \mathrm{d} s=z^{\sharp}(0)-z^{\sharp}(\alpha)=z^{\sharp}(0) .
\end{equation}
Noting that $z^{\sharp}=0$ on $(\alpha, \alpha+\delta^{1/p})$ and that anything that depends on $\alpha$ is in fact fixed by fixing $\lambda$, the latter contribution can be estimate as
\begin{align}
\int_\alpha^{\alpha+\delta^{1/p}} \frac{\lambda^{\frac{1}{p-1}}}{\mathcal{I}_{N-1, N}(s)^{\frac{p}{p-1}}}\bigg(\int_0^{s} (z^{\sharp})^{p-1}(t) \mathrm{d} t\bigg)^{\frac{1}{p-1}} \mathrm{d} s& =\int_\alpha^{\alpha+\delta^{1/p}} \frac{\lambda^{\frac{1}{p-1}}}{\mathcal{I}_{N-1, N}(s)^{\frac{p}{p-1}}}\bigg(\int_0^{\alpha} (z^{\sharp})^{p-1}(t) \mathrm{d} t\bigg)^{\frac{1}{p-1}} \mathrm{d} s\notag\\
& \leq C(N, v,p, \lambda) \delta^{1/p},\label{3.28}
\end{align}
where in the last equality we used that $\mathcal{I}_{N-1, N}$ does not vanish on the compact set $[\alpha, v]$ and hence is bounded below by a constant that only depends on $N, \lambda,v$. Combining $(\ref{3.26})$- $(\ref{3.28})$ gives that (from here on $C(N, v, p,\lambda)$ may change from line to line):
$$
\frac{1}{u^{\sharp}(0)-u^{\sharp}(y)} \int_{u^{\sharp}(y)}^{u^{\sharp}(0)} \frac{(\operatorname{Per}(\{u>t\}))^q}{\mathcal{I}_{N-1, N}(\mu(t))^q} \mathrm{d} t \leq \frac{z^{\sharp}(0)+C(N, v, p,\lambda) \delta^{1/p}}{u^{\sharp}(0)-u^{\sharp}(y)},
$$
so that there exists $t_0 \in(u^{\sharp}(y), u^{\sharp}(0))$ such that
\begin{align}
\frac{(\operatorname{Per}(\{u>t_0\}))^q}{\mathcal{I}_{N-1, N}(\mu(t_0))^q} & \leq \frac{z^{\sharp}(0)+C(N, v, p,\lambda) \delta^{1/p}}{u^{\sharp}(0)-u^{\sharp}(y)} .\label{3.29}
\end{align}

\vspace{0.5cm}

Now, given that $\delta>0$ and an unbounded subset $Q$ of $(p-1, \infty)$ such that $(\ref{1.7})$ holds. We claim that there exists $y \in(\alpha, \alpha+\delta^{1/p})$ such that
\begin{equation}\label{3.24}
u^{\sharp}(y)=\bigg(\frac{1}{\delta^{1/p}} \int_\alpha^{\alpha+\delta^{1/p}} u^{\sharp}(s)^{p-1} \mathrm{d} s \bigg)^{\frac{1}{p-1}}\leq v^{\frac{1}{p-1}}\delta^{1/p}.
\end{equation}
We will prove this claim by considering two cases:

\textbf{Case 1}: $p\geq 2$. $(\ref{1.7})$ says that
$$
\|z\|_{L^t([0, r_\alpha), \mathfrak{m}_{N-1, N})}^{p-1}-\|u^\star\|_{L^t([0,r_v), \mathfrak{m}_{N-1,N})}^{p-1}=\|z\|_{L^t([0, r_\alpha), \mathfrak{m}_{N-1, N})}^{p-1}-\|u\|_{L^t(\Omega, \mathfrak{m})}^{p-1}<\delta,\text{ for all }t\in Q.
$$
Combining this and Theorem $\ref{thm3}$, we deduce that
\begin{equation}\label{1.8}
0\leq (z^{\sharp}(0)-u^{\sharp}(0))^{p-1}\leq z^{\sharp}(0)^{p-1}-u^{\sharp}(0)^{p-1}=\|z\|_{L^{\infty}([0, r_\alpha))}^{p-1}-\|u^\star\|_{L^{\infty}([0, r_v))}^{p-1}\leq \delta,
\end{equation}
where we used $p\geq 2$. Combining $(\ref{1.10})$, $(\ref{1.9})$ and the estimates $(\ref{3.3})$ and $(\ref{3.4})$, we can deduce that
\begin{align*}
\frac{\mathrm{d}}{\mathrm{d} s}(z^{\sharp}(s)^{p-1}-u^{\sharp}(s)^{p-1}) &=(p-1)z^{\sharp}(s)^{p-2}\frac{\mathrm{d}}{\mathrm{d} s}z^{\sharp}(s)-(p-1)u^{\sharp}(s)^{p-2}\frac{\mathrm{d}}{\mathrm{d} s}u^{\sharp}(s)\\
&\leq (p-1)\frac{(\lambda_{N-1, N, \alpha}^p)^{\frac{1}{p-1}}}{\mathcal{I}_{N-1, N}(s)^{\frac{p}{p-1}}}\bigg[u^{\sharp}(s)^{p-2}\bigg(\int_0^{s} (u^{\sharp})^{p-1}(t) \mathrm{d} t\bigg)^{\frac{1}{p-1}}\\
&\quad -z^{\sharp}(s)^{p-2}\bigg(\int_0^{s} (z^{\sharp})^{p-1}(t) \mathrm{d} t\bigg)^{\frac{1}{p-1}}\bigg]\leq 0,
\end{align*}
for $\mathscr{L}^1$-a.e. $s \in[0, s_1]$, where we used $p\geq 2$. Thus $z^{\sharp}(s)^{p-1}-u^{\sharp}(s)^{p-1}$ is nonincreasing on $[0,s_1]$ and, using $(\ref{1.75})$ and $(\ref{1.8})$, we get the estimate that
\begin{align*}
\int_\alpha^{\alpha+\delta^{1/p}} u^{\sharp}(s)^{p-1} \mathrm{d} s&=\int_\alpha^{\alpha+\delta^{1/p}} u^{\sharp}(s)^{p-1}-z^{\sharp}(s)^{p-1} \mathrm{d} s  \leq \int_{s_1}^v u^{\sharp}(s)^{p-1}-z^{\sharp}(s)^{p-1} \mathrm{d} s \\
& =\int_0^{s_1} z^{\sharp}(s)^{p-1}-u^{\sharp}(s)^{p-1} \mathrm{d} s \leq s_1(z^{\sharp}(0)^{p-1}-u^{\sharp}(0)^{p-1}) \leq v \delta .
\end{align*}
This implies the claim immediately.

\textbf{Case 2}: $p\in (1,2)$. $(\ref{1.7})$ says that
$$
(\|z\|_{L^t([0, r_\alpha), \mathfrak{m}_{N-1, N})}-\|u^\star\|_{L^t([0,r_v), \mathfrak{m}_{N-1,N})})^{p-1}=(\|z\|_{L^t([0, r_\alpha), \mathfrak{m}_{N-1, N})}-\|u\|_{L^t(\Omega, \mathfrak{m})})^{p-1}<\delta,\text{for all }t\in Q.
$$
Combining this and Theorem $\ref{thm3}$, we deduce that
\begin{equation}\label{1..8}
0\leq (z^{\sharp}(0)-u^{\sharp}(0))^{p-1}=(\|z\|_{L^{\infty}([0, r_\alpha))}-\|u^\star\|_{L^{\infty}([0, r_v))})^{p-1}\leq \delta .
\end{equation}
Combining $(\ref{1.10})$ and the estimates $(\ref{3.3})$ and $(\ref{3.4})$, we can deduce that
\begin{align*}
\frac{\mathrm{d}}{\mathrm{d} s}(z^{\sharp}(s)-u^{\sharp}(s))
\leq \frac{(\lambda_{N-1, N, \alpha}^p)^{\frac{1}{p-1}}}{\mathcal{I}_{N-1, N}(s)^{\frac{p}{p-1}}}\bigg[\bigg(\int_0^{s} (u^{\sharp})^{p-1}(t) \mathrm{d} t\bigg)^{\frac{1}{p-1}} -\bigg(\int_0^{s} (z^{\sharp})^{p-1}(t) \mathrm{d} t\bigg)^{\frac{1}{p-1}}\bigg]\leq 0,
\end{align*}
for $\mathscr{L}^1$-a.e. $s \in[0, v]$. Thus $z^{\sharp}(s)-u^{\sharp}(s)$ is nonincreasing on $[0,v]$ and, using $(\ref{1.75})$ and $(\ref{1..8})$, we get the estimate that
\begin{align*}
\int_\alpha^{\alpha+\delta^{1/p}} u^{\sharp}(s)^{p-1} \mathrm{d} s&=\int_\alpha^{\alpha+\delta^{1/p}} u^{\sharp}(s)^{p-1}-z^{\sharp}(s)^{p-1} \mathrm{d} s  \leq \int_{s_1}^v u^{\sharp}(s)^{p-1}-z^{\sharp}(s)^{p-1} \mathrm{d} s \\
& =\int_0^{s_1} z^{\sharp}(s)^{p-1}-u^{\sharp}(s)^{p-1} \mathrm{d} s \leq\int_0^{s_1} (z^{\sharp}(s)-u^{\sharp}(s))^{p-1} \mathrm{d} s\\
&\leq  s_1(z^{\sharp}(0)-u^{\sharp}(0))^{p-1}\leq v \delta,
\end{align*}
where we used $p\in (1,2)$. This implies the claim immediately. Hence, we complete the proof of the claim.

\vspace{0.5cm}

Next, pick $\tilde{\delta}>0$ such that
$$
z^{\sharp}(0)-\tilde{\delta}^{\frac{1}{p-1}}>\frac{z^{\sharp}(0)}{2} \text { and } \frac{z^{\sharp}(0)}{2}-v^{\frac{1}{p-1}}\tilde{\delta}^{1/p}>\frac{z^{\sharp}(0)}{4} .
$$
Note that $\tilde{\delta}>0$ depends only on $N, v,p, \lambda$. Then, using $(\ref{3.24})$ and $(\ref{1.8})$, $(\ref{1..8})$, we get 
$$
u^{\sharp}(0)-u^{\sharp}(y)>z^{\sharp}(0)-\delta^{\frac{1}{p-1}}-u^{\sharp}(y) > \frac{z^{\sharp}(0)}{2}-v^{\frac{1}{p-1}}\delta^{1/p} \geq \frac{z^{\sharp}(0)}{4}, \quad \text { for any } \delta \in(0, \tilde{\delta}) .
$$
Plugging $(\ref{3.24})$ and the last estimate into $(\ref{3.29})$ yields 
$$
\frac{(\operatorname{Per}(\{u>t_0\}))^q}{\mathcal{I}_{N-1, N}(\mu(t_0))^q} \leq 1+C(N, v,p, \lambda)(\delta^{\frac{1}{p-1}}+\delta^{\frac{1}{p}}), \quad \text { for any } \delta \in(0, \tilde{\delta}) .
$$
Picking $\tilde{\delta}>0$ smaller if necessary, but again only depending on $N, v,p, \lambda$, we infer that
\begin{equation}\label{3.30}
\frac{(\operatorname{Per}(\{u>t_0\}))^q}{\mathcal{I}_{N-1, N}(\mu(t_0))^q} \leq 1+C(N, v,p, \lambda) \delta^{\frac{1}{p}}, \quad \text { for any } \delta \in(0, \tilde{\delta}) .
\end{equation}

Combining $(\ref{3.30})$ with the B\'{e}rard-Besson-Gallot-type quantitative improvement of the L\'{e}vy-Gromov isoperimetric inequality obtained for the present non-smooth framework in \cite[Lemma 3.1, 3.2]{CMS19}, gives that there exists some constant $C=C(N,v,p,\lambda)>0$ and $\tilde{\delta}=\tilde{\delta}(N, v,p, \lambda)>0$ such that
$$
(\pi-\operatorname{diam}(\mathrm{X}))^N \leq C \delta^{\frac{1}{p}}, \quad \text { for any } \delta \in(0, \tilde{\delta}) .
$$
We complete the proof.
\end{proof}

\begin{ack}
The author would like to express her gratitude to
Professor Jiayu Li for his encouragement and Liming Yin for constant encouragement and discussions.
	
\end{ack}
\end{document}